\documentclass[en,oneside,bbfont,article]{amsart} 
\usepackage[lmargin = 30mm, rmargin = 30mm, bmargin = 25mm, tmargin=25mm]{geometry}
\usepackage[utf8]{inputenc}

\providecommand{\examplename}{Example}

\newtheorem{theorem}{Theorem}[section]
\newtheorem{proposition}[theorem]{Proposition}
\newtheorem{corollary}[theorem]{Corollary}
\newtheorem{lemma}[theorem]{Lemma}
\theoremstyle{remark}
\newtheorem{remarkx}[theorem]{Remark}
\newenvironment{remark}
    {\pushQED{\qed}\remarkx}
    {\popQED\endremarkx}
\theoremstyle{definition}
\newtheorem*{example*}{\protect\examplename}
\newtheorem{examplex}[theorem]{\protect\examplename}
\newenvironment{example}
    {\pushQED{\qed}\examplex}
    {\popQED\endexamplex}
\theoremstyle{definition}

\newtheorem*{assumption*}{Assumption}
\newtheorem{defin}[theorem]{Definition}

\DeclareUnicodeCharacter{0306}{\u{}}
\DeclareUnicodeCharacter{03B1}{\ensuremath{\alpha}}

\usepackage{mathtools}
\usepackage{mathrsfs}
\usepackage{multirow}
\usepackage{amssymb}
\usepackage{makecell}
\usepackage{amstext}
\usepackage{float}
\usepackage{dsfont}
\usepackage{amsthm}
\usepackage{amsmath}
\usepackage{bm}
\usepackage{comment}
\usepackage{caption,subcaption,enumerate,enumitem}
\usepackage{scalefnt}
\usepackage{float,graphicx,verbatim}
\usepackage{pgfplots}
\pgfplotsset{compat=1.16}
\RequirePackage[colorlinks=true,linkcolor=black,
	citecolor=blue,urlcolor=blue]{hyperref}

\usepackage[backend=biber,style=numeric,sortcites=true,giveninits=true,citestyle=numeric,maxbibnames=99]{biblatex}
\appto{\bibsetup}{\sloppy}
\DeclareNameAlias{author}{family-given}

\newcommand\ve{\varepsilon}
\newcommand\N{\mathbb{N}}
\newcommand\Z{\mathbb{Z}}

\newcommand\calQ{\mathcal{Q}}
\newcommand\X{\bm{\mathcal{X}}}
\newcommand\Y{\bm{\mathcal{Y}}}
\newcommand{\V}{\bm{\mathcal{V}}}
\newcommand\ZZ{\bm{\mathcal{Z}}}

\newcommand\Sp{\mathbb{S}}
\newcommand\scrL{\mathscr{L}}

\newcommand{\mW}{\mathcal{W}}
\newcommand\R{\mathbb{R}}

\newcommand\E{\mathds{E}}
\newcommand\e{\mathord{\mathrm{e}}}
\newcommand\p{\mathds{P}}
\newcommand\1{\mathds{1}}
\newcommand\Oh{\mathcal{O}}
\newcommand\oh{\mathrm{o}}
\newcommand{\Borel}{\mathcal{B}}

\newcommand\SV{\mathrm{SV}}
\newcommand\CSV{\mathrm{CSV}}
\newcommand\GSV{\mathrm{GSV}}

\newcommand{\Leb}{\text{\normalfont Leb}}

\newcommand\da{\downarrow}
\newcommand\ua{\uparrow}
\newcommand\la{\leftarrow}
\renewcommand\le{\leqslant}
\renewcommand\ge{\geqslant}

\newcommand\eqd{\overset{d}{=}}
\newcommand\cid{\xrightarrow{d}}
\newcommand\cip{\xrightarrow{\p}}

\newcommand\co{\mathsf{c}}
\newcommand\nf[1]{\normalfont{#1}}

\newcommand{\D}{\mathrm{d}}
\newcommand{\ov}[1]{\overline{#1}}

\newcommand{\wt}[1]{\widetilde{#1}}
\newcommand{\wh}[1]{\widehat{#1}}
\newcommand{\mft}{{\mathfrak{t}}}
\newcommand{\mfc}{{\mathfrak{c}}}

\newcommand{\Lloc}{L_{\mathrm{loc}}^1}

\newcommand{\tra}{{\scalebox{0.6}{$\top$}}}

\newcolumntype{C}{>{\centering\arraybackslash}m{1.7cm}}
\newcolumntype{T}{>{\centering\arraybackslash}m{4.7cm}}
\newcolumntype{L}{>{\centering\arraybackslash}m{7cm}}

\newcommand{\ceil}[1]{\lceil#1\rceil}

\newcommand\jorge[1]{{\color{red}#1}}

\usepackage[foot]{amsaddr}

\title[Rates of convergence for multivariate SDEs driven by L\'evy processes in small-time]{Rates of convergence for multivariate SDEs driven by L\'evy processes in the small-time stable domain of attraction}
\author{Jorge Gonz\'alez C\'azares$^{*}$ \& David Kramer-Bang$^{\dag}$}

\address{$^{*}$ IIMAS--UNAM, Mexico \and $^{\dag}$Department of Mathematics, Aarhus University, DK}

\email{jorge.gonzalez@sigma.iimas.unam.mx}
\email{bang@math.au.dk}

\begin{document}

\begin{abstract}
This paper explores the rates of convergence of solutions for multivariate stochastic differential equations (SDEs) driven by L\'evy processes within the small-time stable domain of attraction (DoA). Explicit bounds are derived for the uniform Wasserstein distance between solutions of two Lévy-driven SDEs, expressed in terms of driver characteristics. These bounds establish convergence rates in probability for drivers in the DoA, and yield uniform Wasserstein distance convergence for SDEs with additive noise. The methodology uses two couplings for Lévy driver jump components, leading to sharp convergence rates tied to the processes' intrinsic properties.
\end{abstract}



\maketitle

\section{Introduction}

Lévy-driven SDEs are essential for modelling systems with jumps or discontinuities, features common in various real-world phenomena. In finance, asset prices exhibit sudden shifts due to market shocks~\cite{tankov2015financial}; in physics, particle trajectories may experience abrupt changes, leading to anomalous diffusion~\cite{MR1809268}, and, moreover, thermalisation and cut-off phenomena in systems modelling energy transport and turbulence~\cite{MR4307706,MR4791608}; in ecology, population dynamics can undergo rapid alterations from environmental events~\cite{LevinEco}; and in climate science, fast non-Gaussian transition behaviours are often attributed to jumps~\cite{climate}. Their ability to capture jump behaviour makes Lévy-driven SDEs crucial for accurate modelling and prediction in these fields.

The small-time behaviour of such models is particularly significant, as it unveils intricate properties of the underlying processes that are obscured over larger time scales. Many L\'evy processes, especially those used in practice, lie in the domain of attraction (DoA) of stable processes, converging to a stable process under appropriate scaling as time approaches zero. This convergence enables approximating complex systems with simpler, well-studied stable processes, improving both theoretical analysis and practical computation.

In this paper, we study the rates of convergence for multivariate SDEs driven by L\'evy processes within the small-time stable DoA. Explicit bounds are established on the Wasserstein distance between solutions of Lévy-driven SDEs. In the DoA case, tightness is demonstrated in the appropriately scaled supremum norm between SDEs driven by the limiting stable process and those within the DoA, using the thinning and comonotonic couplings. These methods effectively address the jump components, providing sharp convergence rates related to driver characteristics.

Our contributions include:
\begin{itemize}
\item Broad bounds are derived for the uniform $L^p$-Wasserstein distance between SDE solutions under both couplings, applicable to a wide class of Lévy drivers.
\item For drivers in the domains of normal attraction (DoNA) or non-normal attraction (DoNNA) of stable processes and satisfying further regularity conditions, convergence in probability guarantees are established for the SDE solutions, exploiting the structure of the L\'evy measures and the SDE.
\item For additive SDEs, we derive bounds in $L^p$-Wasserstein distances, yielding a convergence rate that matches that of the rate of DoA between the drivers.
\end{itemize}

This work builds on a rich literature on L\'evy-driven SDEs, including studies on SDE stability~\cite{MR2001996,MR2160585}, convergence rates~\cite{MR3332269} and approximation methods~\cite{pavlyukevich2025stronguniformwongzakaiapproximations,kosenkova2020orderconvergenceweakwongzakai,MR2802466,MR2759203}. Recent efforts have quantified distances between solutions driven by different L\'evy processes, notably in stable approximations of SDEs with additive noise used in climate science, see~\cite{MR3332269} and the references therein. Our research advances this field by providing explicit, computable rates in the small-time regime, a context critical for applications requiring high-frequency or short-term predictions.

The paper is structured as follows: Section~\ref{sec:framework} outlines preliminaries on L\'evy processes, stable domains, and the SDE framework. Section~\ref{sec:main} presents our main convergence rate results. Illustrations, an outline of the proofs, a literature review and a general discussion non-viable approaches are provided in Subsections~\ref{sec:illustrations},~\ref{subsec:outline},~\ref{subsec:literature} and~\ref{sec:failed_ideas}, respectively. Section~\ref{sec:special} contains useful corollaries and stronger results obtained under additional structural assumptions. In particular, we study the Wasserstein distance for SDEs with additive noise in Section~\ref{subsec:additive}, and general SDEs driven by augmented stable processes in Section~\ref{subsec:augment_stable}. Section~\ref{sec:general_bounds} develops the two coupling techniques employed in this paper and provides general bounds on Wasserstein distances between solutions of SDEs driven by arbitrary L\'evy processes. In Section~\ref{sec:application_DoA}, we present proofs of the main results. Finally, Section~\ref{sec:conclusion} offers concluding remarks and comments on open problems.

\section{Preliminaries and Notation}
\label{sec:framework}

\subsection{Notation} 
\label{subsec:notation}

The following standard notation and conventions are used throughout.

\begin{itemize}[leftmargin=1em]
\item Let $\R_+\coloneqq[0,\infty)$ denote the non-negative real numbers, $\R^d_{\bm0}\coloneqq\R^d\setminus\{\bm{0}\}$ and $\lceil x\rceil\coloneqq\inf\{k\in\Z:k\ge x\}$ denote the smallest integer that is larger or equal to $x$.
\item For functions $f\ge 0$ and $g>0$, we say that $f(t) \lesssim g(t)$ or $f(t)=\Oh(g(t))$ (resp. $f(t)=\oh(g(t))$; $f(t)\sim g(t)$) as $t\da 0$ if $\limsup_{t\da 0}f(t)/g(t)<\infty$ (resp. $\limsup_{t\da 0}f(t)/g(t)=0$; $\lim_{t\da 0}f(t)/g(t)=1$). 
\item Given $\bm{f}:[a,b]\mapsto\R^d$, we denote by $\|\bm{f}\|_{[a,b]}\coloneqq\sup_{t\in[a,b]}|\bm{f}(t)|$ its supremum norm on $[a,b]$, where $|\cdot|$ denotes the Euclidean norm in $\R^d$ and $\R^{d\times m}$ (also known as the Frobenius norm), i.e. $|\bm{x}|^2=\sum_{i=1}^d x_i^2$ for $\bm{x} \in \R^d$ and $|\bm{A}|^2=\sum_{i=1}^d \sum_{j=1}^m a_{i,j}^2$ for $\bm{A} \in \R^{d\times m}$. 
\item A function $\ell$ is \emph{slowly varying} at $a\in[0,\infty]$, denoted $\ell \in \SV_a$, if $\ell$ is measurable, positive in a neighbourhood of $a$ and $\lim_{x \to a} \ell(cx)/\ell(x)=1$ for all $c\in(0,\infty)$ (typically, $a\in\{0,\infty\}$). Moreover, we say $\ell$ is \emph{genuinely slowly varying} at $a$, denoted $\ell\in\GSV_a$, if $\ell\in\SV_a$ has no positive finite limit.
\item A function $\bm{f}:(0,\infty)\to\R^d$ is \emph{locally integrable} on $(0,\infty)$, denoted $\bm{f}\in \Lloc$, if $\bm{f}$ is measurable and $\int_a^b |\bm{f}(x)|\D x<\infty$ for all $0<a<b<\infty$.
\item A family of random elements $\{\bm{\xi}(t)\}_{t\in\mathcal{T}}$ on a normed space with norm $\|\cdot\|$, is said to be tight, if $\lim_{r \to \infty} \sup_{t \in \mathcal{T}}\p(\|\bm{\xi}(t)\|>r)=0$. 
\end{itemize}

Finally, define the class of functions with \emph{controlled slow variation}: $\ell\in\GSV_0(\phi_1,\phi_2)$ if $\ell\in\SV_0$, $\ell(t)=1$ for $t\ge 1$ and the controlling functions satisfy $\phi_1\in\SV_0\cap\SV_\infty\cap \Lloc$, $\phi_2:(0,\infty)\to[0,1]$, $\phi_2(t)\to 0$ as $t \da 0$, and
\[
\left|\ell(xt)/\ell(t)-1 \right|
\le \phi_1(x)\phi_2(t), 
\quad \text{ for all }x>0
\text{ and all }t\in (0,1].
\] 

\begin{remark}
\label{rem:CSV}
(i) If $\ell\in\CSV_0(\phi_1,\phi_2)$, then 
\[
\sup_{t\in(0,1]}|\ell(xt)/\ell(t)|
\le 1+\phi_1(x),
\quad \text{for all }x>0.
\]
(ii) Suppose $\ell\in\CSV_0(\phi_1,\phi_2)$ does not have a positive finite limit at $0$. Then~\cite[Lem.~7.2(a)]{WassersteinPaper} implies that $\int_0^1 \sup_{s\in[0,t]}\phi_2(s)\, t^{-1}\D t=\infty$. Thus, if $\phi_2$ is monotone (non-decreasing), then $\phi_2(t)$ is ``slowly'' converging to $0$ as $t\da 0$. In fact, by~\cite[Lem.~7.1]{WassersteinPaper}, if $\ell$ is differentiable and $t|\ell'(t)|$ is eventually positive and slowly varying at $0$, then $t|\ell'(t)|/\ell(t)\lesssim \phi_2(t)$ (and often $t|\ell'(t)|/\ell(t)\sim \phi_2(t)$, see~\cite[Lem.~6.1--6.3]{WassersteinPaper}) as $t\da 0$, which is a more concrete description of the ``slowness'' of $\phi_2$.\\
(iii) The tightness of some family $\{\bm{\xi}(t)\}_{t\in(0,1]}$ implies the equivalent limits 
$\E[\min\{\phi(t)\|\bm{\xi}(t)\|,1\}]\to 0$ and $\phi(t)\bm{\xi}(t)\cip \bm{0}$ as $t\da 0$ whenever $\phi(t)\to 0$ (see, e.g.~\cite[Lem.~5.2~\&~5.9]{MR4226142}). By virtue of this relationship, our results below describe the rate of uniform convergence in probability.
\end{remark}

\subsection{Preliminaries on L\'evy Processes}\label{sec:prelimi_SDEs_stable}
Fix an arbitrary $T>0$ and let $\bm{X}=(\bm{X}(t))_{t \in [0,T]}$ be a L\'evy process on $\R^d$ (see~\cite[Def.~1.6, Ch.~1]{MR3185174}) and let $\bm{Z}=(\bm{Z}(t))_{t\in [0,T]}$ be an $\alpha$-stable process on $\R^d$ with $\alpha\in (0,2)\setminus\{1\}$ (see~\cite[Sec.~1.2]{MR3808900}). The process $\bm{X}$ is in the small-time DoA of $\bm{Z}$ if $\bm{X}(t)/g(t)\cid \bm{Z}(1)$ as $t \da 0$ for a normalising function $g(t)=t^{1/\alpha}G(t)$ and a function $G\in \SV_0$ (see~\cite[Eq.~(8)]{MR3784492}). The process $\bm{X}$ is in the domain of normal attraction (DoNA) when $G$ has a positive finite limit at $0$ (see~\cite[p.~181]{MR0233400}); otherwise, $\bm{X}$ is in the domain of non-normal attraction (DoNNA). See Theorem~\ref{thm:small_time_domain_stable} below for a precise description of this class. Throughout, for every $t\in (0,1]$, $\bm{X}_t=(\bm{X}_t(s))_{s\in[0,T]}$ denotes a L\'evy process with the same law as $(\bm{X}(st)/g(t))_{s\in[0,T]}$. 

We denote the generating triplet (see~\cite[Def.~8.2]{MR3185174} for a definition) of the $\alpha$-stable L\'evy process $\bm{Z}$ by $(\bm{\gamma_Z},\bm{\Sigma_Z}\bm{\Sigma_Z}^\tra,\nu_{\bm{Z}})$ (corresponding to the cutoff function $\bm{w}\mapsto\1_{D}(\bm{w})$, where $D\coloneqq B_{\bm{0}}(1)$ and $B_{\bm{0}}(r)\coloneqq \{x \in \R^d: |\bm{x}|<r\}$ is the open ball of radius $r>0$). Thus, $\bm{Z}$ has no Gaussian component (i.e. $\bm{\Sigma}_{\bm{Z}}=\bm{0}$) and its L\'evy measure is given by
\begin{equation}\label{eq:mu_measure_defn}
\nu_{\bm{Z}}(A)
\coloneqq c_\alpha\int_{\Sp^{d-1}}\int_0^\infty\1_{A}(x\bm{v})\frac{\D x}{x^{\alpha+1}}\sigma(\D\bm{v}), 
\qquad \text{for }A\in\mathcal{B}(\R^d_{\bm{0}}),
\end{equation}
where $\sigma$ is a probability measure on the unit sphere $\Sp^{d-1}\subset\R^d$. Additionally, if $\alpha\in(0,1)$, the process has zero natural drift (equivalently, 
$\bm{\gamma}_{\bm{Z}}=\int_D\bm{w}\,\nu_{\bm{Z}}(\D\bm{w})$), and if $\alpha\in(1,2)$, the process has zero mean (equivalently, 
$\bm{\gamma}_{\bm{Z}}=-\int_{D^\co}\bm{w}\,\nu_{\bm{Z}}(\D\bm{w})$). We comment that the cases $\alpha\in\{1,2\}$ were excluded in the present work to keep this paper on the shorter end (see Section~\ref{sec:conclusion} below for more details).

\section{The Main Results}
\label{sec:main}

In this section, we study the rate of convergence for multivariate SDEs driven by L\'evy processes in the small-time stable DoA. We give explicit convergence rates for convergence in probability between solutions of SDEs driven by such L\'evy processes and their corresponding limiting stable processes. This is achieved by employing the \emph{thinning} (Subsection~\ref{subsec:thinning}) and \emph{comonotonic} (Subsection~\ref{subsec:como}) couplings. We start with two subsections, dedicated to each of the coupling methods and the resulting convergence rates under appropriate assumptions. 

Consider a Lipschitz continuous function $V:\R^m \to \R^{m\times d}$ with Lipschitz constant $K \in (0,\infty)$, such that $|V(\bm{y})-V(\bm{z})|\le K|\bm{y}-\bm{z}|$ for any $\bm{y},\bm{z}\in\R^{m}$, where $|\cdot|$ is the appropriate Euclidean norm. We consider the solutions $\ZZ=(\ZZ(s))_{s\in[0,T]}$ and $\X_t=(\X_t(s))_{s\in[0,T]}$ (for $t\in(0,1]$) to the following SDEs (unique solutions exist, by~\cite[Thm~6.2.3]{Applebaum_2009}) with initial value $\bm{x}\in\R^m$: 
\begin{equation}
\label{eq:main_SDE_setting}
\X_t(s)\coloneqq \bm{x}+\int_0^s V(\X_t(r-))\D \bm{X}_t(r), 
\quad 
\ZZ(s)\coloneqq \bm{x}+\int_0^s V(\ZZ(r-))\D \bm{Z}(r),
\quad\text{for all }s \in [0,T].
\end{equation} 

\subsection{The Thinning Coupling}
\label{subsec:thinning} 
In the present subsection, the thinning coupling, introduced in~\cite[Sec.~3.1]{WassersteinPaper}, and detailed in Subsection~\ref{sec_thinning_coup}, is examined. To state the results, we require only the technical Assumption~(\nameref{asm:T}), presented below. This assumption is a generalisation of~\cite[Asm~(T)]{WassersteinPaper}, extending its applicability beyond DoNA to include DoNNA.
\begin{assumption*}[T]
\label{asm:T}
Let $\bm{Z}$ be an $\alpha$-stable process with $\alpha\in(0,2)$ and $\bm{X}$ be a L\'evy process in the small-time DoA of $\bm{Z}$ with generating triplet $(\bm{\gamma}_{\bm{X}},\bm{0},\nu_{\bm{X}})$ (in the notation of~\cite[Def.~8.2]{MR3185174}), where the L\'evy measure admits a decomposition $\nu_{\bm{X}}=\nu_{\bm{X}^\co}+\nu_{\bm{X}^\D}$ satisfying: $\nu_{\bm{X}^\D}$ is arbitrary with finite mass $m_0 \coloneqq \nu_{\bm{X}^\D}(\R^d_{\bm{0}})<\infty$ and
\begin{equation}
\label{eq:TNN}
\nu_{\bm{X}^\co}(\D \bm{w})=H(|\bm w|)h(\bm{w})\nu_{\bm{Z}}(\D \bm{w})
\quad\&\quad
|1-h(\bm{w})|\le K_h(1\wedge |\bm{w}|^p),\quad\text{for all}\quad\bm{w} \in \R^d_{\bm{0}},
\end{equation}
constants $K_h\in[0,\infty)$, $p\in(0,\infty)$, measurable functions $h,H_1,H_2\ge 0$ and $H\in\CSV_0(H_1,H_2)$ such that $t\mapsto tH(t)^{-1/\alpha}$ is strictly increasing. 
Moreover, when $\alpha\in(0,1)$, we assume that $\bm{X}$ has zero natural drift, that is, 
$\bm{\gamma}_{\bm{X}}=\int_D\bm{w}\,\nu_{\bm{X}}(\D\bm{w})$.
\end{assumption*}

As discussed in~\cite[Rem.~2.1]{WassersteinPaper}, Assumption~(\nameref{asm:T}) quantifies the regularity of the corresponding density near the origin $\bm{0}$ through a parameter $p > 0$, where larger values of $p$ correspond to higher degrees of asymptotic regularity. This assumption is widely satisfied in practice; for instance, it holds for the class of augmented stable processes introduced in Section~\ref{subsec:augment_stable}, which includes the popular exponentially tempered stable processes~\cite{MR2327834}.

Under Assumption~(\nameref{asm:T}), we may define a function $G$ such that $t\mapsto tG(t^\alpha)$ is the inverse of the map $t\mapsto tH(t)^{-1/\alpha}$ on $(0,1]$. Then $G(t)=1$ for $t\ge 1$, $G$ is slowly varying at $0$ by virtue of~\cite[Thm~1.5.12]{MR1015093} and, denoting $g(t)\coloneqq t^{1/\alpha}G(t)$, we observe that
\begin{equation}
\label{eq:G-H}
t=tG(t^\alpha)H(tG(t^\alpha))^{-1/\alpha}
\implies
G(t^\alpha)=H(tG(t^\alpha))^{1/\alpha}
\implies 
G(t)=H(g(t))^{1/\alpha},
\quad t>0.
\end{equation}
Hence, $G$ and $H$ either both tend to $0$, $\infty$, a positive finite limit, or do not admit a limit. Moreover, under Assumption~(\nameref{asm:T}), the L\'evy measure $\nu_{\bm{X}^\co}$ has as many moments as $\bm{Z}$ (as $h$ is bounded and $H(x)=1$ for $x\ge 1$).

We now introduce a further balancing condition on the measures $\sigma$ and $\nu_{\bm{X}}$ that, for processes in DoNA with $\alpha>1$, can significantly simplify and improve the convergence rate. In the notation of Assumption~(\nameref{asm:T}),
\begin{equation}
\label{eq:symmetry_thin}
\bm{0}
=\bm{\gamma}_{\bm{X}}
+\int_{\R^d\setminus D}\!\!\!\bm{w}\,\nu_{\bm{X}^\co}(\D\bm{w})
-\int_{D}\bm{w}\,\nu_{\bm{X}^\D}(\D\bm{w})
=\int_{\Sp^{d-1}}\!\!\!\bm{v}\,\sigma(\D\bm{v})
=\int_{\Sp^{d-1}}\!\!\! h(x\bm{v})\bm{v}\,\sigma(\D\bm{v}), 
\enskip\text{for all }x>0.
\end{equation}
Note that the decomposition $\nu_{\bm{X}}=\nu_{\bm{X}^\co}+\nu_{\bm{X}^\D}$ in Assumption~(\nameref{asm:T}) need not be unique, even under~\eqref{eq:symmetry_thin}.

\begin{theorem}
\label{thm:main_res_T}
Let $\bm{X}$ be a L\'evy process in the small-time DoA of an $
\alpha$-stable process $\bm{Z}$ for some $\alpha \in (0,2)\setminus\{1\}$, and let Assumption~(\nameref{asm:T}) hold. Let $\X_t$ and $\ZZ$ be the solutions to the SDEs given in~\eqref{eq:main_SDE_setting} on $[0,T]$, driven by $\bm{X}_t$ and $\bm{Z}$, respectively. If $\bm{X}_t\eqd(\bm{X}(st)/g(t))_{s\in[0,T]}$ and $\bm{Z}$ follow the thinning coupling, then the family $\{f(t)^{-1}(\X_t-\ZZ)\}_{t\in(0,1]}$ is tight with respect to the norm $\|\cdot\|_{[0,T]}$, where:
\begin{itemize}[leftmargin=4.5em]
\item[\nf{\textbf{DoNA}}]
If $H\equiv 1$ (hence $G\equiv 1\equiv H_1$, $H_2\equiv 0$, $g(t)=t^{1/\alpha}$), then 
\[
f(t)\coloneqq
\begin{cases}
t^{(p/(2\alpha))\wedge (1-1/\alpha) }, &\alpha\in(1,2)\text{ and \eqref{eq:symmetry_thin} fails}, \\
t^{p/(\lceil\alpha\rceil\alpha)}, &\alpha\in(1,2)\text{ and \eqref{eq:symmetry_thin} holds or $\alpha\in(0,1)$}.
\end{cases}
\]
\item[\nf{\textbf{DoNNA}}]
If $H\in\CSV_0(H_1,H_2)$, $H_2\in\SV_0$ and $G$ is as in~\eqref{eq:G-H}, then
\[
f(t)\coloneqq
H_2(g(t))^{1/\lceil\alpha\rceil}.
\]
\end{itemize}
\end{theorem}

\begin{remark}
\label{rem:thinning-result}
(i) The tightness of the family $\{f(t)^{-1}(\X_t-\ZZ)\}_{t\in(0,1]}$ with respect to the norm $\|\cdot\|_{[0,T]}$ is equivalent to that of $\{f(t)^{-1}\|\X_t-\ZZ\|_{[0,T]}\}_{t\in(0,1]}$.\\
(ii) Under some assumptions, one can simplify the expression $H_2(g(t))=H_2(t^{1/\alpha}G(t))$, in the sense that $H_2(t^{1/\alpha} G(t))\sim H_2(t^{1/\alpha})$ as $t\da 0$. Indeed, by~\cite[Thm~2.3.1]{MR1015093}, if $H_2$ is slowly varying and for some $\gamma,\lambda>0$, $t\mapsto t^{\gamma/\alpha}G(t^\alpha)$ is eventually non-increasing as $t \da 0$ and 
\begin{equation*}
\bigg(\frac{H_2(\lambda t^{1/\alpha})}{H_2(t^{1/\alpha})}-1\bigg)\log (G(t)) \to 0, \quad \text{ as }t \da 0,
\end{equation*} 
then $H_2(t^{1/\alpha} G(t)^\delta)/ H_2(t^{1/\alpha})\to 1$ as $t \da 0$ uniformly in $\delta \in [0,\Delta]$ for $0<\Delta<1/\gamma $.\\
(iii) Although $H$ is always asymptotically equivalent to several $C^\infty$ slowly varying functions (some of which will also be of controlled slow variation), the derivatives of such smooth equivalents need not be asymptotically equivalent. Hence, by Remark~\ref{rem:CSV}, the corresponding controlling functions $H_2$ (as well as the implied convergence rate in Theorem~\ref{thm:main_res_T}) may be affected by the choice of smooth version. However, finding the ``best'' smooth equivalent appears to be difficult (see, e.g.~\cite[Rem.~2.5(IV)]{WassersteinPaper}).
\end{remark}

\subsection{The Comonotonic Coupling}
\label{subsec:como}
In the present subsection, the comonotonic coupling introduced in~\cite[Sec.~3.2]{WassersteinPaper} and detailed in Subsection~\ref{sec:comonot_coup}, is examined. To state the results, we rely on the technical Assumption~(\nameref{asm:C}), presented below. This can be viewed as a unified formulation of Assumptions~(C) and~(S) from~\cite[Asms~(C) \&~(S)]{WassersteinPaper}, and is applicable under both DoNA and DoNNA.

For the $\alpha$-stable process $\bm{Z}$ defined in Subsection~\ref{sec:prelimi_SDEs_stable} above, we define the radial L\'evy measure $\rho_{\bm{Z}}(\D x,\bm{v})\coloneqq \rho_{\bm{Z}}(\D x)\coloneqq c_\alpha x^{-\alpha-1}\D x$ on $\mathcal{B}((0,\infty))$ and note that the right-continuous inverse of its tail $x \mapsto \rho_{\bm{Z}}([x,\infty),\bm{v})$ is given by $\rho_{\bm{Z}}^{\la}(x,\bm{v})=\rho_{\bm{Z}}^{\la}(x)=(c_\alpha/\alpha)^{1/\alpha}x^{-1/\alpha}$ for all $x>0$ and $\bm{v}\in \Sp^{d-1}$.
\begin{assumption*}[C]
\label{asm:C} 
Let $\bm{Z}$ be an $\alpha$-stable process with $\alpha\in(0,2)$ and $\bm{X}$ be a L\'evy process in the small-time DoA of $\bm{Z}$ with generating triplet $(\bm{\gamma}_{\bm{X}},\bm{0},\nu_{\bm{X}})$, i.e. assume $\bm{X}$ has no Gaussian component $\bm{\Sigma_X}=\bm{0}$. Moreover, assume $\nu_{\bm{X}}=\nu_{\bm{X}^{\co}}+\nu_{\bm{X}^{\mathrm{d}}}$, where $\nu_{\bm{X}^{\mathrm{d}}}$ is arbitrary with finite mass  $\nu_{\bm{X}^{\mathrm{d}}}(\R^d_{\bm{0}})<\infty$ and  the L\'evy measure $\nu_{\bm{X}^{\co}}$ 
can be expressed as 
\begin{equation}
\label{eq:radial_tail_decomp_C}
    \nu_{\bm{X}^{\co}}(B)
=\int_{\Sp^{d-1}}\int_0^\infty \1_{B}(x\bm{v})\rho_{\bm{X}^{\co}}(\D x,\bm{v})\sigma(\D\bm{v}) \quad\&\quad
\rho_{\bm{X}^{\co}}([x,\infty),\bm{v})
= h(x \bm{v}) H(x) x^{-\alpha},
\end{equation}
for all $B\in\mathcal{B}(\R^d_{\bm{0}})$,
$x>0$, $\bm{v}\in \Sp^{d-1}$ and some continuous $h:(0,\infty)\times \Sp^{d-1}\to [0,\infty)$ and monotonic (non-increasing or non-decreasing) function $H\in\SV_0$ with $H(x)=1$ for $x\ge 1$. Define the function
\begin{equation}
\label{eq:defn_G_C}
G(x)\coloneqq\int_{\Sp^{d-1}}H(\rho_{\bm{X}^\co}^{\la}(1/x,\bm{v}))^{1/\alpha}\sigma(\D \bm{v}),\qquad x>0,
\end{equation}
where $\rho_{\bm{X}^\co}^{\la}(\cdot,\bm{v})$ is the right-continuous inverse of $x\mapsto \rho_{\bm{X}}^\co([x,\infty),\bm{v})$. Suppose $G\in\CSV_0(G_1,G_2)$ and
\begin{equation}\label{eq:old_assump_(H)_C}
|c_\alpha/\alpha-h(x\bm{v})|
    \le K_h(1\wedge x^{p}) 
\quad \& \quad 
\big|H(\rho_{\bm{X}^\co}^{\la}(1/x,\bm{v}))^{1/\alpha}/G(x)-1\big|
    \le K_Q(1\wedge x^{\delta}),
\end{equation} for all $x>0, \bm{v} \in \Sp^{d-1}$ and some constants $p,\delta>0$ and $K_h,K_Q\ge 0$. 
\end{assumption*}

\begin{remark}\label{rem:DoNA_como_delta_large}
For a process in DoNA and under Assumption~\ref{asm:C}, we have $H\equiv 1 \equiv G$. In particular, the second inequality in~\eqref{eq:old_assump_(H)_C} is satisfied for all $\delta>0$, i.e. 
\begin{equation*}
    \big|H(\rho_{\bm{X}^\co}^{\la}(1/x,\bm{v}))^{1/\alpha}/G(x)-1\big|\equiv 0
\le K_Q(1\wedge x^{\delta}), \quad \text{ for all }x,\delta >0, \, \bm{v} \in \Sp^{d-1}, \text{ and }K_Q\ge 0.
\end{equation*} 
Thus, for processes in DoNA, we always pick $\delta>0$ sufficiently large not to have any influence on the convergence rates. 
\end{remark}

Under Assumption~(\nameref{asm:C}) with $\alpha>1$, the balancing condition, analogous to~\eqref{eq:symmetry_thin} above, is the following:
\begin{equation}
\label{eq:symmetry_como}
\begin{split}
\bm{0}
&=\bm{\gamma}_{\bm{X}}+\int_{\R^d\setminus D}\!\!\bm{w}\,\nu_{\bm{X}^\co}(\D\bm{w})
-\int_{D}\bm{w}\,\nu_{\bm{X}^\D}(\D\bm{w})\\
&=\int_{\Sp^{d-1}}\!\!\bm{v}\,\sigma(\D\bm{v})
=\int_{\Sp^{d-1}}\!\!\rho^{\la}_{\bm{X}_t^\co}(x,\bm{v})\bm{v}\,\sigma(\D\bm{v}), 
\quad\text{for all } x>0.
\end{split}
\end{equation}

\begin{theorem}
\label{thm:main_res_C}
Let $\bm{X}$ be a L\'evy process in the small-time DoA of an $
\alpha$-stable process $\bm{Z}$ for some $\alpha \in (0,2)\setminus\{1\}$, and let Assumption~(\nameref{asm:C}) hold. Let $\X_t$ and $\ZZ$ be the solutions to the SDEs given in~\eqref{eq:main_SDE_setting} on $[0,T]$, driven by $\bm{X}_t$ and $\bm{Z}$, respectively. If $\bm{X}_t\eqd(\bm{X}(st)/g(t))_{s\in[0,T]}$ and $\bm{Z}$ follow the comonotonic coupling, then the family $\{f(t)^{-1}(\X_t-\ZZ)\}_{t\in(0,1]}$ is tight with respect to the norm $\|\cdot\|_{[0,T]}$, where:
\begin{itemize}[leftmargin=4.5em]
\item[\nf{\textbf{DoNA}}]
\!If $H\equiv 1$ (hence $G\equiv G_1\equiv 1$, $G_2\equiv 0$, $g(t)=t^{1/\alpha}$), then  
\[
f(t)\coloneqq
\begin{cases}
t^{(p/\alpha)\wedge (1-1/\alpha) }(1+|\log t|\1_{\{p=\alpha-1\}}), &\alpha\in(1,2)\text{ and \eqref{eq:symmetry_como} fails}, \\
t^{p/\alpha}, &\alpha\in(1,2)\text{ and \eqref{eq:symmetry_como} holds or }\alpha\in(0,1).
\end{cases}
\]
\item[\nf{\textbf{DoNNA}}]
\!If $H\in\GSV_0$, $G$ in~\eqref{eq:defn_G_C} is in $\CSV_0(G_1,G_2)$, $G_2\in\SV_0$ and $g(t)=t^{1/\alpha}G(t)$, then $f=G_2$.
\end{itemize}
\end{theorem}

\begin{remark}
\label{rem:TvsC}
Let the assumptions of both Theorems~\ref{thm:gen_bound_thin} and~\ref{thm:gen_bound_como} hold with the same parameter $p>0$ (see Corollary~\ref{cor:augmented_stable_rates}, where these results are applied to the class of \emph{augmented stable processes}) and suppose $\alpha>1$. The guaranteed convergence rate (as a power of $t$) of Theorem~\ref{thm:gen_bound_como} is better than that of Theorem~\ref{thm:gen_bound_thin} by up to a factor of $2=\lceil\alpha\rceil$ if either both balancing conditions hold or if they both fail and $p$ is sufficiently small. This suggests that accumulating small errors in every jump (which is how the comonotonic coupling is constructed) is asymptotically better than doing a perfect match for as many jumps as possible (which is how the thinning coupling is constructed).
\end{remark}

\subsection{Illustrations of Paths of Solutions to SDEs and Tightness}\label{sec:illustrations}
In the small-time DoA regime, even small discrepancies in the driving L\'evy processes can be magnified due to the sensitivity of the Lipschitz function $V$ in the SDE dynamics. This phenomenon is especially pronounced in the presence of jumps, where slight differences in the jumps (magnitude or angle) can result in markedly different solution paths. Figure~\ref{fig:trajectory_SDEs} shows that trajectories driven by nearly identical α-stable processes can diverge significantly, even when sharing the same initial condition and drift structure. This highlights the need for precise $L^p$-type estimates of pathwise differences to quantify convergence in probability.

\begin{figure}[ht]
    \centering
  \begin{subfigure}[ht]{.49\linewidth}
    \centering\includegraphics[width=1\linewidth]{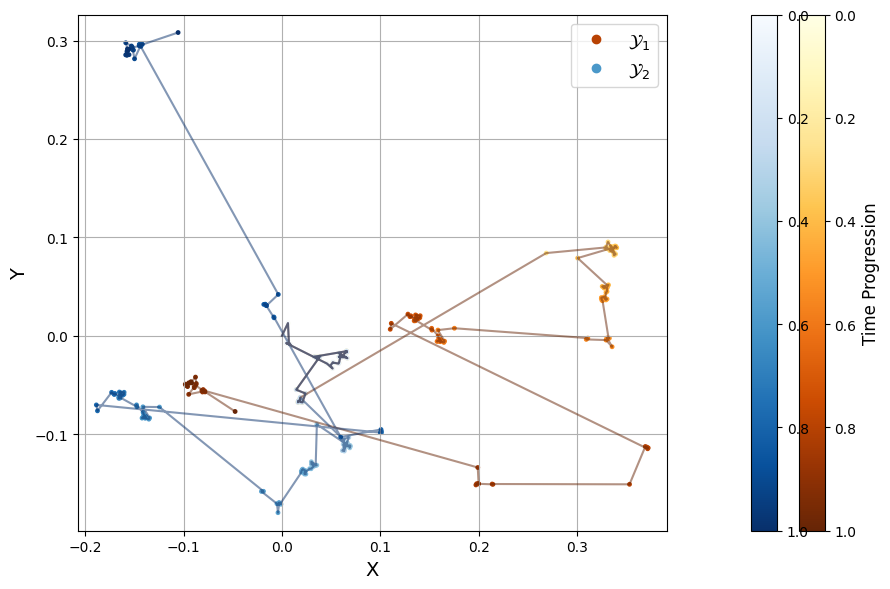}
  \end{subfigure}
  \begin{subfigure}[ht]{.49\linewidth}
    \centering\includegraphics[width=1\linewidth]{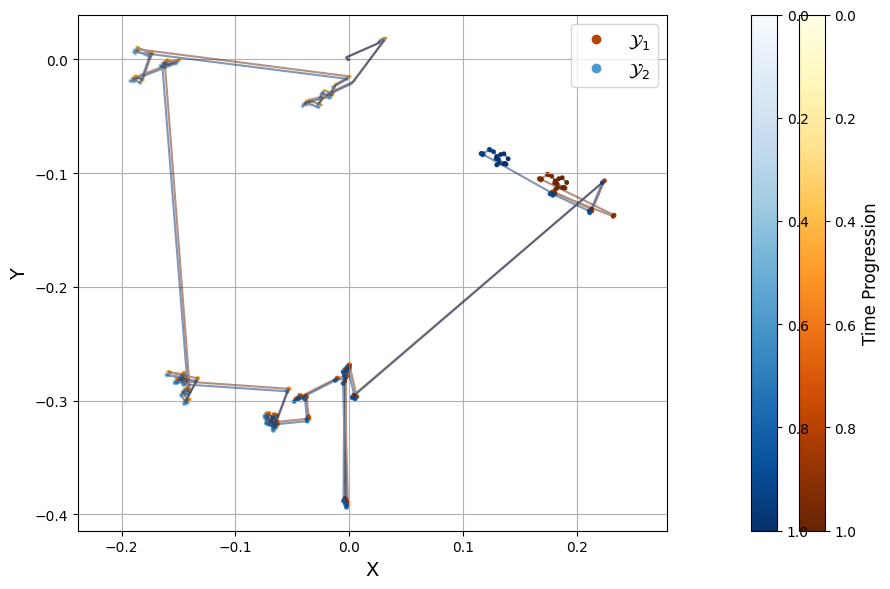}
  \end{subfigure}
    \caption{\small Each picture depicts the paths of the SDE solutions $\Y(t)\coloneqq \int_0^t V(\Y(r-))\D \bm{Y}(r)$ in $\R^2$ for $t \in [0,1]$, where $\bm{Y}\in \R$ is a one-dimensional $\alpha$-stable process (red) or a process in its DoA (blue) and $\bm{V}(\bm{x})$ is the rotation matrix with angle $|\bm{x}|$.}
    \label{fig:trajectory_SDEs}
\end{figure}

The $L^p$-bounds in~\cite{WassersteinPaper} are designed for direct comparisons between L\'evy processes and are unsuitable for the nonlinear SDEs considered here (see Subsection~\ref{subsec:additive} for the additive case). Notably, Figure~\ref{fig:trajectory_SDEs} also shows examples where the solution paths remain close throughout the interval $[0,1]$. These examples indicate that the primary factor in trajectory similarity is whether the paths jump in the same direction. Conversely, differences in jump magnitudes, though nontrivial, tend to cause smaller deviations. This is evident in the first picture of Figure~\ref{fig:trajectory_SDEs}, where the paths diverge dramatically after jumps in different directions.

This observation is key to understanding the role and effectiveness of the couplings used. The thinning and comonotonic couplings differ significantly in their constructions. The thinning coupling starts by identifying a common dominating L\'evy measure, such as the sum of the original measures, under which both L\'evy measures are absolutely continuous with bounded densities. A Poisson random measure with this intensity is thinned to generate the jump measures of the individual processes, maximizing shared jumps. In contrast, the comonotonic coupling assumes that the L\'evy processes have radial decompositions with potentially distinct angular components. A shared angular measure is constructed, and LePage's series representation synchronizes jump directions. Magnitudes are coupled via the comonotonic (optimal transport) coupling, using right-continuous inverses of the radial tail L\'evy measures at common Poisson epochs. This yields a coupling, where the directional alignment of the jumps are preserved, while allowing for differences in jump size.

Consequently, the comonotonic coupling ensures that jumps occur in the same direction, reducing trajectory divergence due to directional differences. This sometimes improves convergence rates compared to the thinning coupling, which may permit mismatched jump directions, as can be seen in Theorems~\ref{thm:main_res_T} \&~\ref{thm:main_res_C}. See, for example,~\cite[Fig.~1]{WassersteinPaper} for a visual comparison that highlights the distinction between the couplings of the jumps.

Figure~\ref{fig:tail_thin} illustrates the probabilistic behaviour of the pathwise deviations of two L\'evy-driven SDEs, induced by subjecting the L\'evy drivers to the thinning coupling. The process $\bm{Z}$ is an isotropic $\alpha$-stable process with $\alpha=.7$, while $\bm{X}$ is an exponentially tempered stable process, whose jump measure is the result of thinning the jump measure of $\bm{Z}$. The chosen function $V$ takes values in the space of orthogonal matrices (i.e., rotation matrices), so that the difference in the trajectories of $\ZZ$ and $\X_t$ is determined by the aggregation of missing jumps in the thinned process as well as the differences in the rotation angles. The picture shows the numerical estimate of tail probabilities $\p(f(t)^{-1}\|\X_t - \ZZ\|_{[0,T]} > r)$ (where $f(t)=t^{1/\alpha}$ is as in Theorem~\ref{thm:main_res_T}(DoNA)) for several values of $t$, revealing a consistent decay to $0$ in the point-wise supremum as $r$ increases, and illustrating the tightness of this family. As $t \downarrow 0$, the processes become increasingly closer (fewer jumps are thinned out), and the corresponding tail probabilities quickly drop to a small level, remain flat for long and then decay polynomially. These numerical results align with the theoretical guarantees for the thinning coupling in Theorem~\ref{thm:main_res_T}(DoNA): the family $\{f(t)^{-1}\|\X_t-\ZZ\|_{[0,T]}\}_{t \in (0,1]}$ is tight. Moreover, as appreciated in the logarithmic scale of the graph, the supremum tail function $\sup_{t\in(0,1]}\p(f(t)^{-1}\|\X_t - \ZZ\|_{[0,T]} > r)$ appears to be very heavy, reinforcing the lack of uniform integrability and lack of control in Wasserstein distance of the aforementioned family of random variables.

\begin{figure}[ht]
    \centering
        \includegraphics[width=.95\linewidth]{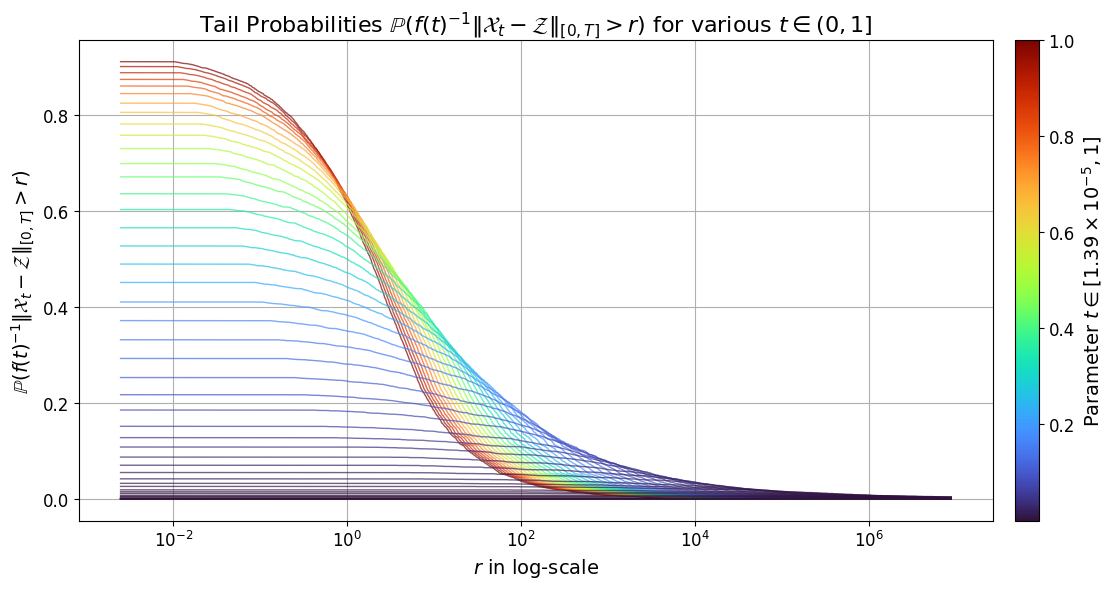}
    \caption{Tail probabilities of the rescaled pathwise difference between two coupled (via the thinning coupling introduced in Section~\ref{sec_thinning_coup}) solutions to the SDEs $\X_t(t) = \int_0^t V(\X_t(r-)) \D \bm{X}_t(r)$ and $\ZZ(t) = \int_0^t V(\ZZ(r-)) \D \bm{Z}(r)$ in $\R^2$ driven by an $\alpha$-stable L\'evy processes with $\alpha=0.7$ and $V(\bm{x})$ is a rotation matrix with angle $|\bm{x}|$. The picture shows the full family of tail curves indexed by $t\in\{\exp(1-\e^{k/20})\}_{k\in\{0,\ldots,50\}}$ of $5000$ sampled paths.}
    \label{fig:tail_thin}
\end{figure}

\subsection{Organisation and Strategy of the Proofs.} 
\label{subsec:outline}
The organization of the proofs for the main results is described here, along with a roadmap for establishing convergence rates in Wasserstein distance and probability (for both thinning and comonotonic couplings).
\begin{itemize}[leftmargin=1.5em, nosep]
\item[\textbf{1.}] First, Section~\ref{sec:general_bounds} establishes general bounds on the uniform $L^2$-distance $\E[\|\Y_1-\Y_2\|_{[0,T]}^2]^{1/2}$ (resp. uniform $L^1$-distance $\E[\|\Y_1-\Y_2\|_{[0,T]}]$) for SDEs driven by coupled L\'evy processes $\bm{Y}_1$ and $\bm{Y}_2$ with a finite second moment (resp. with finite mean and paths of finite variation) in terms of the drivers' characteristics. These bounds are presented in Theorems~\ref{thm:gen_bound_thin} \&~\ref{thm:gen_bound_como}, and their proofs are based on grouping up the appropriate components, using the Lipschitz assumptions, the Burkholder--Davis--Gundy--Novikov (BDGN) maximal inequality (sometimes known as the Burkholder--Davis--Gundy or BDG inequality), and Gr\"onwall's inequality. 
\item[\textbf{2.}] Section~\ref{sec:application_DoA} introduces a martingale change of measure $M_\theta$ (given in~\eqref{eq:M_theta_thinning} for the thinning coupling and~\eqref{eq:M_theta_como} for the comonotonic coupling) that induces an exponential tempering, so that the drivers $\bm{X}_t$ and $\bm{Z}$ of the SDEs in~\eqref{eq:main_SDE_setting} have finite moments of any order. This martingale change of measure is designed carefully to not depend on $t$. We then apply the bounds from step 1 under the equivalent probability measure. The resulting bound for $\E[\|\X_t-\ZZ\|_{[0,T]}^pM_\theta]$ will be expressed in terms of integrals and other expressions of the characteristics. These will then be bounded as functions of $t$ under appropriate conditions in Propositions~\ref{prop:integrals_DoA_thin} and~\ref{prop:integrals_DoA_como}. For processes in DoNA (resp. DoNNA), the resulting bound will be of the form $\Oh(t^q)$ for some $q>0$ (resp. $\Oh(f(t))$ for some $f \in \SV_0$) and fixed $\theta>0$. 
\item[\textbf{3.}] Finally, Subsection~\ref{subsec:proofs_main_results} provides the proofs of the main results with the aid of Lemma~\ref{lem:Lp-Girsanov-to-cip}. This auxiliary lemma essentially translates bounds in $L^p$ under an equivalent probability measure to establish tightness of families of random variables under the original measure~$\p$.
\end{itemize}

\subsection{Literature Comparison}
\label{subsec:literature}

This work draws inspiration from~\cite{MR3784492,WassersteinPaper,deng2024optimalwasserstein1distancesdes,MR3332269,deng2024totalvariationdistancesdes}. The primary motivation,~\cite{WassersteinPaper}, studies the rate of convergence of multivariate Lévy processes to their stable limits in the small-time regime (i.e., the rate of$\bm{X}_t$ converging weakly to $\bm{Z}$ as $t \downarrow 0$) via the thinning and comonotonic couplings introduced therein. Lower bounds (derived analytically) and upper bounds (derived via the couplings) on the Wasserstein distance $\mW_q(\bm{X}_t,\bm{Z})$ are derived, often with matching rates, demonstrating the effectiveness of these couplings in the DoNA and DoNNA. The analysis reveals a known but understudied contrast between the regimes: convergence is polynomial in DoNA, while in DONNA it is slower than any slowly varying function integrable at $0$ against the measure $t^{-1}\D t$. This mirrors the contrasting behaviour in Theorems~\ref{thm:main_res_T} \&~\ref{thm:main_res_C}.

While this paper draws on the coupling constructions from~\cite{WassersteinPaper}, the technical development diverges significantly. Specifically, the $L^p$ bounds from~\cite[Sec.~3]{WassersteinPaper} are not directly applicable to the SDEs here, as there are no general ways to directly link the proximity of the drivers in $L^p$ to that of the SDE solutions, as illustrated by the many convergence analyses of SDEs and Wong-Zakai-type results, e.g.,~\cite{kosenkova2020orderconvergenceweakwongzakai,pavlyukevich2025stronguniformwongzakaiapproximations,MR3332269,MR2759203,deng2024optimalwasserstein1distancesdes}. Section~\ref{sec:general_bounds} therefore derives new bounds, relying on tools such as Gr\"onwall's and BDGN's maximal inequalities. While using the coupling framework of~\cite{WassersteinPaper}, these estimates are crucial for handling the complexities of L\'evy-driven SDEs, highlighting the distinct technical challenges posed by this context. In fact, our focus on convergence in probability, rather than Wasserstein distances, leads to upper bounds that can be much faster than the general lower bounds established in~\cite{WassersteinPaper} for Wasserstein distances. Additionally, some techniques used in~\cite{WassersteinPaper} appear not to be applicable in our context. For instance, the partial duality of Wasserstein distances is absent in our context, so we cannot obtain lower bounds in DoNA. Similarly, the use of auxiliary $L^q$ norms (along with Jensen's inequality) for certain $q\in(0,\alpha)\cap(0,1]$ used in~\cite{WassersteinPaper} appears to lead to unusable bounds in the context of SDEs. Indeed, such bounds lead to ordinary differential inequalities for Wasserstein distances with fractional powers, rendering any Bellman--Gr\"onwall-type inequality unusable due to the lack of uniqueness for ODE solutions with fractional powers.

The stability of SDEs were the subject of study of~\cite{deng2024totalvariationdistancesdes,deng2024optimalwasserstein1distancesdes,PTRFAleks,MR4733911,MR4791608}. The papers~\cite{PTRFAleks,MR4733911,MR4791608} analyse stationarity and convergence rates, while~\cite{deng2024totalvariationdistancesdes,deng2024optimalwasserstein1distancesdes} study the total variation and $L^1$-Wasserstein distances between the marginal laws of solutions to SDEs driven by an isotropic $\alpha$-stable process and Brownian motion. The obtained bounds only yield vanishing rates when $\alpha \ua 2$ for marginals at fixed time $t$ and even uniformly on intervals $[t,\infty)$ for $t>0$. In contrast, our work focuses on the uniform convergence in probability on $[0,T]$ for the $\alpha$-stable small-time DoA, in addition to our non-asymptotic bounds on the Wasserstein distance between the solutions of SDEs driven by two general L\'evy processes. This change of regime, introduces different mathematical challenges and renders ergodic techniques (e.g., via Lyapunov conditions) from~\cite{deng2024totalvariationdistancesdes} inapplicable to our setting. 
Similarly, the techniques in~\cite{deng2024optimalwasserstein1distancesdes} are not suited for quantifying the convergence of L\'evy-driven SDEs with drivers in the small-time DoA of stable processes.

The papers~\cite{MR3332269,MR4307706,MR4791608} analyse and compare solutions of SDEs with additive noise (i.e., as in~\eqref{eq:simple_lin_SDEs}). While the authors of~\cite{MR4307706} are generally interested in the regularisation by noise (i.e., understanding the limit behaviour as the additive noise vanishes), both papers~\cite{MR4307706,MR4791608} analyse an abrupt or cutoff thermalisation phenomenon. In contrast, the authors of~\cite{MR3332269} are interested comparing solutions of SDEs as the driving L\'evy processes become closer to each other, obtaining non-asymptotic bounds on the truncated Wasserstein distance (defined in~\eqref{eq:general_wasserstein_trunk}, using the cost function $\rho(\bm{x},\bm{y}) = |\bm{x} - \bm{y}| \wedge 1$) in terms of the drivers' characteristics. Notably, the authors work with the truncated Wasserstein distance to ensure finiteness of the metric even in the presence of heavy tails, using the elementary estimate $y \wedge 1 \le 4\arctan(y)$ for $y \ge 0$, in combination with It\^o's formula and Gr\"onwall's lemma, to establish their result for this class of SDEs. These techniques, however, do not easily extend to more general SDEs or the stronger $L^p$-Wasserstein distances considered in the present paper (see Section~\ref{subsec:additive}). However, there is a partial agreement in the coupling used therein and the comonotonic coupling introduced in~\cite{WassersteinPaper} and used here. In~\cite{MR3332269}, the authors match the intensities of the large-jump components of both driving L\'evy processes to an arbitrary parameter $r>0$. Then, the large-jump components are coupled to have the same number of jumps and so that each pair of large jumps is coupled optimally (i.e., attaining the optimal transport cost, which coincides with the comonotonic coupling of real random variables), while the small-jump components are simply coupled independently. By letting $r\to\infty$ in the resulting bounds, we arrive at a bound that is very similar to the corresponding bound for the comonotonic coupling for this class of SDEs under the bounded metric. Unfortunately, even the limiting bound in~\cite{MR3332269}, absent of the structural assumption on its radial decomposition, is written in terms of limits of Wasserstein distances, which are far from explicit. This difficulty can be appreciated in the one-dimensional nature of the examples in~\cite{MR3332269}, since in one dimension the Wasserstein distance is explicitly written in terms of the comonotonic coupling.

In~\cite{MR4779850}, the authors investigate optimal Markovian couplings of Markov processes, where optimality is defined through the minimisation of concave transport costs between the time-marginal distributions of the coupled processes. Their focus is on one-dimensional L\'evy processes with finite activity and unimodal (but not necessarily symmetric) jump distributions. The main result of the paper is that, under certain structural conditions, an explicit optimal Markovian coupling can be constructed in this setting. This construction combines McCann’s theory of optimal transport and Rogers’ characterisation of couplings for random walks with a novel uniformisation technique, which allows the authors to address all finite-activity L\'evy processes. The method relies critically on the finite jump activity assumption, so it does not directly extend to even the infinite-activity L\'evy setting considered in~\cite{WassersteinPaper}, where the processes (and their $\alpha$-stable limits) all have infinite activity.

Several works have analysed Euler--Maruyama-type discretisations of SDEs, Wong--Zakai-type theorems and convergence rates of other numerical schemes, including~\cite{MR2759203,MR2802466,pavlyukevich2025stronguniformwongzakaiapproximations,kosenkova2020orderconvergenceweakwongzakai}, to name a few. Convergence properties of such schemes usually are related to the small-time and small-jump dynamics of the SDE and its L\'evy driver, with strong error bounds and multilevel variations often involving coupling the small-jump components with a further Brownian term via, e.g., the celebrated Koml\'os--Major--Tusn\'ady coupling. These approaches are especially effective for L\'evy processes with highly infinite activity, such as $\alpha$-stable processes, enabling nearly optimal convergence rates (up to logarithmic factors) for processes with highly active small jumps. For instance,~\cite[Thm~3.1]{MR2759203} gives a near-optimal bound between the solution of the SDEs driven by the original process and the process with the Gaussian approximation. Crucially, the only difference in the drivers is a small-jump martingale (with fixed cutoff and exponential moments of any order) and a Brownian motion with the same variance. These features, however, are never present in the small-time DoA studied in the present work. In fact, any discrepancy in the L\'evy measures will eventually be pushed to the large-jump components after the normalisation required for the weak convergence.

\subsection{The Conflict Between Jump Tail Indices at Zero and Infinity}
\label{sec:failed_ideas}

This subsection discusses ideas that initially seemed promising but led to unexpected difficulties, thereby elucidating the limitations of existing theory.

It is natural to expect the regularity and stability of processes to carry over to the solutions of driven SDEs through a Bellman--Gr\"onwall inequality. Specifically, one might expect the distance between SDE solutions to be expressed in terms of model characteristics and the distance between drivers. However, Wong--Zakai theorems (e.g.,~\cite{kosenkova2020orderconvergenceweakwongzakai,pavlyukevich2025stronguniformwongzakaiapproximations}) show that this relationship is not always linear or simple, making SDEs with additive noise an exception rather than the rule (see Corollary~\ref{cor:additive_SDE}). Instead, similar bounds required examining the driver's characteristics and couplings (see Theorems~\ref{thm:gen_bound_thin} and~\ref{thm:gen_bound_como}). In turn, these bounds required inequalities such as the BDGN and Bellman--Gr\"onwall inequalities. As expected, direct applications for general semimartingales require second moment assumptions; reducing this to a first moment requires paths of finite variation.

It might be possible to partially fill this gap using a continuum of possibilities balancing the summability of small jumps and the finiteness of moments. However, the stable limit here has jumps in $\ell^2\setminus\ell^{\alpha}$ a.s. and with infinite $\alpha$-moment, precluding direct Bellman--Gr\"onwall inequalities. One can attempt to solve this as in~\cite{WassersteinPaper}, by applying Jensen's inequality with a fractional power (or any other concave function). However, this leads to ordinary differential inequalities with fractional powers and whose corresponding ODEs have non-unique solutions. In turn, this lack of uniqueness makes \emph{any} type of Bellman--Gr\"onwall inequality unfructiferous, since even zero initial conditions cannot force the solutions to remain nil. Even when considering a bounded or truncated metric (as in~\cite{MR3332269}), this issue is not immediately resolved, since we are eventually forced to employ the BDGN inequality, leading to the same problem. In some sense, a tailored BDGN inequality for a bounded metric appears to be necessary to control the convergence in Wasserstein distance of general SDE solutions in the small-time stable DoA. We are only able to derive rates for the uniform convergence in probability through an exponential change of measure that gives all processes sufficiently many moments. 

Using a tailored BDGN inequality for the bounded Wasserstein metric can lead to a convergence rate, as the authors found out. This requires varying the cutoff level for large versus small jumps and taking limits, potentially leading to convergence. However, the difficulty in recovering integrability at both $0$ and $\infty$, and the non-uniqueness of generalized Bellman--Gr\"onwall inequalities, mean that a cutoff level leading to convergence would also lead to a very slow convergence rate. Indeed, this approach yields a convergence rate that is incorrect by at least a logarithmic factor, rendering it intractable, as even with DoNA and a truncated metric, the rate would be some power of logarithm.

\section{Special Cases}
\label{sec:special}

\subsection{Wasserstein Distances for SDEs with Additive Noise}
\label{subsec:additive}
This section considers $L^q$-type distances to derive convergence rates for solutions of SDEs with additive noise. Let $\X = (\X(s))_{s \in [0,T]}$ and $\Y = (\Y(s))_{s \in [0,T]}$ be $\R^m$-valued stochastic processes, and $\varrho$ be a metric on $\R^m$. Define the uniform $L^q$-Wasserstein distance between paths via
\begin{equation}\label{eq:general_wasserstein_trunk}
\mW_q(\X, \Y)
\coloneqq \inf_{\X' \eqd \X, \, \Y' \eqd \Y}
\E\bigg[\sup_{s \in [0,T]} \varrho(\X'(s), \Y'(s))^q\bigg]^{1/(q \vee 1)},
\quad \text{for }q > 0,
\end{equation}
where the infimum is taken over all couplings $(\X', \Y')$ such that $\X' \eqd \X$ and $\Y' \eqd \Y$. Throughout, we typically use the Euclidean metric $\varrho(\bm{x}, \bm{y}) = |\bm{x} - \bm{y}|$ for $\bm{x},\bm{y} \in \R^m$, but in some cases it is advantageous to consider the truncated Wasserstein distance: $\varrho(\bm{x}, \bm{y}) = |\bm{x} - \bm{y}| \wedge 1$ for $\bm{x},\bm{y} \in \R^m$ (see, e.g.,~\cite{MR3332269}). For $\R^d$-valued random variables $\bm{\xi}$ and $\bm{\zeta}$, the corresponding $L^q$-Wasserstein distance is defined by 
\[
\mW_{q}(\bm{\xi},\bm{\zeta}) 
= \inf_{\bm{\xi}' \eqd \bm{\xi},\,\bm{\zeta}' 
\eqd \bm{\zeta}} \E\big[\varrho(\bm{\xi}', \bm{\zeta}')^q\big]^{1/(q \vee 1)},
\qquad\text{for }q>0,
\]
where the infimum is taken over all couplings $(\bm{\xi}',\bm{\zeta}')$ with $\bm{\xi}'\eqd\bm{\xi}$ and $\bm{\zeta}'\eqd\bm{\zeta}$.

In the ensuing discussion, we will work with the standard Wasserstein distance, corresponding to $\varrho(\bm{x}, \bm{y}) = |\bm{x} - \bm{y}|$ for $\bm{x},\bm{y} \in \R^d$. Under DoNA,~\cite[Thm~2.2]{WassersteinPaper} gives an explicit rate at which the rescaled L\'evy process $\bm{X}_t$ converges to its stable limit $\bm{Z}$ in Wasserstein distance as $t \da 0$. The corresponding result in DoNNA is given in~\cite[Thm~2.5]{WassersteinPaper}. For the class of SDEs with additive noise, bounds for $\mW_q(\bm{X}_t, \bm{Z})$ can be transferred to $\mW_q(\X_t, \ZZ)$. Consider now the following SDEs with additive noise (as in~\cite{MR3332269}):
\begin{equation}
\label{eq:simple_lin_SDEs}
\X_t(s) = \bm{x} + \int_0^s V(\X_t(r-)) \, \D r 
    + \bm{X}_t(s),
\qquad
\ZZ(s) = \bm{x} + \int_0^s V(\ZZ(r-)) \, \D r 
    + \bm{Z}(s),
\end{equation}
for $s \in [0,T]$ and $t \in (0,1]$, where $\bm{X}_t$ is a pure-jump L\'evy process in the DoA of the $\alpha$-stable process $\bm{Z}$, both in $\R^d$, and $V:\R^d \mapsto \R^{d}$ is assumed to be Lipschitz continuous (in the notation used above, we have $d=m$), see also Section~\ref{subsec:additive-SDE}.

\begin{corollary}
\label{cor:additive_SDE}
Let $\bm{X}$ be in the small-time DoA of the $\alpha$-stable process $\bm{Z}$ with $\alpha \in (0,2)\setminus\{1\}$. Fix any $T>0$ and let $\X_t$ and $\ZZ$ be the solutions to the SDEs given in~\eqref{eq:simple_lin_SDEs} and $q \in (0,1]\cap(0,\alpha)$ satisfy $\E[|\bm{X}_1|^q]<\infty$. Then, 
\[
\mW_q(\X_t,\ZZ)
\le \Oh(\mW_q(\bm{X}_t,\bm{Z})), \quad \text{ as }t \da 0.
\]
Moreover, if Assumption~(\nameref{asm:C}) holds for $p \in (0,\infty)\setminus\{\alpha-1\}$ and $\delta>0$, then the following hold.
\begin{itemize}[leftmargin=4.5em]
\item[\nf{\textbf{DoNA}}] Suppose $H\equiv 1$ and $g(t)=t^{1/\alpha}$, then, as $t \da 0$, 
\[
\mW_q\big(\bm{X}_t,\bm{Z}\big)
=\begin{dcases}
\Oh\big(t^{\min\{pq/\alpha,1-q/\alpha\}}\big),&\alpha\in(0,1),\\
\Oh\big(t^{q\min\{p/\alpha,1-1/\alpha\}}\big),&\alpha\in(1,2).
\end{dcases}
\]
\item[\nf{\textbf{DoNNA}}] Suppose $H\in\GSV_0$, $G$ in~\eqref{eq:defn_G_C} is in $\CSV_0(G_1,G_2)$, $G_2\in\SV_0$ and $g(t)=t^{1/\alpha}G(t)$. If $q \ne \alpha/(p+1),\alpha/(\alpha\delta+1)$, then 
\[
\mW_q(\bm{X}_t,\bm{Z})
=\Oh ( G_2(t)^q),
\qquad\text{as }t \da 0.
\]
\end{itemize}
\end{corollary}

\begin{proof}
First, we note from Proposition~\ref{prop:simple_lin_SDE_gen_bound} for all $q \in (0,1]\cap(0,\alpha)$, that
\begin{equation}
\label{eq:Oh_wasserstein_drivers}
\mW_q(\X_t,\ZZ) = \Oh(\mW_q(\bm{X}_t,\bm{Z})), \quad \text{ for all }T>0, \text{ as }t \da 0.
\end{equation}
The specified convergence rates require Assumption~(\nameref{asm:C}) to hold. Assume that $\alpha \in (0,2)\setminus\{1\}$, and let $q \in (0,1]\cap(0,\alpha)$ and couple $\bm{X}_t$ and $\bm{Z}$ through the comonotonic coupling. The claims then follow directly in both cases from~\eqref{eq:Oh_wasserstein_drivers} together with~\cite[Prop.~2.3 \& Thm~2.5]{WassersteinPaper}.
\end{proof}

\begin{remark}
(a) In Corollary~\ref{cor:additive_SDE}, we imposed the assumptions required for the comonotonic coupling, namely, Assumption~(\nameref{asm:C}). It would have also been possible to instead impose Assumption~(\nameref{asm:T}) and employ the thinning coupling. In this case, by invoking~\eqref{eq:Oh_wasserstein_drivers} and~\cite[Thm~2.2]{WassersteinPaper}, the result follows for DoNA. Although the corresponding bounds on $\mW_q(\bm{X}_t, \bm{Z})$ are not explicitly given in~\cite{WassersteinPaper} for DoNNA under Assumption~(\nameref{asm:T}), they can be obtained from~\cite[Prop.~3.2]{WassersteinPaper} and Proposition~\ref{prop:integrals_DoA_thin} below.\\
(b) The methods used in Corollary~\ref{cor:additive_SDE} depend crucially on the fact that we have an SDE with additive noise, and the methods cannot be generalised to hold in the general case of~\eqref{eq:main_SDE_setting}. Indeed, in the case of SDEs with additive noise, the error is essentially an aggregate of the errors of the drivers. However, in the context of our more general SDEs, a minor error in the drivers, coupled with a rotation from the function $V$, can lead to divergent paths for the solutions, see Figure~\ref{fig:trajectory_SDEs}. Hence, more care is required, which leads to many difficulties and pitfalls that are discussed in greater detail in Section~\ref{sec:failed_ideas} below.
\end{remark}

\subsection{Augmented \texorpdfstring{$\alpha$}{alpha}-stable Drivers}\label{subsec:augment_stable}

Here, the general class of \emph{augmented $\alpha$-stable processes} from~\cite[Sec.~2.5]{WassersteinPaper} is recalled and our main results from Section~\ref{sec:main} are applied. Augmented $\alpha$-stable processes (Definition~\ref{def:augmented-stable} below) form a rich and flexible class of L\'evy processes that arise naturally in numerous applications such as finance, physics, and engineering. As highlighted in Example~\ref{ex:subclass_augmented_stable}, this class encompasses a wide variety of important models such as Rosi\'nski's tempered stable processes, truncated stable processes, and many meromorphic L\'evy processes. Their relevance stems from the fact that they may modify the heavy-tailed behaviour and preserve most of the small-jump activity of stable processes while allowing for tractable modifications in the tail or directional structure via the tempering function~$\calQ$. 

\begin{defin}
\label{def:augmented-stable}    
A L\'evy process $\bm{X}=(\bm{X}(t))_{t \in [0,1]}$ is an augmented $\alpha$-stable process if it has no Gaussian component, and, for some $\alpha\in(0,2]$, its L\'evy measure $\nu_{\bm{X}}$ has the form
\begin{equation} \label{eq:temp_stab_alpha}
\nu_{\bm{X}}(A)
= \int_0^\infty\int_{\Sp^{d-1}}  \1_{A}(x\bm{v}) \calQ(x,\bm{v})  \sigma(\D \bm{v})x^{-\alpha-1}\D x, \quad \text{ for } A \in \mathcal{B}(\R^d_{\bm{0}}),
\end{equation}
where $\sigma$ is a probability measure on $\mathcal{B}(\Sp^{d-1})$ and $\calQ:(0,\infty)\times \Sp^{d-1} \mapsto [0,\infty)$ is a measurable function satisfying $\calQ(\cdot,\bm{v})\in\SV_0$ for $\bm{v}\in\Sp^{d-1}$. We will refer to $\calQ$ as the \emph{augmenting} function of $\bm{X}$.
\end{defin}

For $\nu_{\bm{X}}$ to be a proper L\'evy measure, it is necessary that $\int_0^\infty (x^2\wedge 1)\calQ(x,\bm{v})x^{-\alpha-1}\D x<\infty$ for $\sigma$-a.e. $\bm{v}\in\Sp^{d-1}$. While the class of augmented stable processes could be extended beyond the requirement that $\calQ(\cdot,\bm{v})\in\SV_0$ for all $\bm{v}\in\Sp^{d-1}$, our definition is motivated by the characterisation in Theorem~(\nameref{thm:small_time_domain_stable}) below of the small-time DoA of stable processes, and hence we impose this restriction on $\calQ$.

\begin{corollary}
\label{cor:augmented_stable_rates}
Let $\alpha \in (0,2)\setminus\{1\}$, and assume that $\bm{X}$ is an augmented $\alpha$-stable process with L\'evy measure $\nu_{\bm{X}}$ given as in~\eqref{eq:temp_stab_alpha}, and $\bm{Z}$ is an $\alpha$-stable process with L\'evy measure $\nu_{\bm{Z}}$ given as in~\eqref{eq:mu_measure_defn}. Then, in the following cases, for every fixed $T>0$, there exist couplings between $\bm{X}_t$ and $\bm{Z}$ for which the solutions $\X_t$ and $\ZZ$ to the SDEs given in~\eqref{eq:main_SDE_setting}, driven by $\bm{X}_t\eqd(\bm{X}(st)/g(t))_{s\in[0,1]}$ and $\bm{Z}$ respectively, make the family $\{f(t)^{-1}\|\X_t-\ZZ\|_{[0,T]}\}_{t\in(0,1]}$ tight, where $f,g$ are specified below.

\noindent{\nf{\textbf{(a) (DoNA)}}}
If $|\calQ(x,\bm{v})-c_\alpha| \le K(1\wedge x^p)$ for some $K,p>0$ and all $x>0$, $\bm{v}\in \Sp^{d-1}$, then
\[
g(t)\coloneqq t^{1/\alpha},
\quad\text{and}\quad
f(t)\coloneqq
\begin{cases}
t^{(p/\alpha)\wedge (1-1/\alpha) }, &\alpha\in(1,2)\text{ and \eqref{eq:symmetry_thin} fails}, \\
t^{p/\alpha}, &\alpha\in(1,2)\text{ and \eqref{eq:symmetry_thin} holds or $\alpha\in(0,1)$}.
\end{cases}
\]

\noindent{\nf{\textbf{(b) (DoNNA)}}} Assume $\calQ(x,\bm{v})\!=\! c_\alpha H(x)$ for all $x>0$, $\bm{v}\in \Sp^{d-1}$ and some monotone $H\in\SV_0\cap C^1$. 

\leftskip5mm\noindent{\nf{\textbf{(i)}}} If $H\in\CSV_0(H_1,H_2)$, $H_2\in\SV_0$, $t\mapsto tH(t)^{-1/\alpha}$ is increasing with inverse $t\mapsto tG(t^\alpha)$ for some $G\in\SV_0$, then 
\[
g(t)\coloneqq t^{1/\alpha}G(t), \quad \text{ and } \quad f(t)\coloneqq
H_2(g(t))^{1/\lceil\alpha\rceil}.
\]

\noindent{\nf{\textbf{(ii)}}} Set $\varrho(x)\coloneqq c_\alpha\int_x^\infty H(y) y^{-\alpha-1}\D y$ and let $\varrho^{\la}$ be its right-continuous inverse. Suppose $G:x\mapsto H(\varrho^{\la}(1/x))^{1/\alpha}\in \CSV_0(G_1,G_2)$, for some function $G_2\in\SV_0$. Then $g(t)\coloneqq t^{1/\alpha}G(t)$ and $f=G_2$.
\end{corollary}

\begin{proof}
Parts (a) and (b-ii) follow from Theorem~\ref{thm:main_res_C} and Part (b-i) follows from Theorem~\ref{thm:main_res_T}.
\end{proof}

\begin{remark}
\label{rem:G2-vs-H2}
Given the similarity in the assumptions of Parts (b-i) and (b-ii) in Corollary~\ref{cor:augmented_stable_rates}, it is natural to compare both cases:\\
(a) First, we show that the slowly varying functions $G^\mft$ and $G^\mfc$ from the thinning coupling and comonotonic coupling, respectively, are asymptotically equivalent. Let $G^\mft$ be as in Part (b-i), so that $s\mapsto sG^\mft(s^\alpha)$ is the inverse of map $s\mapsto sH(s)^{-1/\alpha}$ and, by~\eqref{eq:G-H}, $G^\mft(s)=H(g^\mft(s))^{1/\alpha}$ (where $g^\mft(s)=s^{1/\alpha}G^\mft(s)$). Then, setting $k_\alpha\coloneqq c_\alpha/\alpha$, by Karamata's theorem~\cite[Thm~1.5.11]{MR1015093}, $\varrho(x)\sim k_\alpha H(x)x^{-\alpha}$ and hence $\varrho^\la(1/x)\sim (k_\alpha x)^{1/\alpha}G^\mfc(k_\alpha x)=g^\mfc(k_\alpha x)\sim k_\alpha^{1/\alpha}g^\mfc(x)$ as $x\da 0$. Thus, 
\[
G^\mfc(x)
=H(\varrho^\la(1/x))^{1/\alpha}
\sim H(g^\mfc(k_\alpha x))^{1/\alpha}
\sim H(g^\mfc(x))^{1/\alpha}, 
\quad \text{ as }x \da 0.
\]
Then, taking $x=s^\alpha$ and doing elementary manipulations, we deduce that
\[
s G^\mfc(s^\alpha)H(s G^\mfc(s^\alpha))^{-1/\alpha}
\sim s,\quad s\da 0,
\]
which shows that $s\mapsto s G^\mfc(s^\alpha)$ is an asymptotic inverse of $s\mapsto sH(s)^{-1/\alpha}$ and hence, asymptotically equivalent to $s\mapsto s G^\mft(s^\alpha)$. This give the claim: $G^\mfc(s)\sim G^\mft(s)$ as $s\da 0$.

(b) We now explore the relationship between $H_2\circ g$ and $G_2$ whenever both $H\in\CSV_0(H_1,H_2)$ as in Part (b-i) and $G\in\CSV_0(G_1,G_2)$ as in Part (b-ii). For simplicity, to reduce confusion in the definition of $g$ and by virtue of Remark~\ref{rem:G2-vs-H2}(a), we assume $G$ is as in~\eqref{eq:G-H}. If $H$ and $G$ are both $C^1$ such that $t\mapsto t|H'(t)|$ and $t\mapsto t|G'(t)|$ are both slowly varying at $0$, then, by~\cite[Lem.~6.1]{WassersteinPaper}, $H_2$ and $G_2$ are larger than some multiples of $\wt H:t\mapsto t|H'(t)|/H(t)$ and $\wt G:t\mapsto t|G'(t)|/G(t)$, respectively. If $g'$ is eventually monotone, the Monotone Density Theorem~\cite[Thm~1.7.2]{MR1015093} gives $t g'(t)\sim \alpha^{-1}g(t)$ as $t\da0$ and hence (since $G(t)=H(g(t))^{1/\alpha}$) 
\[
\frac{t\cdot G'(t)}{G(t)}
=\frac{tg'(t)\cdot H'(g(t))}{\alpha H(g(t))}
\sim \frac{g(t)\cdot H'(g(t))}{\alpha^2 H(g(t))}
\quad\implies\quad
\wt G(t)\sim \alpha^{-2}\wt H(g(t))
\quad\text{as }t\da 0.
\]
This suggests that the bounds in Part (b-ii) are sharper than those of Part (b-i) by a factor of $\wt G^{1/2}$ whenever $\alpha>1$, both assumptions hold, and $G_2$ and $H_2$ are proportional to $\wt G$ and $\wt H$, respectively.
\end{remark}

The class of augmented $\alpha$-stable processes in our context is a broad class of processes as demonstrated by the ensuing examples (see a full discussion of the class in~\cite[Ex.~2.16]{WassersteinPaper}). 

\begin{example}[{Augmented stable L\'evy processes~\cite[Ex.~2.16]{WassersteinPaper}}]
\label{ex:subclass_augmented_stable}
We summarise how various classes of L\'evy processes in the literature relate to the class of augmented stable processes.
\begin{itemize}[leftmargin=1em, nosep]
\item \textbf{Rosi\'nski's tempered stable processes}~\cite[p.~680]{MR2327834}, including the exponentially tempered stable processes, admit the decomposition in~\eqref{eq:temp_stab_alpha} but require $\calQ(\cdot,\bm{v})$ to be completely monotone (i.e., strictly decreasing and convex) with $\calQ(x,\bm{v}) \to 0$ as $x\to\infty$, instead of merely being slowly varying. For instance, a $\beta$-stable process is also a Rosi\'nski tempered $\alpha$-stable process if $\beta>\alpha$, making the index $\alpha$ non-unique~\cite[p.~680]{MR2327834}. Exponentially tempered stable processes are both augmented and Rosi\'nski tempered with $\calQ(x,\bm{v}) = \exp(-\lambda(\bm{v})x)$ for some measurable $\lambda:\Sp^{d-1}\to[0,\infty)$.
\item \textbf{Truncated stable processes} are augmented stable with $\calQ(x,\bm{v}) = \1_{\{x \le c(\bm{v})\}}$ for some measurable $c:\Sp^{d-1}\to(0,\infty)$. These are not tempered stable under Rosi\'nski’s definition.
\item \textbf{Kuznetsov’s $\beta$-class of L\'evy processes}~\cite{MR2724421} extends naturally to $\R^d$ by introducing a common index in all directions, resulting in a subclass of augmented stable processes with $\calQ(x,\bm{v}) = c(\bm{v}) \e^{-\lambda(\bm{v})\beta(\bm{v})x} x^{-\alpha-1}(1 - \e^{-\beta(\bm{v})x})^{-\alpha-1}$, where $c, \beta, \lambda: \Sp^{d-1} \to (0,\infty)$ are measurable functions.
\item \textbf{Meromorphic L\'evy processes} on $\R$~\cite{MR2977987} also fall under the augmented stable class. By~\cite[Cor.~3]{MR2977987}, their L\'evy density takes the form
\[
\pi(x)
=\1_{\{x>0\}}\sum_{n\in\N}a_n\rho_n\e^{-\rho_n x}
+\1_{\{x<0\}}\sum_{n\in\N}\hat a_n\hat\rho_n\e^{\hat\rho_n x},
\quad x\in\R\setminus\{0\},
\]
where the coefficients satisfy specific regularity conditions. If $\rho_n = \hat\rho_n = n$, then by Karamata’s Tauberian theorem~\cite[Cor.~1.7.3]{MR1015093}, the process is augmented $\alpha$-stable if and only if $\sum_{n=1}^{\lfloor x\rfloor} n a_n$ and $\sum_{n=1}^{\lfloor x\rfloor} n \hat a_n$ are both regularly varying at infinity with index $\alpha + 1$. 
\qedhere
\end{itemize}
\end{example}

We will now see some straightforward applications of Corollary~\ref{cor:augmented_stable_rates} to some of these examples.

\begin{example}[Corollary~\ref{cor:augmented_stable_rates} in the case of (classical) exponentially tempered stable processes]\label{ex:exp_temp_stable}
Let $\bm{X}$ and $\bm{Z}$ be as in Corollary~\ref{cor:augmented_stable_rates}(a) with $\calQ(x,\bm{v})=c_\alpha\e^{-\lambda(\bm{v})x}$ for all $x>0$, $\bm{v}\in \Sp^{d-1}$, where $\lambda:\Sp^{d-1}\to[0,\infty)$ is bounded. Thus, $\bm{X}$ is an exponentially tempered stable process and by~\cite[Ex.~2.17]{WassersteinPaper}, it follows $|\e^{-\lambda(\bm{v})x}-1|
\le K(1 \wedge x)$, for all $(x,\bm{v})\in (0,\infty)\times \Sp^{d-1}$, where $K=\sup_{\bm{w}\in\Sp^{d-1}}|\lambda(\bm{w})|$. Hence, by Corollary~\ref{cor:augmented_stable_rates}(a), the family $\{f(t)^{-1}\|\X_t-\ZZ\|_{[0,T]}\}_{t\in(0,1]}$ is tight for all $T>0$, where 
\[
f(t)\coloneqq
\begin{cases}
t^{1-1/\alpha }, &\alpha\in(1,2)\text{ and \eqref{eq:symmetry_thin} fails}, \\
t^{1/\alpha}, &\alpha\in(1,2)\text{ and \eqref{eq:symmetry_thin} holds or $\alpha\in(0,1)$}.
\end{cases}\qedhere
\]
\end{example}

Next, we give an example (inspired by~\cite[Ex.~2.18]{WassersteinPaper}) in DoNNA where the function $G$ is not asymptotically constant, and see how the rates can deteriorate to become slower than a power of $\log$. 

\begin{example}[Augmented $\alpha$-stable processes with arbitrarily slow convergence rate]
\label{ex:arbitrarily_slow} Let $\bm{X}$ and $\bm{Z}$ be as in Corollary~\ref{cor:augmented_stable_rates}(b-ii). Let $\ell_n$ be recursively defined as: $\ell_1(t)\coloneqq \log(e+t)$ and $\ell_{n+1}(t)=\log(e+\ell_n(t))$ for $t>0$. If either $H(x)=\ell_n(1/x)$ or $H(x)=\ell_n(1/x)^{-1}$ for some $n\in\N$ (i.e. $G(x)\sim\ell_n(x^{-1/\alpha})$ or $G(x)\sim\ell_n(x^{-1/\alpha})^{-1}$ as $x\to\infty$), then~\cite[Lem.~6.3]{WassersteinPaper} shows that for small $t>0$ that $$G_2(t)\coloneqq \prod_{k=1}^n(e+\ell_k(1/t))^{-1}\ge |1-G(2t)/G(t)|/\log 2.$$ Moreover, by~\cite[Lem.~6.1]{WassersteinPaper}, by~\cite[Lem.~6.1]{WassersteinPaper}, it follows that $G \in \CSV_0(G_1,G_2)$ for $G_2$ defined in the equation above, and $G_1\equiv \log 2$. Thus, by Corollary~\ref{cor:augmented_stable_rates}(b-ii), the family $\{f(t)^{-1}\|\X_t-\ZZ\|_{[0,T]}\}_{t\in(0,1]}$ is tight, where $g(t)\coloneqq t^{1/\alpha}G(t)$ and $f(t)=G_2(t) =\prod_{k=1}^n(e+\ell_k(1/t))^{-1}$. Thus, even though the function $\ell_n$ is ``nearly constant'' for large $n\in\N$, the convergence rate is slower than $\log(1/t)^{-1-\ve}$ for any $\ve>0$ for $\alpha \in (0,2)\setminus\{1\}$. 
\end{example}

\begin{remark}
Under suitable $L^p$-integrability conditions, the bounds in Theorems~\ref{thm:gen_bound_thin} and~\ref{thm:gen_bound_como} can be leveraged to obtain explicit estimates on the $L^p$-Wasserstein distance between two augmented $\alpha$-stable processes. These integrability conditions are, for instance, satisfied when the processes admit exponential moments. This includes Rosi\'nski's exponentially tempered stable processes, which are characterised by the tempering function $\calQ(x,\bm{v}) = \e^{-\lambda(\bm{v})x}$ for $x > 0$ and $\bm{v} \in \Sp^{d-1}$, where $\lambda: \Sp^{d-1} \to (0,\infty)$ is bounded and strictly positive.

In this context, both Theorems~\ref{thm:gen_bound_thin} and~\ref{thm:gen_bound_como} reveal that the dominant term in the Wasserstein bound arises from the difference $\calQ_1(x,\bm{v}) - \calQ_2(x,\bm{v})$ for $x>0$ and $\bm{v} \in \Sp^{d-1}$. Therefore, deriving sharp bounds on $\mW_p(\Y_1, \Y_2)$ (for $p = 1$ or $2$) effectively reduces to accurately quantifying the discrepancy between the corresponding tempering functions.
\end{remark}

\section{General Bounds}
\label{sec:general_bounds}
In this section, we utilise thinning and comonotonic couplings to derive general $ L^p$-bounds on the distance between solutions to SDEs driven by general L\'evy processes. For a L\'evy process $\bm{Y}_i$ on $\R^d$, let $\Y_i$ be the $\R^m$-valued solution to the L\'evy driven SDE, given by
\begin{equation}
\label{eq:SDE_setting}
\Y_i(t)\coloneqq 
    \bm{x}_i+\int_0^t V(\Y_i(r-))\D \bm{Y}_i(r), 
\quad \text{ for }\bm{x}_i \in \R^m ,
    \text{  }i\in\{1,2\},
\end{equation} and a Lipschitz continuous function $V:\R^m \to \R^{m\times d}$ with Lipschitz constant $K \in (0,\infty)$. Additionally, for convenience, we denote $\V_i\coloneqq V(\Y_i)$ for $i\in\{1,2\}$. Recall, that the existence of a unique solution to these SDEs, is ensured by~\cite[Thm~6.2.3]{Applebaum_2009}. If the process $\bm{Y}_i$ has finite second-moment, a Bellman--Gr\"onwall argument similar to the one we will use ourselves (see, e.g.~\cite[Thm~67, p.~342]{MR2273672})\footnote{Note that only the Lipschitz property (and not the differentiability) is required in the proof of~\cite[Thm~67, p.~342]{MR2273672}.}, implies the existence of a finite function $C_2(\Y_2;\cdot)$ satisfying $\E[\|\Y_2\|^2_{[0,t]}]\le C_2(\Y_2;t)$ for all $t\ge 0$. Similarly, if $\bm{Y}_2$ has finite variation (including $\bm{\Sigma}_2=\bm{0}$) and a finite first moment, there exists a finite function $C_1(\Y_2;\cdot)$ such that $\E[\|\Y_2\|_{[0,t]}] \le C_1(\Y_2;t)$ for all $t\ge 0$. However, since the dependence of the functions $C_j$ on the characteristics of $\bm{Y}_2$ is not made explicit, we derive such a result ourselves in Theorem~\ref{thm:gen_bound_thin} below with an explicit dependence.

Theorems~\ref{thm:gen_bound_thin} and~\ref{thm:gen_bound_como} below both imply bounds on the Wasserstein distance of solutions of SDEs. However, their current phrasing contains more information: they spell out exactly what coupling attains those bounds.

\subsection{Thinning Coupling}\label{sec_thinning_coup}
For $i=1,2$, let $\Xi_i$ be the (Poisson) jump measure of  $\bm{Y}_i$ on $\R_+\times \R^d_{\bm{0}}$ with corresponding compensated measure $\wt\Xi_{i}=\Xi_{i}-\Leb\otimes\nu_{i}$, where $\Leb$ denotes the Lebesgue measure on $\R_+$. Then the L\'evy--It\^o decomposition reads $\bm{Y}_i(t)=\bm{\gamma}_{i} t+\bm{\Sigma}_{i}\bm{B}_{i}(t)+\bm{S}_{i}(t)+\bm{L}_{i}(t)$, where $\bm{B}_i$ is a standard Brownian motion on $\R^m$,
\begin{equation}
\label{eq:thinning_coupling_S_j_L_j}
\bm{S}_i(t)
\coloneqq\int_{(0,t]\times D} \bm{w}\,\wt\Xi_i(\D s,\D\bm{w}),
\qquad
\bm{L}_i(t)
\coloneqq\int_{(0,t]\times D^\co}\bm{w}\,\Xi_i(\D s,\D\bm{w}),
\qquad D^\co\coloneqq \R^d_{\bm0} \setminus D,
\end{equation}
and the processes $\bm{B}_i$, $\bm{S}_i$ and $\bm{L}_i$ are independent. Note that $\bm{S}_i$ is an $L^2$-martingale and $\bm{L}_i$ is a compound Poisson process. Further note that, if both processes have jumps of finite variation, then they both admit the alternative L\'evy--It\^o decomposition $\bm{Y}_i(t)=\bm{b}_{i} t+\bm{\Sigma}_{i}\bm{B}_{i}(t)+\bm{J}_{i}(t)$, $t\ge 0$, where
\begin{equation}
\label{eq:thinning_coupling_FV}
\bm{b}_i
\coloneqq\bm{\gamma}_i-\int_{D}\bm{w}\,\nu_i(\D\bm{w}),
\qquad
\bm{J}_i(t)
\coloneqq\int_{(0,t]\times \R^d_{\bm 0}} \bm{w}\,\Xi_i(\D s,\D\bm{w}),
\quad \text{ for }i\in \{1,2\} \text{ and all }t\ge 0.
\end{equation} Similarly, if both processes have a finite first moment (and possibly infinite variation), the processes also admit the L\'evy--It\^o decomposition $\bm{Y}_i(t)=\bm{a}_i t+\bm\Sigma_i\bm{B}_i(t)+\bm{M}_i(t)$, where
\begin{equation}
\label{eq:thinning_coupling_FM}
\bm{a}_i
\coloneqq\bm{\gamma}_i+\int_{D^\co}\bm{w}\,\nu_i(\D\bm{w}),
\qquad
\bm{M}_i(t)
\coloneqq\int_{(0,t]\times \R^d_{\bm 0}} \bm{w}\,\wt\Xi_i(\D s,\D\bm{w}),
\quad \text{ for }i\in \{1,2\} \text{ and all }t\ge 0.
\end{equation}

Taking $\bm B_1=\bm B_2\eqqcolon \bm B$, the problem of coupling $\bm{Y}_1$ and $\bm{Y}_2$ is reduced to coupling the Poisson random measures $\Xi_1$ and $\Xi_2$ as in~\cite[Sec.~3.1]{WassersteinPaper}. Choose any L\'evy measure $\nu$ on $\R^d_{\bm0}$ that dominates both $\nu_1$ and $\nu_2$ with Radon--Nikodym derivatives bounded by $1$ $\nu$-a.e., i.e.\ $f_1=\D \nu_1/\D\nu\le 1$ and $f_2=\D \nu_2/\D\nu\le 1$ $\nu$-a.e. (for instance, we may take $\nu=\nu_1+\nu_2$). Let $\Xi=\sum_{n\in\N} \delta_{(U_n,\bm{\upsilon}_n)}$ be a Poisson random measure on $\R_+\times\R^d_{\bm 0}$, with mean measure $\Leb\otimes\nu$ and corresponding compensated measure $\wt\Xi=\Xi-\Leb\otimes\nu$. Consider an independent sequence $(\vartheta_n)_{n\in\N}$ of iid uniform random variables on $[0,1]$. By the Marking and Mapping Theorems~\cite{MR1207584}, we may couple $\Xi_1$ and $\Xi_2$ as thinnings of $\Xi$ as follows:
\begin{equation}\label{eq:thining_coupling_defn}
\Xi_1
=\sum_{n\in\N} \1_{\{\vartheta_n\le f_1(\bm{\upsilon}_n)\}}\delta_{(U_n,\bm{\upsilon}_n)},
\qquad\text{and}\qquad
\Xi_2
=\sum_{n\in\N} \1_{\{\vartheta_n\le f_2(\bm{\upsilon}_n)\}}\delta_{(U_n,\bm{\upsilon}_n)}.
\end{equation}

Recall that the BDGN inequality~\cite[Thm~1.1]{MR3463679} states that, for universal constants $0<c_0<C_0<\infty$ and any c\`{a}dl\`{a}g local martingale $\bm{M}$ on a separable Hilbert space (such as $\R^m$) with $\bm{M}_0=\bm{0}$, 
\begin{equation}
\label{eq:BDGN}
c_0\E\big[[\bm{M}]_t\big]
\le \E\bigg[\sup_{s\in[0,t]}|\bm{M}_s|^2\bigg]
\le C_0\E\big[[\bm{M}]_t\big],
\quad t>0,
\end{equation}
where $|\bm{M}|$ denotes the norm of $\bm{M}$ and $[\bm{M}]$ denotes its scalar quadratic variation, that is, the unique non-decreasing c\`{a}dl\`{a}g adapted process for which $|\bm{M}|^2-[\bm{M}]$ is a local martingale.

Given two vectors $\bm{v}_1,\bm{v}_2$ on any vector space and a measure $\mu$ on $\R^d_{\bm{0}}$, we denote $\Delta\bm{v}\coloneqq\bm{v}_1-\bm{v}_2$ and
\[
\mu(\varphi;p)
\coloneqq \int_{\R^d_{\bm 0}}|\bm{w}|^p|\varphi(\bm{w})|\mu(\D\bm{w}),
\qquad\text{for any }p\ge 0,
\enskip\text{ and measurable }\varphi.
\]
For the tuples $(\bm{x}_i,\bm{a}_i,\bm{\Sigma}_i,f_i)$, $i\in\{1,2\}$, and dominating measure $\nu$ as above, define $(\kappa_2^\mft,\eta_2)$ via
\begin{equation}
\label{eq:kappa_eta2}
\begin{split}
\kappa_2^\mft(t)
& \coloneqq 
6|\Delta\bm{x}|^2
+6tC_2(\V_2;t)\big[|\Delta\bm{a}|^2t
+2C_0\big(
|\Delta\bm\Sigma|^2
+\nu(\Delta f;2)\big)\big],\\ 
\eta_2(t)
&\coloneqq 6|\bm{a}_1|^2 K^2 t^2
+6C_0K^2\big(2|\bm{\Sigma}_1|^2
+\nu_1(1;2)
\big)t.
\end{split}
\end{equation}
When both processes are of finite variation (including $\bm{\Sigma}_i=\bm{0}$) with finite mean, we further define
\begin{equation}
\label{eq:kappa_eta1}
\kappa_1^\mft(t)
\coloneqq 
|\Delta\bm{x}|
+ t C_1(\V_2;t)\big(|\Delta\bm{b}|
+\nu(\Delta f;1)\big),
\qquad
\eta_1(t)
\coloneqq 
K\big(|\bm{b}_1|
+\nu_1(1;1)
\big)t.
\end{equation}

\begin{theorem}
\label{thm:gen_bound_thin}
Let $\Y_1$ and $\Y_2$ be as in~\eqref{eq:SDE_setting}, where the drivers $\bm{Y}_1$ and $\bm{Y}_2$ follow the thinning coupling and recall $(\kappa_1^\mft,\eta_1)$ and $(\kappa_2^\mft,\eta_2)$ from~\eqref{eq:kappa_eta2} \&~\eqref{eq:kappa_eta1}.\\
{\nf{(a)}} Suppose $\nu_2(1;2)=\int_{\R^d_{\bm 0}}|\bm{w}|^2\nu_2(\D\bm{w})<\infty$, then $\E[\|\Y_2\|^2_{[0,t]}]\le C_2(\Y_2;t)$ and $\E[\|\V_2\|^2_{[0,t]}]\le C_2(\V_2;t)$ for all $t\ge 0$, where $C_2(\V_2;t)\coloneqq 2(|V(\bm{x}_2)|^2+KC_2(\Y_2;t))$ and 
\begin{equation}
\label{eq:C_2(Y_2)}
C_2(\Y_2;t)
\coloneqq
6t|V(\bm{x}_2)|^2\big[|\bm{a}_2|^2t
    +2C_0\big(
    |\bm\Sigma_2|^2
    +\nu_2(1;2)\big)\big]\e^{6|\bm{a}_2|^2t^2
+6C_0K^2(2|\bm{\Sigma}_2|^2
+\nu_2(1;2)) t}.
\end{equation}
If, moreover, $\nu(f_1+f_2;2)=\int_{\R^d_{\bm 0}}|\bm{w}|^2(\nu_1+\nu_2)(\D\bm{w})<\infty$, then
\begin{equation}
\label{eq:ER_T_bound_with_GB}
\E\big[\|\Delta\Y\|_{[0,t]}^2\big] 
\le \kappa_2^\mft(t)\e^{\eta_2(t)},
\quad \text{for any }t\ge 0.
\end{equation}
{\nf{(b)}} 
Suppose $\bm{\Sigma}_2=\bm{0}$ and $\nu_2(1;1)=\int_{\R^d_{\bm 0}}|\bm{w}|\nu_2(\D\bm{w})<\infty$, then we have $\E\|\Y_2\|_{[0,t]}\le C_1(\Y_2;t)$ and $\E\|\V_2\|_{[0,t]}\le C_1(\V_2;t)$ for all $t\ge 0$, where $C_1(\V_2;t)\coloneqq |V(\bm{x}_2)|+KC_1(\Y_2;t)$ and 
\begin{equation}
\label{eq:C_1(Y_2)}
C_1(\Y_2;t)
\coloneqq 
t|V(\bm{x}_2)| \big(|\bm{b}_2|
    +\nu_2(1;1)\big)\e^{K(|\bm{b}_2|
+\nu_2(1;1))t}.
\end{equation}
If, moreover, $\bm{\Sigma}_1=\bm{\Sigma}_2=\bm{0}$ and $\nu(f_1+f_2;1)=\int_{\R^d_{\bm 0}}|\bm{w}|(\nu_1+\nu_2)(\D\bm{w})<\infty$, then
\begin{equation}
\label{eq:ER_T_bound_with_GB-FV}
 \E\big[\|\Delta\Y\|_{[0,t]}\big] 
 \le \kappa_1^\mft(t)\e^{\eta_1(t)},
 \quad\text{for any }t\ge 0.
\end{equation}
\end{theorem}

\begin{proof}[Proof of Theorem~\ref{thm:gen_bound_thin}(a)]
We will establish~\eqref{eq:ER_T_bound_with_GB} under the assumption that $\E[\|\Y_2\|^2_{[0,t]}]\le C_2(\Y;t)$ for all $t\ge 0$ and some finite function $C_2(\Y_2;\cdot)$, which, by the Lipschitz property of $V$, then implies $\E[\|\V_2\|^2_{[0,t]}]\le 2(|V(\bm{x}_2)|^2+KC_2(\Y_2;t))\eqqcolon C_{2}(\V_2;t)$. Then, the formula~\eqref{eq:C_2(Y_2)} for $C_2(\Y_2;\cdot)$ will follow from this case, by replacing the pair of processes $(\Y_1,\Y_2)$ with $(\Y_2,\bm{x}_2)$ (the assumed functions corresponding to $C_2(\Y;\cdot)$ and $C_2(\V_2;\cdot)$ being identically equal to $|\bm{x}_2|^2$ and $|V(\bm{x}_2)|^2$, respectively). 

The proof is split into 3 steps. In Step 1, we bound $R_2(t)\coloneqq \E[\|\Delta\Y\|_{[0,t]}^2]$ in terms of suprema of integrals with respect to the drift and jump martingale and Brownian components. In Step 2, we control the expectations of those terms. In Step 3, we combine the bounds in Steps 1 and 2 to establish a self-referential. Then, Gr\"onwall's inequality gives~\eqref{eq:ER_T_bound_with_GB}. Throughout, we consider $t\le T$.

\underline{\textbf{Step 1.}} Recall the quadratic to arithmetic means (QM-AM) inequality: $|\sum_{i=1}^k x_i/k|^2 \le \sum_{i=1}^k x_i^2/k$ for any $x_1,\ldots,x_k \in \R$, which follows by convexity and Jensen's inequality. Then, using the L\'evy--It\^o decomposition $\bm{Y}_i(t)=\bm{a}_i t+\bm\Sigma_i\bm{B}(t)+\bm{M}_i(t)$, we get
\begin{gather}
R_2(t)
=\E\big[\|\Delta\Y\|_{[0,t]}^2\big]
=\E\big[\|\Y_1-\Y_2\|_{[0,t]}^2\big]
\le 6
\E\big[|\Delta\bm{x}|^2+\xi_1+\xi_2+\xi_3+\xi_4+\xi_5
\big],
\quad\text{where}\label{eq:R(T)_gen_bound}
\\ \nonumber
\xi_1 \coloneqq \bigg\|\int_0^\cdot \Delta\V(s-)\,\D \bm{M}_1(s)\bigg\|_{[0,t]}^2,
\quad\xi_2\coloneqq \bigg\|\int_0^\cdot \V_2(s-)\big(\D \bm{M}_1(s)-\D \bm{M}_2(s)\big)\bigg\|_{[0,t]}^2, 
\\ \nonumber
\xi_3 \coloneqq \bigg\|\int_0^\cdot \Delta\V(s-)\bm{a}_1\,\D s\bigg\|_{[0,t]}^2, 
\enskip
\xi_4 \coloneqq \bigg\|\int_0^\cdot \V_2(s-)\Delta\bm{a}\,\D s\bigg\|_{[0,t]}^2,
\enskip
\xi_5\coloneqq
\bigg\|\int_0^\cdot \Delta(\V(s-)\bm{\Sigma})\,\D\bm{B}(s)\bigg\|_{[0,t]}^2.
\end{gather}

\underline{\textbf{Step 2.}}
The goal of this step is to bound each $\E[\xi_j]$ for $j\in\{1,\ldots,5\}$.

{\underline{Drift terms.}} By sub-multiplicativity $|\bm{A} \bm{w}| \le |\bm{A}||\bm{w}|$, for $\bm{A} \in \R^{m \times d}$ and $\bm{w} \in \R^d$, where $|\bm{A}|$ is the Frobenius norm of $\bm{A}$, we obtain
\[
\E[\xi_{3}+\xi_{4}]
\le \E\bigg[ 
\bigg(|\bm{a}_1|\int_0^t |\Delta\V(s-)|\,\D s\bigg)^2
+\bigg(|\Delta\bm{a}|\int_0^t|\V_2(s-)|\,\D s\bigg)^2\bigg].
\]
Next, H\"older's inequality, Tonelli's theorem and the Lipschitz property of $V$ imply that
\begin{align}
\E[\xi_{3}+\xi_{4}]
&\le T|\bm{a}_1|^2\int_0^t \E\big[|\Delta\V(s-)|^2\big]\,\D s
+T|\Delta\bm{a}|^2\int_0^t\E\big[|\V_2(t-)|^2\big]\,\D s
\nonumber\\
\label{eq:drift_bound_thinning_p<2}
&\le 
T|\bm{a}_1|^2K^2 \int_0^t R_2(s)\,\D s
+ T^2C_2(\V_2;T)|\Delta\bm{a}|^2.
\end{align}

{\underline{Pure-jump martingale terms.}} 
Using~\eqref{eq:thinning_coupling_FM}, we see that
\begin{equation*}
\xi_{1}+\xi_{2}
= 
\bigg\|\int_{[0,\cdot]\times \R^d_{\bm0}}\!\!\Delta\V(s-)\bm{w}\,\wt{\Xi}_1(\D s, \D \bm{w}) \bigg\|_{[0,t]}^2
+
\bigg\|\int_{[0,\cdot]\times \R^d_{\bm0}}\!\!\V_2(s-)\bm{w}\, \big(\wt{\Xi}_1(\D s, \D \bm{w})-\wt{\Xi}_2(\D s, \D \bm{w}) \big)\bigg\|_{[0,t]}^2 .
\end{equation*} 
For any real-valued function $f$, denote $f^+:=\max\{0,f\}$. With $\Xi_1$ and $\Xi_2$ as in~\eqref{eq:thining_coupling_defn}, define the Poisson random measures 
\begin{equation*}
\Lambda_+\coloneqq\sum_{n\in\N} \1_{\{f_{2}(\bm{\upsilon}_n)<\vartheta_n\le f_{1}(\bm{\upsilon}_n)\}}\delta_{(U_n,\bm{\upsilon}_n)}
\qquad\text{and}\qquad
\Lambda_-\coloneqq\sum_{n\in\N} \1_{\{f_{1}(\bm{\upsilon}_n)<\vartheta_n\le f_{2}(\bm{\upsilon}_n)\}}\delta_{(U_n,\bm{\upsilon}_n)},
\end{equation*}
with mean measures $\Leb\otimes(f_{1}-f_{2})^+\nu$ and $\Leb\otimes(f_{2}-f_{1})^+\nu$, respectively. Hence, $\Xi_{1}-\Xi_{2}=\Lambda_+-\Lambda_-$. The processes $\Lambda_+$ and  $\Lambda_-$ are independent since they are thinnings of the same Poisson random measure with disjoint supports. Let $\wt\Lambda_+$ and $\wt\Lambda_-$ denote their respective compensated measures and define the L\'evy processes $\bm{M}^\pm=(\bm{M}^\pm(t))_{t\ge0}$ via $\bm{M}^\pm(t):=\int_{(0,t]\times D}\bm{w}\,\wt\Lambda_\pm(\D s,\D\bm{w})$, where $\pm\in\{+,-\}$. The independent square-integrable martingales $\bm{M}^+$ and $\bm{M}^-$ satisfy $\bm{M}_1-\bm{M}_2=\bm{M}^+-\bm{M}^-$ pathwise, so the BDGN inequality~\cite[Thm~1.1]{MR3463679} (see also~\cite[Thm~1(b)]{MR394861}) gives
\begin{align}
\E[\xi_{2}]
& =
\E \Bigg[\bigg\|\int_0^\cdot \V_2(s-) \,\D(\bm{M}^+(s)-\bm{M}^-(s))\bigg\|_{[0,t]}^2\Bigg]\nonumber\\
&\le
2\E \Bigg[\bigg\|\int_0^\cdot \V_2(s-) \,\D\bm{M}^+(s)\bigg\|_{[0,T]}^2
+\bigg\|\int_0^\cdot \V_2(s-) \,\D\bm{M}^-(s)\bigg\|_{[0,T]}^2\Bigg]\nonumber\\
& \le 2C_0\int_{[0,T]\times \R^d_{\bm0}} \E\big[|\V_2(s-)|^2\big] |\bm{w}|^2|\Delta f(\bm{w})| \,\D s \otimes \nu(\D\bm{w}) 
\le 2C_0 T C_2(\V_2;T) \nu(\Delta f;2).
\label{eq:bound_small_jumps_thinning_1_p<2}
\end{align} 

Finally, by the Lipschitz property of $V$, the BDGN inequality and Tonelli's theorem, we see that
\begin{align}
\E[\xi_{1}]
&= 
\E\Bigg[\bigg\|\int_{[0,\cdot]\times \R^d_{\bm0}}\!\!\Delta\V(s-)\bm{w}\,\wt{\Xi}_1(\D s, \D \bm{w}) \bigg\|_{[0,t]}^2 \Bigg]
\nonumber\\
&\le 
C_0 \int_{[0,t]\times \R^d_{\bm0}} \E\big[|\Delta\V(s-)|^2\big] |\bm{w}|^2f_1(\bm{w})\,\D s \otimes \nu(\D\bm{w})
\le C_0K^2 \nu_1(1;2)\int_0^t R_2(s)\,\D s.
\label{eq:bound_small_jumps_thinning_2_p<2}
\end{align}

{\underline{Brownian term.}} 
The Lipschitz property of $V$, Jensen's and the BDGN~\cite[Thm~1.1]{MR3463679} inequalities, Tonelli's theorem, and the fact that the Frobenius norm dominates the operator norm imply that
\begin{align}
\E[\xi_5]
&\le C_0\E\bigg[\int_0^t|\Delta(\V(s-)\bm{\Sigma})|^2\,\D s \bigg]
\le 2C_0\E\bigg[\int_0^t\big(|\V_2(t-)\Delta\bm{\Sigma}|^2
    +|\Delta\V(s-)\bm{\Sigma}_1|^2\big)\,\D s\bigg]\nonumber\\
&\le 2C_0\bigg(TC_2(\V_2;T)|\Delta\bm\Sigma|^2 
+ K^2|\bm\Sigma_1|^2 \int_0^t R_2(s)
\,\D s\bigg).\label{eq:xi_5}
\end{align}

\underline{\textbf{Step 3.}} Assembling the bounds in~\eqref{eq:R(T)_gen_bound}--\eqref{eq:xi_5} leads to
\begin{equation*}
R_2(t)
\le \kappa_2^\mft(T) + \frac{\eta_2(T)}{T}\int_0^t R_2(s)\,\D s,
\quad t\le T.
\end{equation*}
Gr\"onwall's inequality gives $R_2(t) \le \kappa_2^\mft(T)\e^{\eta_2(T)t/T}$ for $t\in[0,T]$, so taking $t=T$ yields~\eqref{eq:ER_T_bound_with_GB}.
\end{proof}

\begin{proof}[Proof of Theorem~\ref{thm:gen_bound_thin}(b)]
In spirit, we follow the same steps of the proof of Part (a), but without applying the QM--AM or maximal inequalities. Again, the formula for $C_1(\Y_2;\cdot)$ will follow from the general case. Define $R_1(t)\coloneqq \E[\|\Delta\Y\|_{[0,t]}]$ for $t\ge 0$ and consider throughout $t\le T$. 

\underline{\textbf{Step 1.}} 
Let $\bm{J}_i$ and $\bm{b}_i$ be as in~\eqref{eq:thinning_coupling_FV}. Thus, using the triangle inequality, we obtain
\begin{equation}
R_1(t)=\E\|\Delta\Y\|_{[0,t]}=\E\|\Y_1-\Y_2\|_{[0,t]}
\le |\Delta\bm{x}|+\E[\xi_1+\xi_2+\xi_3],
\quad\text{where}\label{eq:R(T)_gen_bound-FV}
\end{equation}
\vspace{-6mm}
\begin{align*}
&\xi_1\coloneqq
\bigg\|\int_0^\cdot \Delta\V(t-)\bm{b}_1\,\D t\bigg\|_{[0,t]},  
&&\xi_2\coloneqq
\bigg\|\int_0^\cdot \V_2(t-)\Delta\bm{b}\,\D t\bigg\|_{[0,t]}\\
&\xi_3 \coloneqq \bigg\|\int_0^\cdot \Delta\V(s-)\,\D\bm{J}_1(s)\bigg\|_{[0,t]},
&&\xi_4\coloneqq \bigg\|\int_0^\cdot \V_2(s-)(\D\bm{J}_1(s)-\D\bm{J}_2(s))\bigg\|_{[0,t]}.
\end{align*} 

\underline{\textbf{Step 2.}}
The goal of this step is to bound each $\E[\xi_j]$ for $j\in\{1,\ldots,4\}$. 

{\underline{Drift terms.}} Using that $|\bm{A} \bm{w}| \le |\bm{A}||\bm{w}|$, where $\bm{A} \in \R^{d \times d}$, $\bm{w} \in \R^d$ and $|\bm{A}|$ is the Frobenius norm, Tonelli's theorem and the Lipschitz property of $V$, we obtain
\begin{align}
\label{eq:xi1-FV}
\E[\xi_{1}]
&\le \E\bigg[|\bm{b}_1|\int_0^t|\Delta\V(s-)|\,\D s \bigg] 
\le K|\bm{b}_1|\int_0^t R_1(s)\,\D s, \quad \text{ and }\\
\label{eq:xi2-FV}
\E[\xi_{2}]
&\le \E\bigg[|\Delta\bm{b}|\int_0^t|\V_2(s-)|\,\D s \bigg] 
\le T C_1(\V_2;T)|\Delta\bm{b}|.
\end{align} 

{\underline{Bounding $\xi_3$.}} 
The triangle inequality, the Lipschitz property of $V$, Tonelli's theorem and the compensation formula~\cite[Thm~10.21]{MR4226142} give
\begin{equation}
\label{eq:xi_3-middle-step-FV}
\E[\xi_3] 
\le\E\bigg[\int_{[0,t]\times\R^d_{\bm0}}|\Delta\V(s-)| \!\cdot\! |\bm{w}|\,\Xi_1(\D s,\D\bm{w})\bigg]
\le K\nu_1(1;1)\int_0^t R_1(s) \,\D s.
\end{equation} 

{\underline{Bounding $\xi_4$.}} 
Let $\Lambda_+$ and $\Lambda_-$ be as in the proof of Part~(a). Using~\eqref{eq:thinning_coupling_S_j_L_j}, the triangle inequality, the fact that Poisson integrals are infinite sums, the compensation formula~\cite[Thm~10.21]{MR4226142}, and Tonelli's theorem, we see that
\begin{equation}
\label{eq:xi4-FV}
\E[\xi_{4}]
\le\E\bigg[\int_{[0,t]\times\R^d_{\bm0}}|\V_2(s-)| |\bm{w}|(\Lambda_+(\D s,\D\bm{w})+\Lambda_-(\D t,\D\bm{w}))\bigg]
\le TC_1(\V_2;T) \nu(\Delta f;1).
\end{equation} 

\underline{\textbf{Step 3.}} Assembling the bounds in~\eqref{eq:R(T)_gen_bound-FV}--\eqref{eq:xi4-FV} yields
\begin{equation*}
R_1(t)
\le \kappa_1^\mft(T) + \frac{\eta_1}{T} \int_0^t R_1(s)\,\D s,
\quad t\le T.
\end{equation*}
Again, an application of Gr\"onwall's inequality for fixed $T\ge 0$ and $t\in[0,T]$ yields~\eqref{eq:ER_T_bound_with_GB-FV}.
\end{proof}

\subsection{Comonotonic Coupling}\label{sec:comonot_coup}
Recall from~\cite[Sec.~3.2]{WassersteinPaper}, the definition of the comonotonic coupling, which decomposes the jumps of a L\'evy process into its magnitude (i.e. norm) and angular component. The main idea behind the coupling of the L\'evy processes $\bm{Y}_1$ and $\bm{Y}_2$ is to couple their respective Poisson random measures of jumps via a comonotonic coupling of the magnitudes of jumps, while simultaneously aligning their angular components. 

We now describe this construction. Recall that the L\'evy processes $\bm{Y}_1$ and $\bm{Y}_2$ in $\R^d$ have characteristic triplets $(\bm{\gamma}_{i},\bm{\Sigma}_{i}\bm{\Sigma}_{i}^\tra,\nu_{i})$ for $i=1,2$. For $i=1,2$, suppose the L\'evy measure $\nu_{i}$ of $\bm{Y}_i$ admits a radial decomposition (see~\cite[p.~282]{MR647969}): there exists a probability measure $\sigma_{i}$ on the unit sphere $\Sp^{d-1}$ (with convention $\Sp^{0}\coloneqq\{-1,1\}$) such that:
\[
\nu_{i}(B)= \int_{\Sp^{d-1}}\int_0^\infty \1_{B}(x\bm{v})\rho_{i}^0(\D x,\bm{v})\sigma_{i}(\D \bm{v}), \quad \text{ for any } B\in \mathcal{B}(\R^d \setminus \{\bm{0}\}),
\] where $\{\rho^0_{i}(\cdot,\bm{v})\}_{\bm{v}\in \Sp^{d-1}}$ is a measurable family of L\'evy measures on $(0,\infty)$. Define the probability measure $\sigma \coloneqq (\sigma_{1}+\sigma_{2})/2$ on $\Sp^{d-1}$ and the Radon-Nikodym derivatives $f^\sigma_{i}(\bm{v})\coloneqq \sigma_{i}(\D\bm{v})/\sigma(\D\bm{v})\le 2$ for $\bm{v} \in \Sp^{d-1}$ and $i=1,2$. Consider the following radial decompositions of $\nu_{i}$: 
\begin{equation}\label{eq:radial_decomp}
\nu_{i}(B)= \int_{\Sp^{d-1}}\int_0^\infty \1_{B}(x\bm{v})\rho_{i}(\D x,\bm{v})\sigma(\D \bm{v}), \quad \text{ for }B\in \mathcal{B}(\R^d \setminus \{\bm{0}\}), 
\end{equation} where $\rho_{i}(\cdot,\bm{v})\coloneqq f^\sigma_{i}(\bm{v})\rho_{i}^0(\cdot,\bm{v})$ for $\bm{v} \in \Sp^{d-1}$ and $i=1,2$. The advantage of the decomposition in~\eqref{eq:radial_decomp}, compared to the one in the display above, is that the angular components of jumps are sampled from the same measure $\sigma$ on $\Sp^{d-1}$, making it possible to couple the jumps of $\bm{Y}_1$ and $\bm{Y}_2$ by only coupling their magnitudes and matching their angles. 

For every $\bm{v}\in \Sp^{d-1}$ and $i=1,2$, let $u\mapsto\rho_{i}^{\la}(u,\bm{v})$ be the generalised right-continuous inverse of $x\mapsto\rho_{i}([x,\infty),\bm{v})$. 
Let $\Xi$ be a Poisson random measure on $\R_+\times\R_+\times \Sp^{d-1}$ with measure $\Leb\otimes\Leb\otimes\sigma$ and compensated measure $\wt\Xi$, given by:
\begin{equation}\label{eq:marked_PPP_cmnotonic}
    \Xi\coloneqq \sum_{n \in \N}\delta_{(U_n,\Gamma_n,\bm{V}_n)}, \qquad \wt\Xi(\D t,\D x,\D \bm{v})=\Xi(\D t, \D x, \D \bm{v})-\D t \otimes \D x \otimes \sigma(\D \bm{v}).
\end{equation} 

Next, we note that (by~\cite[Prop.~3.3]{WassersteinPaper})
, the small-jump components of $\bm{Y}_1$ and $\bm{Y}_2$ take the form
\begin{equation}\label{eq:comono_coupling_1}    \wh{\bm{S}}_i(t)\!\coloneqq\!\!\int_{[0,t]\times [1,\infty)\times \Sp^{d-1}}
\!\!\bm{v}\rho_{i}^{\la}(x,\bm{v})\wt \Xi(\D t, \D x, \D \bm{v}), \quad \text{ for }i \in \{1,2\} \text{ and all }t\ge 0. 
\end{equation}
The big-jump components of $\bm{Y}_1$ and $\bm{Y}_2$ can similarly be expressed as
\begin{equation}\label{eq:comono_coupling_2}
\wh{\bm{L}}_i(t)\coloneqq \int_{[0,t]\times (0,1)\times \Sp^{d-1}}\bm{v}\rho_{i}^{\la}(x,\bm{v}) \Xi(\D t, \D x, \D \bm{v}), \quad \text{ for }i \in \{1,2\} \text{ and all }t\ge 0. 
\end{equation} 
By~\cite[Prop.~3.3]{WassersteinPaper}, there exist constants $\bm{\varpi}_{i}\in\R^d$ and independent of standard Brownian motions $\bm{B}_i$ on $\R^d$, such that $\bm{Y}_i(t)= \bm{\varpi}_{i}t+\bm{\Sigma}_{i}\bm{B}_i(t)+\wh{\bm{S}}_{i}(t)+\wh{\bm{L}}_{i}(t)$, for all $t\ge 0$ and $i=1,2$, where
\begin{equation*}
    \bm{\varpi}_i \coloneqq \bm{\gamma}_i - \int_{\Sp^{d-1}} \bm{v} \int_0^\infty \big(x\1_{(0,1)}(x)-\rho_i^{\la}(\rho_i([x,\infty),\bm{v}),\bm{v})\1_{\{\rho_i([x,\infty),\bm{v})\ge 1\}}(x)\big) \rho_i(\D x,\bm{v})  \sigma(\D \bm{v}).
\end{equation*} 
 
If $\bm{Y}_1$ and $\bm{Y}_2$ both have either jumps of finite variation or a finite first moment, they admit alternative decompositions (which agree with the ones introduced above). We now derive them in this context while preserving the notation used therein. If $\bm{Y}_1$ and $\bm{Y}_2$ both have jumps of finite variation, then they both admit the alternative L\'evy--It\^o decomposition $\bm{Y}_i(t)=\bm{b}_{i} t+\bm{\Sigma}_{i}\bm{B}_{i}(t)+\bm{J}_{i}(t)$, $t\ge 0$, where
\begin{equation}
\label{eq:thinning_coupling_FV_CM}
\bm{b}_i
\coloneqq\bm{\varpi}_i-\int_{[1,\infty)\times \Sp^{d-1}}
\!\!\bm{v}\rho_{i}^{\la}(x,\bm{v})\,\D x\, \sigma(\D \bm{v}),
\quad
\bm{J}_i(t)
\coloneqq \int_{[0,t]\times (0,\infty)\times \Sp^{d-1}}\bm{v}\rho_{i}^{\la}(x,\bm{v}) \Xi(\D s, \D x, \D \bm{v}).
\end{equation}
If $\bm{Y}_1$ and $\bm{Y}_2$ both have a finite first moment (and possibly infinite variation), the processes also admit the L\'evy--It\^o decomposition $\bm{Y}_i(t)=\bm{a}_i t + \bm\Sigma_i\bm{B}_i(t) + \bm{M}_i(t)$, $t \ge 0$, where
\begin{equation}
\label{eq:thinning_coupling_FM_CM}
\bm{a}_i
\coloneqq\bm{\varpi}_i+\int_{(0,1)\times \Sp^{d-1}}
\bm{v}\rho_{i}^{\la}(x,\bm{v})\,\D x \,\sigma(\D \bm{v}),
\quad
\bm{M}_i(t)
\coloneqq \int_{[0,t]\times (0,\infty)\times \Sp^{d-1}}
\!\!\bm{v}\rho_{i}^{\la}(x,\bm{v})\wt \Xi(\D s, \D x, \D \bm{v}).
\end{equation}

Throughout, we will choose $\bm{B}\coloneqq \bm{B}^{\bm{X}}=\bm{B}^{\bm{Y}}$, for a standard Brownian motion $\bm{B}$. Let $C_2(\Y_2;\cdot)$ and $C_1(\Y_2;\cdot)$ (as well as $C_2(\V_2;\cdot)$ and $C_1(\V_2;\cdot)$) be as in~\eqref{eq:C_2(Y_2)} and~\eqref{eq:C_1(Y_2)} and $\eta_2$ and $\eta_1$ be as in~\eqref{eq:kappa_eta2} and~\eqref{eq:kappa_eta1}. Further let 
\[
\mu(A)
\coloneqq\int_{\Sp^{d-1}}\int_0^\infty \1_A(x\bm{v})\,\D x\,\sigma(\D\bm{v})
\quad\text{and}\quad
\mu(\varphi)
\coloneqq\int_{\Sp^{d-1}}\int_0^\infty \varphi(x,\bm{v})\,\D x\,\sigma(\D\bm{v}),
\]
for any $A\in\Borel(\R^d_{\bm 0})$, measurable $\varphi$ and $p\ge 0$. Moreover, define
\begin{align}\label{eq:defn_kappa_2_cm}
\kappa^\mfc_2(t)
&\coloneqq 6|\Delta\bm{x}|^2
+ 6 t C_2(\V_2;t) \big[t|\Delta\bm{a}|^2 
+ 2 C_0 |\Delta\bm\Sigma|^2
+ C_0 \mu\big(|\Delta\rho^\la|^2\big)\big],\\
\label{eq:defn_kappa_1_cm}
\kappa^\mfc_1(t)
&\coloneqq |\Delta\bm{x}|
+ t C_1(\V_2;t) \big(|\Delta\bm{b}| 
+ C_0\mu\big(|\Delta\rho^{\la}|\big)\big).
\end{align}

\begin{theorem}\label{thm:gen_bound_como}
Suppose the L\'evy measures of $\bm{Y}_1$ and $\bm{Y}_2$ admit a radial decomposition and let the processes follow the comonotonic coupling. Let $\Y_1$ and $\Y_2$ be the solutions to the SDEs in~\eqref{eq:SDE_setting} and recall $\eta_2,\eta_1,\kappa_2^\mfc$ \& $\kappa_1^\mfc$ from~\eqref{eq:kappa_eta2},~\eqref{eq:kappa_eta1},~\eqref{eq:defn_kappa_2_cm} \&~\eqref{eq:defn_kappa_1_cm}.\\
{\nf{(a)}} Suppose $\mu(|\rho^\la_1|^2+|\rho^\la_2|^2) = \int_{\R^d_{\bm 0}}|\bm{w}|^2(\nu_1+\nu_2)(\D\bm{w})<\infty$, then
\begin{equation*}
\E\big[\|\Delta\Y\|_{[0,t]}^2\big] 
\le \kappa_2^\mfc(t)\e^{\eta_2(t)},
\quad \text{for all }t\ge 0.
\end{equation*}
{\nf{(b)}} Suppose $\bm{\Sigma}_1=\bm{\Sigma}_2=\bm{0}$, $\mu(\rho^\la_1+\rho^\la_2) = \int_{\R^d_{\bm 0}}|\bm{w}|(\nu_1+\nu_2)(\D\bm{w})<\infty$, then
\begin{equation*}
 \E\big[\|\Delta\Y\|_{[0,t]}\big] 
 \le \kappa_1^\mfc(t)\e^{\eta_1(t)},
 \quad \text{for all }t\ge 0.
\end{equation*}
\end{theorem}

\begin{proof}[Proof of Theorem~\ref{thm:gen_bound_como}(a)]
The proof follows the same structure as that of Theorem~\ref{thm:gen_bound_thin}(a). Let $R_2(t)\coloneqq\E[\|\Delta\Y\|^2_{[0,t]}]$ and consider throughout $t\le T$.

\underline{\textbf{Step 1.}} Recall the QM-AM inequality: $|\sum_{i=1}^k x_i/k|^2 \le \sum_{i=1}^k x_i^2/k$ for any $x_1,\ldots,x_k \in \R$. Then, since $\int_{D^\co}|\bm{w}|\nu_i(\D\bm{w})<\infty$, the L\'evy--It\^o decomposition $\bm{Y}_i(t)=\bm{a}_i t+\bm\Sigma_i\bm{B}(t)+\bm{M}_i(t)$ in~\eqref{eq:thinning_coupling_FM_CM} gives
\begin{gather}
R_2(t)=
\E[\|\Y_1-\Y_2\|_{[0,t]}^2]
\le 
6(|\Delta\bm{x}|^2+\E[\xi_1+\xi_2+\xi_3+\xi_4+\xi_5]),
\quad\text{where}\label{eq:R(T)_gen_bound_CM}
\\ \nonumber
\xi_1 \coloneqq \bigg\|\int_0^\cdot \Delta\V(s-)\,\D \bm{M}_1(s)\bigg\|_{[0,t]}^2,
\quad
\xi_2\coloneqq \bigg\|\int_0^\cdot \V_2(s-)\big(\D \bm{M}_1(s)-\D \bm{M}_2(s)\big)\bigg\|_{[0,t]}^2, 
\\ \nonumber
\xi_3 \coloneqq \bigg\|\int_0^\cdot \Delta\V(s-)\bm{a}_1\,\D s\bigg\|_{[0,t]}^2, 
\enskip
\xi_4 \coloneqq \bigg\|\int_0^\cdot \V_2(s-)\Delta\bm{a}\,\D s\bigg\|_{[0,t]}^2,
\enskip
\xi_5\coloneqq
\bigg\|\int_0^\cdot \Delta\big(\V(s-)\bm{\Sigma}\big)\,\D \bm{B}(s)\bigg\|_{[0,t]}^2.
\end{gather}

\underline{\textbf{Step 2.}}
The goal of this step is to bound each $\E[\xi_j]$ for $j\in\{1,\ldots,5\}$.

{\underline{Drift \& Brownian terms.}} The bounds on the drift and Brownian terms follow exactly as in the proof of Theorem~\ref{thm:gen_bound_thin}(a). Indeed, as in~\eqref{eq:drift_bound_thinning_p<2} and~\eqref{eq:xi_5}, we have 
\begin{equation}
\label{eq:drift_bound_thinning_p<2_CM}
\begin{split}
\E[\xi_{3}+\xi_{4}] 
&\le 
T|\bm{a}_1|^2K^2 \int_0^t R_2(s)\,\D s
+ T^2 C_2(\V_2;T) |\Delta\bm{a}|^2, 
\quad \text{ and} \\
\E[\xi_5] 
&\le 2C_0 \bigg(TC_2(\V_2;T)|\Delta\bm\Sigma|^2 
+ K^2|\bm\Sigma_1|^2 \int_0^t R_2(s)\,\D s\bigg).
\end{split}
\end{equation}

{\underline{Pure-jump martingale terms.}} 
Using~\eqref{eq:thinning_coupling_FM_CM}, we see that
\begin{align*}
\xi_{1}+\xi_{2}
&= \bigg\|\int_{[0,\cdot]\times (0,\infty)\times \Sp^{d-1}} \Delta\V(s-)\bm{v}\rho_{1}^{\la}(x,\bm{v}) \,\wt \Xi(\D s, \D x, \D \bm{v}) \bigg\|_{[0,t]}^2\\
&\quad +
\bigg\|\int_{[0,\cdot]\times (0,\infty)\times \Sp^{d-1}} \V_2(s-)\bm{v} \,\Delta\rho^{\la}(x,\bm{v})\,\wt\Xi(\D s, \D x, \D \bm{v})\bigg\|_{[0,t]}^2 .
\end{align*} 
The BGDN inequality~\cite[Thm~1(b)]{MR394861} then gives
\begin{equation}\label{eq_Xi_2_cm}
\E[\xi_2]
\le C_0 TC_2(\V_2;T)\mu\big((\Delta\rho^{\la})^2\big).
\end{equation}
Similarly, by the Lipschitz property of $V$, the BDGN inequality~\cite[Thm~1(b)]{MR394861} and Tonelli's theorem, 
\begin{equation}
\label{eq:bound_small_jumps_thinning_2_p<2_cm}
\E[\xi_{1}]
=\E\Bigg[\bigg\|\int_{[0,\cdot]\times (0,\infty)\times \Sp^{d-1}} \!\!\!\!\!\!\Delta\V(s-)\bm{v}\rho_{1}^{\la}(x,\bm{v})\,\wt \Xi(\D s, \D x, \D \bm{v}) \bigg\|_{[0,t]}^2 \Bigg]
\le C_0 K^2 \mu\big((\rho_1^\la)^2\big)\int_0^t \! R_2(s)\,\D s.
\end{equation}

\underline{\textbf{Step 3.}} Assembling the bounds in~\eqref{eq:R(T)_gen_bound_CM}--\eqref{eq:bound_small_jumps_thinning_2_p<2_cm} yields
\[
R_2(t)
\le \kappa^\mfc_2(T) + \frac{\eta_2(T)}{T}\int_0^t R_2(s)\,\D s,
\quad t\le T.
\]
Applying Gr\"onwall's inequality, the claim follows.
\end{proof}

\begin{proof}[Proof of Theorem~\ref{thm:gen_bound_como}(b)]
Again, we follow the steps of the proof of Theorem~\ref{thm:gen_bound_thin}(b). Define $R_1(t)\coloneqq \E\|\Delta\Y\|_{[0,t]}$ for $t\ge 0$ and consider throughout $t\le T$. 

\underline{\textbf{Step 1.}} 
Let $\bm{J}_i$ and $\bm{b}_i$ be as in~\eqref{eq:thinning_coupling_FV_CM} and denote $\V_i(t)=V(\Y_i(t))$ for $t\ge 0$ and $i=1,2$. Thus, using the triangle inequality, we obtain
\begin{equation}
R_1(t)=\E\|\Delta \Y\|_{[0,t]}
\le |\Delta\bm{x}|+\E[\xi_1+\xi_2+\xi_3],
\quad\text{where}\label{eq:R(T)_gen_bound-FV_cm}
\end{equation}
\vspace{-7mm}
\begin{align*}
&\xi_1\coloneqq
\bigg\|\int_0^\cdot \Delta\V(s-)\bm{b}_1\,\D s\bigg\|_{[0,t]},  
&&\xi_2\coloneqq
\bigg\|\int_0^\cdot \V_2(s-)\Delta\bm{b}\,\D t\bigg\|_{[0,t]}\\
&\xi_3 \coloneqq \bigg\|\int_0^\cdot \Delta\V(s-)\D\bm{J}_1(s)\bigg\|_{[0,t]},
&&\xi_4\coloneqq \bigg\|\int_0^\cdot \V_2(s-)(\D\bm{J}_1(s)-\D\bm{J}_2(s))\bigg\|_{[0,t]}.
\end{align*} 

\underline{\textbf{Step 2.}}
The goal of this step is to bound each $\E[\xi_j]$ for $j\in\{1,\ldots,4\}$. 

{\underline{Drift terms.}} As before, we obtain
\begin{equation}
\label{eq:xi1-FV_cm}
\E[\xi_{1}]
\le K|\bm{b}_1|\int_0^t R_1(s)\,\D s, 
\quad \text{and} \quad 
\E[\xi_{2}]
\le TC_1(\V_2;T)|\Delta\bm{b}|.
\end{equation} 

{\underline{Bounding $\xi_3$.}} 
The triangle inequality, the Lipschitz property of $V$, the compensation formula~\cite[Thm~10.21]{MR4226142}, and Tonelli's theorem, give
\begin{equation}
\label{eq:xi_3-middle-step-FV_cm}
\E[\xi_3] 
\le\E\Bigg[\int_{[0,t]\times (0,\infty) \times \Sp^{d-1}}|\Delta\V(s-)| \rho^{\la}_1(x,\bm{v}) \Xi(\D s,\D x,\D\bm{v})\Bigg]
\le K\mu(\rho_1^\la) \int_0^t R_1(s)\, \D s.
\end{equation} 

{\underline{Bounding $\xi_4$.}} 
By the BDGN inequality~\cite[Thm~1(b)]{MR394861}, 
\begin{align}\label{eq_Xi_2_cm_2}
\E[\xi_4]
&\le C_0 T C_1(\V_2;T) \mu\big(|\Delta\rho^{\la}|\big).
\end{align}

\underline{\textbf{Step 3.}} Assembling in~\eqref{eq:R(T)_gen_bound-FV_cm}--\eqref{eq_Xi_2_cm_2} and applying Gr\"onwall's inequality finally gives the result.
\end{proof}

\subsection{Wasserstein Distance Between Solutions of SDEs with Additive Noise}
\label{subsec:additive-SDE}
Given a L\'evy process $\bm{Y}$ and $\bm{y} \in \R^d$, consider the SDE with additive noise, given by 
\begin{equation}
\label{eq:simple_lin_SDEs_prop}
\Y(s) = \bm{y} + \int_0^s V(\Y(r-)) \, \mathrm{d}r + \bm{Y}(s),
\end{equation}
for Lipschitz continuous function $V:\R^d \mapsto \R^{d}$. Next, we find a Lipschitz principle for the solutions in the $L^q$-Wasserstein distance $\mW_q$ defined in~\eqref{eq:general_wasserstein_trunk} with $\rho(\bm{x},\bm{y})=|\bm{x}-\bm{y}|$ and time horizon $T>0$.

\begin{proposition}
\label{prop:simple_lin_SDE_gen_bound} 
Fix $q \in (0,1]$ and, for $i\in\{1,2\}$, let $\bm{Y}_i$ be a L\'evy process with finite $q$-moment and $\Y_i$ be the solution to the SDE in~\eqref{eq:simple_lin_SDEs_prop} driven by $\bm{Y}_i$ with initial value $\bm{y}_i\in\R^d$. Then, for any $T>0$,
\begin{equation*}
\mW_q(\Y_1,\Y_2) 
\le \big(|\Delta\bm{y}|^q +\mW_{q}(\bm{Y}_1,\bm{Y}_2) \big)
\times\begin{dcases}
\e^{KT}, & q=1,\\
 N(N+1) 2^{N-1}, & q \in (0,1),
\end{dcases}
\end{equation*}
where $N \coloneqq \ceil{TK2^{1/q}} \in \N$.
\end{proposition}

\begin{proof}
Part~1. Suppose $q=1$, consider any coupling $(\bm{Y}_1,\bm{Y}_2)$ on $[0,T]$, let $R(t)\coloneqq \E[\|\Delta\Y\|_{[0,t]}]$ and $\V_i\coloneqq V(\Y_i)$. The triangle inequality, the Lipschitz continuity of $V$ and Tonelli's theorem, give
\begin{align*}
R(t) &\le |\Delta\bm{y}|+ \E\|\Delta\bm{Y}\|_{[0,t]}
+\E\bigg\| \int_0^\cdot \Delta\V(s-)\,\D s \bigg\|_{[0,t]} \\
&\le |\Delta\bm{y}|+ \E\|\Delta\bm{Y}\|_{[0,t]} + K \int_0^t \E[\|\Delta\Y\|_{[0,s]}]\, \D s
\le|\Delta\bm{y}| + \E\|\Delta\bm{Y}\|_{[0,T]} + K \int_0^t R(s)\, \D s,
\end{align*} 
for all $t\le T$. Hence, Gr\"onwall's inequality gives
\[
R(t) \le \big(|\Delta\bm{y}|
+ \E\|\Delta\bm{Y}\|_{[0,T]}\big) \e^{Kt},
\quad t\le T.
\]
Setting $t=T$ and taking infimum over all couplings of $(\bm{Y}_1,\bm{Y}_2)$ on $[0,T]$ gives the first claim.

Part~2. Suppose $q \in (0,1)$. A direct application of Gr\"onwall's appears to be impossible, so we will use an iterative bound, resulting in a discrete version of Gr\"onwall's inequality. Again, consider any coupling of $(\bm{Y}_1,\bm{Y}_2)$ on $[0,T]$. By the sub-additivity of $x \mapsto x^q$ for all $x \ge 0$ and the Lipschitz continuity of $V$, 
\begin{align*}
\E[\|\Delta\Y\|_{[0,t]}^q] 
&\le |\Delta\bm{y}|^q
+\E[\|\Delta\bm{Y}\|_{[0,t]}^q] 
+ \E\Bigg[\bigg\|\int_0^\cdot \Delta\V(s-) \,\D s \bigg\|^q_{[0,t]}\Bigg]\\
&\le |\Delta\bm{y}|^q
+\E[\|\Delta\bm{Y}\|_{[0,t]}^q] 
+ (Kt)^q \E\big[\|\Delta\Y(s-)\|^q_{[0,t]}\big], 
\end{align*} 
for all $t\le T$. Hence, for all $t\in (0,K^{-1})$, it holds that
\begin{equation}\label{eq:wasserstein_first_step}
\E[\|\Delta\Y\|_{[0,t]}^q] 
\le (1-(Kt)^q)^{-1} \big(|\Delta\bm{y}|^q 
+ \E\big[\|\Delta\bm{Y}\|_{[0,t]}^q\big]\big).
\end{equation} 
An extension of~\eqref{eq:wasserstein_first_step} for all $t>0$, will follow by iterating~\eqref{eq:wasserstein_first_step} in this case. 

Fix $t^* \in (0,K^{-1})$, consider~\eqref{eq:wasserstein_first_step} applied to $t^*$ and set $C = C_{q,K,t^*} \coloneqq (1-(Kt^*)^q)^{-1} \in (1,\infty)$. The goal is now to iterate the inequality and show that, for all $N \ge 1$,
\begin{equation}\label{eq:induction_N_step}
\E\big[\|\Delta\Y\|_{[0,Nt^*]}^q\big]
\le |\Delta\bm{y}|^q \sum_{k=1}^N C^k + \E\big[\|\Delta\bm{Y}\|_{[0,t^*]}^q\big] \sum_{k=1}^N (N-k+1)C^k.
\end{equation} 
To show this, we do a proof by induction. The base case follows from~\eqref{eq:wasserstein_first_step}. Assume~\eqref{eq:induction_N_step} holds for some $N\ge 1$. It remains to show that~\eqref{eq:induction_N_step} holds for $N+1$. 

Applying the induction hypothesis conditional on the $\sigma$-algebra $\sigma(\bm{Y}_1(s),\bm{Y}_2(s);s\in[0,t^*])$ and integrating, we get
\begin{align*}
\E\big[\|\Delta\Y\|_{[t^*,(N+1)t^*]}^q\big]
&\le\E\big[|\Delta\Y(t^*)|^q\big] \sum_{k=1}^N C^k 
+\E\big[\|\Delta\bm{Y}\|_{[0,t^*]}^q\big]\sum_{k=1}^N (N-k+1)C^k\\
&\le C\big(|\Delta\bm{y}|^q + \E\big[\|\Delta\bm{Y}\|_{[0,t^*]}^q\big]\big) \sum_{k=1}^N C^k 
+\E\big[\|\Delta\bm{Y}\|_{[0,t^*]}^q\big]\sum_{k=1}^N (N-k+1)C^k.
\end{align*}
Hence, we may conclude the induction:
\begin{align*}
\E\big[\|\Delta\Y\|_{[0,(N+1)t^*]}^q\big]
&\le \E\big[\|\Delta\Y\|_{[0,t^*]}^q\big]
+ \E\big[\|\Delta\Y\|_{[t^*,(N+1)t^*]}^q\big]\\
&\le |\Delta\bm{y}|^q \sum_{k=1}^{N+1} C^{k}
+ \E\big[\|\Delta\bm{Y}\|_{[0,t^*]}^q\big] \sum_{k=1}^{N+1} ((N+1)-k+1)C^k.
\end{align*} 

Given $T>0$, set $N \coloneqq \ceil{K2^{1/q}T} \in \N$ and $t^*\coloneqq T/N \in (0,K^{-1})$. Then~\eqref{eq:induction_N_step} gives
\[
\E\big[\|\Delta\Y\|_{[0,T]}^q\big]
\le\big(|\Delta\bm{y}|^q 
    + \E\big[\|\Delta\bm{Y}\|_{[0,T]}^q\big]\big) C^{N} 
\sum_{k=1}^N kC^{1-k},
\enskip\text{where}\enskip
\sum_{k=1}^N kC^{1-k}
\le \frac{N(N+1)}{2},
\]
since $T=Nt^*$ and $t^*\le T$. By construction, $(Kt^*)^q=(KT/N)^q \le 1/2$ and hence $C \in (1,2]$. Thus, 
\begin{equation*}
\E\big[\|\Delta\Y\|_{[0,T]}^q\big]
\le \big(|\Delta\bm{y}|^q +\E\big[\|\Delta\bm{Y}\|_{[0,T]}^q\big]\big) N(N+1) 2^{N-1}. 
\end{equation*}
Taking infima over all couplings of $(\bm{Y}_1,\bm{Y}_2)$, we obtain the claim.
\end{proof}

\section{Applications to Domains of Attraction}
\label{sec:application_DoA} 
Let us consider again the SDE solutions $\X$ and $\ZZ$ in~\eqref{eq:main_SDE_setting}:
\begin{equation*}
\X_t(s)\coloneqq \bm{x}+\int_0^s V(\X_t(r-))\D \bm{X}_t(r), \quad 
\ZZ_s\coloneqq \bm{x}+\int_0^s V(\ZZ(r-))\D \bm{Z}(r), \, \text{ for }s \in [0,T] \text{ and }t \in (0,1],
\end{equation*}
where $V$ is Lipschitz, $\bm{X}$ is a L\'evy process in the small-time DoA of the stable process $\bm{Z}$, $g$ is the normalising function of $\bm{X}$, satisfying $\bm{X}(t)/g(t)\cid \bm{Z}(1)$ as $t\da 0$, and $\bm{X}_t\eqd (\bm{X}(st)/g(t))_{s\in[0,T]}$.

The purpose of this section is to study the rate at which $(\X_t(s))_{s \in [0,1]}$ converges in distribution to $(\ZZ(s))_{s \in [0,1]}$ as $t \da 0$, using that the driver $\bm{X}_t$ of $\X_t$ converge to the driver $\bm{Z}$ of $\ZZ$ as $t \da 0$. To study this, we will apply the thinning coupling introduced in the previous sections to both the domain of normal attraction and the domain of non-normal attraction. 

In the ensuing theorem, we assume the existence of some coupling, under which, some probability measure $\p_\theta$ exists. The construction and proof that this probability measure exists is given explicitly in Subsection~\ref{sec:thinning_technic_bound} (see~\eqref{eq:P_theta^t}) below for the thinning coupling, and in Section~\ref{sec:comono_technic_bound} (see~\eqref{eq:P_theta^t_cm}) for the comonotonic coupling. 

\begin{theorem}
\label{thm:roc_tempered_sde}
Let $\bm{X}$ be in the DoA of an $\alpha$-stable process $\bm{Z}$ for some $\alpha \in (0,2)\setminus\{1\}$. Let $\X_t$ and $\ZZ$ be the solutions of the SDEs in~\eqref{eq:main_SDE_setting} driven by $\bm{X}_t$ and $\bm{Z}$, respectively, and fix any $T,\theta>0$. Then, there exist couplings between $\bm{X}_t$ and $\bm{Z}$, and a probability measure $\p_\theta$ (dependent on the choice of coupling) for which $\sup_{t\in(0,1]}f(t)^{-r}\E_\theta\big[\|\X_t-\ZZ\|_{[0,T]}^r\big]<\infty$, where $r\coloneqq\lceil\alpha\rceil\in\{1,2\}$ and $f$ is given in Table~\ref{tab:thm_5.1}.
\vspace{-4mm}
\renewcommand{\arraystretch}{1.2}
\begin{table}[H]
\centering
\begin{tabular}{|C|L|T|} \hline
Range of $\alpha$ & DoNA under~(\nameref{asm:T}) with $H \equiv 1$ & DoNNA under~(\nameref{asm:T}) with $H\in\CSV_0(H_1,H_2)$, 
$H_2\in\SV_0$\\ \hline 
$\alpha \in (1,2)$ & 
If~\eqref{eq:symmetry_thin} holds, $f(t) \coloneqq t^{p/(2\alpha)}$; otherwise, $f(t)
\coloneqq t^{p/(2\alpha)} 
+ t^{1-1/\alpha}(1
+|\log t|\1_{\{p=\alpha-1\}})$ & $f(t)
=\sqrt{H_2(g(t))}$ \\ \hline
$\alpha \in (0,1)$ & $f(t)
\coloneqq t^{p/\alpha}$ & $f(t)
\coloneqq H_2(g(t))$ \\ \hline
\end{tabular}
\begin{tabular}{|C|L|T|} \hline
Range of $\alpha$ & DoNA under~(\nameref{asm:C}) with $H \equiv 1$ & DoNNA under~(\nameref{asm:C}) with $G\in\CSV_0(G_1,G_2)$, 
$G_2\in\SV_0$\\ \hline 
$\alpha \in (1,2)$ & If ~\eqref{eq:symmetry_como} holds, $f(t)\coloneqq t^{p/\alpha} 
$; otherwise, $f(t)
\coloneqq t^{p/\alpha}+
t^{1-1/\alpha}(1+|\log t|\1_{\{p=\alpha-1\}})$ & $f(t)\coloneqq G_2(t)$ \\ \hline
$\alpha \in (0,1)$ & $f(t)\coloneqq t^{p/\alpha}
$ & $f(t) \coloneqq G_2(t)$ \\ \hline
\end{tabular}\vspace{-2mm}
\caption{Values of the function $f$ in Theorem~\ref{thm:roc_tempered_sde}.}
\label{tab:thm_5.1}
\end{table}
\end{theorem}

Remarkably, the convergence rate can be arbitrarily fast (for appropriate $p,\delta>0$) in the case $\alpha\in(0,1)$ and even in the case $\alpha\in(1,2)$ under the additional balancing conditions. This stands in contrast with the general ``hard limits'' (i.e., unimprovable lower bounds) found in~\cite{WassersteinPaper} for Wasserstein distances. This can only be explained by the fact that the family of tight random variables established in the present work are ``far'' from being uniformly integrable.

In the following subsection, we will establish estimates required for the proof of Theorem~\ref{thm:roc_tempered_sde}. Before proceeding, we introduce two functions that will appear in such estimates but do not appear in Theorem~\ref{thm:roc_tempered_sde}. For $H\in\CSV_0(H_1,H_2)$ and $G\in\CSV_0(G_1,G_2)$ as in Assumptions~(\nameref{asm:T}) and~(\nameref{asm:C}), respectively, define the functions
\[
\ov{H}_1(x)
\coloneqq\1_{\{x>1\}}\int_1^x (1+H_1(y))\frac{\D y}{y},
\quad\text{and}\quad
\ov{G}_1(x)
\coloneqq\1_{\{x>1\}}\int_1^x (1+G_1(y))^{\alpha-1} \frac{\D y}{y}.
\]
By Karamata's theorem~\cite[Thm~1.5.9]{MR1015093}, $H_1\in\SV_\infty$ implies $\ov{H}_1\in\SV_\infty$ and $\ov{H}_1(x)/(1+H_1(x))\to\infty$ as $x\to\infty$ (resp. $G_1\in\SV_\infty$ implies $\ov{G}_1\in\SV_0$ and $\ov{G}_1(x)/(1+G_1(x))\to\infty$ as $x\to\infty$). To avoid many more subcases from arising in Theorem~\ref{thm:roc_tempered_sde}, we imposed additional assumptions on $H_2$ and $G_2$.

\subsection{Convergence of the Characteristics}
\label{sec:conv_characteristics}

In preparation for the bounds in this subsection, we note that, under either Assumption~(\nameref{asm:T}) or~(\nameref{asm:C}), if $\int_{D^\co}|\bm{w}|\,\nu_{\bm{X}^\co}(\D\bm{w})<\infty$, we have
\[
\E[\bm{X}^\co(1)]
=\bm{\gamma}_{\bm{X}^\co}
+\int_{D^\co}
\bm{w}\,\nu_{\bm{X}^\co}(\D\bm{w}),
\quad\text{where}\quad
\bm{\gamma}_{\bm{X}^\co}
=\bm{\gamma}_{\bm{X}}
-\int_{D}\bm{w}\,\nu_{\bm{X}^\D}(\D\bm{w}).
\]

\subsubsection{Thinning Coupling}
\label{sec:thinning_technic_bound}

The idea behind the proof of Theorem~\ref{thm:roc_tempered_sde} under Assumption~(\nameref{asm:T}) is to let the L\'evy process $\bm{X}^\co$ corresponding to the measure $\nu_{\bm{X}^\co}$ and $\bm{Z}$ follow the thinning coupling. We then condition on the event that $\bm{X}$ does not have jumps originating from the L\'evy measure $\nu_{\bm{X}^\D}$ on $[0,tT]$, which has probability $\exp(-tT\nu_{\bm{X}^\D}(\R^d_{\bm 0}))\to 1$ as $t\da 0$. Further, we will next construct an equivalent probability measure under which both processes have as many moments as necessary for our comparison inequalities to give the desired convergence rates. This requires the construction of a dominating L\'evy measure and certain calculations with respect to it, which we do next.

We will construct all necessary processes for our results in the same probability space. Recall that, under Assumption~(\nameref{asm:T}), we have~\eqref{eq:G-H}, i.e.:
\begin{equation*}
t=tG(t^\alpha)H(tG(t^\alpha))^{-1/\alpha}
\implies
G(t^\alpha)=H(tG(t^\alpha))^{1/\alpha}
\implies 
G(t)^{\alpha}=H(g(t)).
\end{equation*}
Under Assumption~(\nameref{asm:T}), $g(t)=t^{1/\alpha}G(t)$ and, by~\eqref{eq:G-H}, the L\'evy measure of $\bm{X}_t\eqd (\bm{X}(st)/g(t))_{s\in[0,T]}$ admits the decomposition
\begin{equation}\label{eq:nu_Xt-decomposition_TNN}
\begin{aligned}
\nu_{\bm{X}_t}(\D\bm{w})
&=t g(t)^d\nu_{\bm{X}}(\D(g(t)\bm{w}))
=
t g(t)^d\nu_{\bm{X}^\co}(\D(g(t)\bm{w}))
+
t g(t)^d\nu_{\bm{X}^\D}(\D(g(t)\bm{w}))\\
&=
G(t)^{-\alpha} H(g(t)|\bm{w}|)h(g(t)\bm{w})\nu_{\bm{Z}}(\D\bm{w})
+t g(t)^d\nu_{\bm{X}^\D}(\D(g(t)\bm{w}))\\
&=
\frac{H(g(t)|\bm{w}|)}{H(g(t))}h(g(t)\bm{w})\nu_{\bm{Z}}(\D\bm{w})
+t g(t)^d\nu_{\bm{X}^\D}(\D(g(t)\bm{w})).
\end{aligned}
\end{equation}
Thus, we may define the measures 
\[
\nu_{\bm{X}^\co_t}(\D\bm{w})
\coloneqq\frac{H(g(t)|\bm{w}|)}{H(g(t))}h(g(t)\bm{w})\nu_{\bm{Z}}(\D\bm{w}),
\quad 
\nu_{\bm{X}^\D_t}(\D\bm{w})
\coloneqq
t g(t)^d\nu_{\bm{X}^\D}(\D(g(t)\bm{w})),
\]
which satisfy $\nu_{\bm{X}^\D_t}(\R^d_{\bm0})=tm_0$ and $\nu_{\bm{X}_t}=\nu_{\bm{X}_t^\co}+\nu_{\bm{X}_t^\D}$. 

Note that, by Assumption~(\nameref{asm:T}), for any $t\in(0,1]$ and $\bm{w}\in\R^d_{\bm0}$ we have
\begin{equation*}
\max\bigg\{\frac{H(g(t)|\bm{w}|)}{H(g(t))}h(g(t)\bm{w}),1\bigg\} 
\le (1+K_h)[1+H_1(|\bm{w}|)H_2(g(t))] 
\le (1+K_h)[1+H_1(|\bm{w}|)].
\end{equation*}
We may define $\nu(\D\bm{w})\coloneqq (1+K_h) (1+ H_1(|\bm{w}|))\nu_{\bm{Z}} (\D\bm{w})$ on $\Borel(\R^d\setminus\{\bm0\})$. Note that $\nu_{\bm{X}_t^\co}\ll\nu$ and $\nu_{\bm{Z}}\ll\nu$ with Radon--Nikodym derivatives $\D\nu_{\bm X_t^\co}/\D\nu$ and $\D\nu_{\bm Z}/\D\nu$ bounded by $1$. Consider two independent Poisson random measures $\Xi^\co$ 
and $\Xi^\D$ on the set $(0,T]\times \R^d_{\bm0}$ with corresponding mean measures $\Leb\otimes \nu$ and $\Leb\otimes \nu_{\bm{X}^\D}$. Let $\bm{X}^\D$ be the driftless compound Poisson process with jump measure $\Xi^\D$ and, for all $t\in(0,1]$, let $\bm{X}^\co_t$ and $\bm{Z}$ be L\'evy processes with generating triplets $(\bm{\gamma}_{\bm{X}^\co_t},\bm{0},\nu_{\bm{X}_t^\co})$ and $(\bm{\gamma}_{\bm{Z}},\bm{0},\nu_{\bm{Z}})$, coupled via the thinning coupling of Section~\ref{sec_thinning_coup}, where 
\[
\bm{\gamma}_{\bm{X}^\co_t}
=\frac{t}{g(t)}\bigg(\bm{\gamma}_{\bm{X}}
-\int_{D}\bm{w}\,\nu_{\bm{X}^\D}(\D\bm{w})
\bigg)
+ c_\alpha\int_{\Sp^{d-1}}\bm{v}\int_1^{1/g(t)}
\frac{H(g(t)x)}{H(g(t))}h(g(t)x\bm{v})
\frac{\D x}{x^{\alpha}}\sigma(\D\bm{v}),
\] 
and the Poisson jump measures of both processes are obtained by thinning $\Xi^\co$. Note the identity in law $\bm{X}^\co_t+\bm{X}^\D_t\eqd \bm{X}_t$, where $\bm{X}^\D_t\coloneqq (\bm{X}^\D(st)/g(t))_{s\in[0,T]}$. 

Fix $\theta>0$ and by Campbell's formula~\cite[p.~28]{MR1207584}, let
\begin{equation}\label{eq:vartheta_defn}
\vartheta
\coloneqq-\log\E\bigg[\exp\bigg(\!-\theta\int_{(0,1]\times D^\co}|\bm{w}|\Xi^\co(\D t,\D\bm{w})\bigg)\bigg]
=\int_{D^\co}\big(1-\e^{-\theta|\bm{w}|}\big)\nu(\D\bm{w})>0.
\end{equation}
Define the equivalent probability measure $\p_\theta\approx \p$ given by the Radon--Nikodym derivative
\begin{equation}
\label{eq:P_theta^t}  
\frac{\D\p_\theta}{\D\p}=\exp\bigg(\vartheta T-\theta\int_{(0,T]\times D^\co}|\bm{w}|\Xi^\co(\D s,\D\bm{w})\bigg).
\end{equation}
Under Assumption~(\nameref{asm:T}) and $\p_\theta$, by~\cite[Thm~33.1]{MR3185174}, the process $\bm{X}^\D_t$ is L\'evy with measure $\nu_{\bm{X}_t^\D}$ and independent of $(\bm{X}^\co_t,\bm{Z})$, whose components are L\'evy with measures given by 
\[
\nu_{\bm{X}_t^\co}^\theta(\D\bm{w})
\coloneqq \e^{-\theta|\bm{w}|\1_{D^\co}(\bm{w})}\nu_{\bm{X}_t^\co}(\D\bm{w}),
\quad\text{and}\quad
\nu_{\bm{Z}}^\theta(\D\bm{w})
\coloneqq\e^{-\theta|\bm{w}|\1_{D^\co}(\bm{w})}\nu_{\bm{Z}}(\D\bm{w}),
\]
respectively, both dominated by $\nu^\theta(\D\bm{w})\coloneqq \e^{-\theta|\bm{w}|\1_{D^\co}(\bm{w})}\nu(\D\bm{w})$. The corresponding Radon--Nikodym derivatives do not depend on $\theta$ and $\Xi^\co$ has mean measure $\Leb\otimes\nu^\theta$ under $\p_\theta$. Since $\p_\theta$ only modifies the ``probability'' of observing a jump, the processes $\bm{X}^\co$ and $\bm{Z}$ have the same natural drifts $\bm{b}_{\bm{X}^\co}=\bm{b}_{\bm{Z}}=\bm{0}$ as their un-tempered counterparts when $\alpha\in(0,1)$. More generally, since these measures agree with their un-tempered counterparts on $D$, the processes $\bm{X}^\co$ and $\bm{Z}$ have the same drift parameters $\bm{\gamma}_{\bm{X}}$ and $\bm{\gamma}_{\bm{Z}}$, respectively. However, the drifts $\bm{a}_{\bm{X}_t^\co}^\theta\coloneqq\E_\theta[\bm{X}_t^\co(1)]$ and $\bm{a}_{\bm{Z}}^\theta\coloneqq\E_\theta[\bm{Z}(1)]$ corresponding to the fully compensated jump measures under $\p_\theta$ generally \emph{do} depend on $\theta$.

\begin{proposition}
\label{prop:integrals_DoA_thin}
Let Assumption~(\nameref{asm:T}) hold, $G$ and $g(t)\coloneqq t^{1/\alpha}G(t)$ be as in~\eqref{eq:G-H}, consider the processes $\bm{X}_t^\co+\bm{X}_t^\D\eqd\bm{X}_t$, $t\in(0,1]$, and $\bm{Z}$ as above and fix $\theta>0$ and $r>\alpha$. Then $\nu^\theta(D^\co)\le\nu(D^\co)<\infty$ and 
\begin{equation}
\label{eq:thin-w^r-dist-mu}
\int_{\R^d_{\bm 0}} |\bm{w}|^r\bigg|\frac{\D\nu_{\bm{Z}}}{\D\nu}(\bm{w})
    -\frac{\D\nu_{\bm{X}_t^\co}}{\D\nu}(\bm{w})\bigg| \nu^\theta(\D\bm{w})
\lesssim H_2(g(t)) + g(t)^p,
\quad\text{as }t\da 0.
\end{equation}
If $\alpha\in(1,2)$ and $\int_{D^\co} |\bm{w}|\nu_{\bm{X}^\co}(\D\bm{w})<\infty$, then as $t\da 0$,
\[
\big|\bm{a}_{\bm{X}^\co_t}^\theta-\bm{a}_{\bm{Z}}^\theta\big|
\lesssim
H_2(g(t)) + t/g(t) + g(t)^p 
+ g(t)^{\alpha-1} \big[1+H_1(1/g(t))
    +\ov{H}_1(1/g(t))\1_{\{p=\alpha-1\}}\big],
\]
where the term $t/g(t)$ vanishes if $\E[\bm{X}^\co(1)]=\bm{0}$ and all but the term $t/g(t)$ vanish if, for all $x>0$, $\int_{\Sp^{d-1}}\bm{v}\,\sigma(\D\bm{v})=\int_{\Sp^{d-1}}h(x\bm{v})\bm{v}\,\sigma(\D\bm{v})=\bm{0}$ ($\bm{a}^\theta_{\bm{X}^\co_t}\equiv \bm{a}^\theta_{\bm{Z}}=\bm{0}$ if both conditions hold). If $\alpha\in(0,1)$, then
\[
\big|\bm{a}_{\bm{X}^\co_t}^\theta-\bm{a}_{\bm{Z}}^\theta\big|
\lesssim H_2(g(t))+g(t)^p,
\quad\text{as }t\da 0,
\]
and we have $\bm{a}^\theta_{\bm{X}^\co_t}\equiv \bm{a}^\theta_{\bm{Z}}=\bm{0}$ if $\int_{\Sp^{d-1}}\bm{v}\,\sigma(\D\bm{v})=\int_{\Sp^{d-1}}h(x\bm{v})\bm{v}\,\sigma(\D\bm{v})=\bm{0}$ for all $x>0$.
\end{proposition}

\begin{proof}
The slow variation of $H_1$ ensures that the integrability of certain functions depends only on the power of the argument and not on any other characteristics of $H_1$ via~\cite[Prop.~1.5.10]{MR1015093}. Indeed, for instance, the first part of the bound $\nu^\theta(D^\co)\le \nu(D^\co)<\infty$ is elementary and the integrability of the latter follows from the slow variation of $H_1$ and Karamata's theorem~\cite[Thm~1.5.11]{MR1015093}. 

Note from Assumptions~(\nameref{asm:T}) that, for any $t\in(0,1]$, $x>0$ and $\bm{v}\in\Sp^{d-1}$, we have
\begin{equation}
\label{eq:Z-vs_X^c-T}
\begin{split}
\bigg|1-\frac{H(g(t)x)}{H(g(t))}h(g(t)x\bm{v})\bigg|
&\le\bigg|1-\frac{H(g(t)x)}{H(g(t))}\bigg|
+K_h(1+H_1(x))[1\wedge (g(t)x)^p]\\
&\le
H_1(x) H_2(g(t))
+K_h (1+H_1(x)) [1\wedge (g(t)x)^p].
\end{split}
\end{equation}
Thus,~\cite[Prop.~1.5.10]{MR1015093} and~\eqref{eq:Z-vs_X^c-T} give, as $t\da 0$,
\begin{multline*}
\int_{D} |\bm{w}|^r\bigg|\frac{\D\nu_{\bm{Z}}}{\D\nu}(\bm{w})-\frac{\D\nu_{\bm{X}_t^\co}}{\D\nu}(\bm{w})\bigg| \nu^\theta(\D\bm{w})
=c_\alpha\int_{\Sp^{d-1}}\int_0^1 \bigg|1-\frac{H(g(t)x)}{H(g(t))}h(g(t)x\bm{v})\bigg|\frac{\D x}{x^{\alpha-r+1}}\sigma(\D \bm{v})\\
\lesssim 
H_2(g(t))\int_0^1H_1(x)\frac{\D x}{x^{\alpha-r+1}}
+g(t)^p \int_0^{1} (1+H_1(x))x^p \frac{\D x}{x^{\alpha-r+1}}
\lesssim H_2(g(t))+g(t)^p.
\end{multline*} 
Similarly, since $x\mapsto x^{p+r}\e^{-\theta x}$ is bounded by $(\e\theta/(p+r))^{-p-r}$ on $[1,\infty)$, by~\eqref{eq:Z-vs_X^c-T} we get, as $t\da 0$, 
\begin{multline*}
\int_{D^\co} |\bm{w}|^r\bigg|\frac{\D\nu_{\bm{Z}}}{\D\nu}(\bm{w})-\frac{\D\nu_{\bm{X}_t^\co}}{\D\nu}(\bm{w})\bigg| \nu^\theta(\D\bm{w})
\lesssim
\int_{\Sp^{d-1}}\int_1^\infty \bigg|1-\frac{H(g(t)x)}{H(g(t))}h(g(t)x\bm{v})\bigg| 
\frac{\D x}{x^{\alpha+p+1}}\sigma(\D \bm{v})\\
\lesssim 
H_2(g(t))\int_1^\infty H_1(x)
\frac{\D x}{x^{\alpha+p+1}}
+ g(t)\int_1^\infty 
(1+H_1(x))\frac{\D x}{x^{\alpha+1}}
\lesssim 
H_2(g(t))
+g(t)^p,
\end{multline*}
where we used Karamata's theorem~\cite[Thm~1.5.11]{MR1015093}. The second and third terms in the previous display are bounded by multiplicative multiples of $g(t)^{\alpha+p}[1+H_1(1/g(t))]$ and $g(t)^p$, respectively, where the latter is asymptotically larger by the slow variation of $H_1$, the regular variation of $g$ and Potter's bound~\cite[Thm~1.5.6]{MR1015093}. Combining the previous two displays gives the first claim. 

Suppose $\alpha\in(1,2)$. Then $\E[\bm{Z}(1)]=\bm{0}$, so that $\bm{a}_{\bm{X}_t^\co}^\theta\coloneqq\E^\theta[\bm{X}_t^\co(1)]$ and $\bm{a}_{\bm{Z}}^\theta\coloneqq\E^\theta[\bm{Z}(1)]$ are given by:
\begin{align*}
\bm{a}_{\bm{Z}}^\theta 
&= \bm{\gamma}_{\bm{Z}}+\int_{D^\co }\bm{w}\nu_{\bm{Z}}^\theta(\D\bm{w})
=\int_{D^\co }\bm{w}\big(\nu_{\bm{Z}}^\theta(\D\bm{w}) - \nu_{\bm{Z}}(\D\bm{w})\big)
=c_\alpha\int_{\Sp^{d-1}}\bm{v}\,\sigma(\D\bm{v})\int_1^\infty \big(\e^{-\theta x}-1\big)\frac{\D x}{x^{\alpha}}\\
\bm{a}_{\bm{X}_t^\co}^\theta
&= \frac{t}{g(t)}\E[\bm{X}^\co(1)]
+\E^\theta[\bm{X}_t^\co(1)]-\E[\bm{X}_t^\co(1)]\\
&=\frac{t}{g(t)}\E[\bm{X}^\co(1)]
+c_\alpha\!\int_{\Sp^{d-1}}\!\!\bm{v}\int_1^\infty\!\!
\big(\e^{-\theta x}-1\big)
\frac{H(g(t)x)}{H(g(t))}h(g(t)x\bm{v})
\frac{\D x}{x^{\alpha}}\sigma(\D\bm{v}). 
\end{align*} 
Proceeding as before, by Assumption~(\nameref{asm:T}) and~\eqref{eq:Z-vs_X^c-T}, we obtain, as $t \da 0$,
\begin{align*}
\big|\bm{a}_{\bm{X}_t^\co}^\theta
-\bm{a}_{\bm{Z}}^\theta\big| 
&\lesssim t/g(t)
+\int_{\Sp^{d-1}}\!\!\bm{v}\int_1^\infty
\big(1-\e^{-\theta x}\big)
\bigg|1-\frac{H(g(t)x)}{H(g(t))}h(g(t)x\bm{v})
\bigg|
\frac{\D x}{x^{\alpha}}\sigma(\D\bm{v})
\\
&\lesssim t/g(t)
+H_2(g(t))\int_1^\infty H_1(x)\frac{\D x}{x^{\alpha}}
+\int_1^\infty (1+H_1(x))[1\wedge (g(t)x)^p]\frac{\D x}{x^{\alpha}}.
\end{align*}
Then, breaking the last integral in two and applying Karamata's theorem, we obtain
\begin{align*}
\int_1^\infty (1+H_1(x))[1\wedge (g(t)x)^p]\frac{\D x}{x^{\alpha}}
&=g(t)^p\int_1^{1/g(t)} (1+H_1(x))\frac{\D x}{x^{\alpha-p}}
+\int_{1/g(t)}^\infty (1+H_1(x))\frac{\D x}{x^{\alpha}}\\
&\lesssim
 g(t)^p + g(t)^{\alpha-1} 
    \big[1+H_1(1/g(t))
    +\ov{H}_1(1/g(t))\1_{\{p=\alpha-1\}}\big].
\end{align*} 
The bound on $|\bm{a}_{\bm{X}_t^\co}^\theta
-\bm{a}_{\bm{Z}}^\theta|$ then follows easily. The statements on the special cases where $\E[\bm{X}^\co(1)]=\bm{0}$ or $\int_{\Sp^{d-1}}\bm{v}\,\sigma(\D\bm{v})=\int_{\Sp^{d-1}}h(x\bm{v})\bm{v}\,\sigma(\D\bm{v})=\bm{0}$ for all $x>0$, follow with ease.

In the finite variation case $\alpha \in (0,1)$, using that $\bm{X}^\co$ (and hence $\bm{X}^\co_t$) and $\bm{Z}$ have zero natural drift, the boundedness of $x\mapsto x^{p+1}\e^{-\theta x}$ by $(\e\theta/(p+1))^{-p-1}$ on $[1,\infty)$, and~\eqref{eq:thin-w^r-dist-mu}--\eqref{eq:Z-vs_X^c-T}, we obtain, as $t \da 0$, 
\begin{multline*}
\big|\bm{a}_{\bm{X}_t^\co}^\theta
-\bm{a}_{\bm{Z}}^\theta\big|
= \bigg|\int_{\R^d_{\bm0}}\bm{w}\big(\nu_{\bm{X}_t^\co}^\theta 
- \nu_{\bm{Z}}^\theta\big)(\D\bm{w})\bigg|
\le \int_{\R^d_{\bm0}} |\bm{w}|\cdot\bigg|\frac{\D\nu_{\bm{Z}}}{\D\nu}(\bm{w})-\frac{\D\nu_{\bm{X}_t^\co}}{\D\nu}(\bm{w})\bigg| \nu^\theta(\D\bm{w})
\lesssim H_2(g(t)) + g(t)^p.
\end{multline*}
The special case where $\int_{\Sp^{d-1}}\bm{v}\,\sigma(\D\bm{v})=\int_{\Sp^{d-1}}h(x\bm{v})\bm{v}\,\sigma(\D\bm{v})=\bm{0}$ for all $x>0$ again follows easily.
\end{proof}

\subsubsection{Comonotonic Coupling}
\label{sec:comono_technic_bound}

We start with the following lemma, which is an adaptation of~\cite[Lem.~4.10]{WassersteinPaper} to the setting of the present section, and proved here for completeness.

\begin{lemma}
\label{lem:bounds_h_2_multidim}
Under Assumption~(\nameref{asm:C}), there exists a function $\wt h:(0,\infty)\times \Sp^{d-1}\to (0,\infty)$ and a constant $K_{\wt h} \ge 0$, such that, for all $x>0$ and $\bm{v} \in \Sp^{d-1}$, 
\begin{equation}
\label{eq:rho_X^la}
\rho_{\bm{X}^\co}^{\la}(x,\bm{v})=x^{-1/\alpha}G(1/x)\wt h(x\bm{v}) \quad \text{and}\quad 
\big|( c_\alpha/\alpha )^{1/\alpha}-\wt h(x\bm{v})\big|
\le K_{\wt h}\big(1\wedge (x^{-p/\alpha}G(1/x)^p+x^{-\delta})\big).
\end{equation}
\end{lemma}

\begin{proof}
By continuity of $x\mapsto \rho_{\bm{X}^\co}([x,\infty),\bm{v})$, we have $\rho_{\bm{X}^{\co}}([\rho_{\bm{X}^\co}^{\la}(x,\bm{v}),\infty),\bm{v})
=x$ for all $\bm{v}\in \Sp^{d-1}$ and $x>0$. Hence, for all $\bm{v}\in \Sp^{d-1}$ and $x>0$,
\begin{align*}
x
&=h(\rho_{\bm{X}^\co}^{\la}(x,\bm{v})\bm{v}) H(\rho_{\bm{X}^\co}^{\la}(x,\bm{v}))^\alpha\rho_{\bm{X}^\co}^{\la}(x,\bm{v})^{-\alpha},
\quad\text{implying}\\
\rho_{\bm{X}^\co}^{\la}(x,\bm{v})
&=x^{-1/\alpha}H(\rho_{\bm{X}^\co}^{\la}(x,\bm{v}))h(\rho_{\bm{X}^\co}^{\la}(x,\bm{v})\bm{v})^{1/\alpha} =x^{-1/\alpha}G(1/x)\frac{H(\rho_{\bm{X}^\co}^{\la}(x,\bm{v}))}{G(1/x)}h(\rho_{\bm{X}^\co}^{\la}(x,\bm{v})\bm{v})^{1/\alpha}.
\end{align*}
Thus, the first part of~\eqref{eq:rho_X^la} holds if 
$\wt h(x\bm{v})\coloneqq (H(\rho_{\bm{X}^\co}^{\la}(x,\bm{v}))/G(1/x)) h(\rho_{\bm{X}^\co}^{\la}(x,\bm{v})\bm{v})^{1/\alpha} \in (0,\infty)$.

Assume that~\eqref{eq:old_assump_(H)_C} in Assumption~(\nameref{asm:C}) holds for some $p,\delta>0$. Since $h$ is bounded by $K_h$ and by~\eqref{eq:old_assump_(H)_C}, it follows that $\wt h(x,\bm{v})\le (K_Q+1) (c_\alpha/\alpha+K_h)^{1/\alpha}$ for all $x>0$ and $\bm{v}\in \Sp^{d-1}$. Moreover, define $\wt K \coloneqq (1+\alpha K_h/c_\alpha)^{1/\alpha}$ and $\wt M \coloneqq \max\{|\wt K^r-1|/|\wt K-1|,1\}$, then the elementary inequality $|z^r-1|\le \wt M |z-1|$ for any $r \ge 0$ and $z\in [0,\wt K]$ and the triangle inequality yield $|y z^r-1|\le \wt M |z-1|+|y-1|z^r$, which implies for all $x>0$ and $\bm{v} \in \Sp^{d-1}$, that
\begin{align*}
\bigg|\bigg(\frac{c_\alpha}{\alpha}\bigg)^{1/\alpha}&-\wt h(x,\bm{v})\bigg|
=\left(\frac{c_\alpha}{\alpha}\right)^{1/\alpha} \left|1-\left(\frac{\alpha}{c_\alpha}h(\rho_{\bm{X}^\co}^{\la}(x,\bm{v})\bm{v})\right)^{1/\alpha}\frac{H(\rho_{\bm{X}^\co}^{\la}(x,\bm{v}))}{G(1/x)}  \right| \\
&\le \wt M \left(\frac{c_\alpha}{\alpha}\right)^{1/\alpha-1} \left|\frac{c_\alpha}{\alpha}-h(\rho_{\bm{X}^\co}^{\la}(x,\bm{v})\bm{v})\right|+\bigg|\frac{H(\rho_{\bm{X}^\co}^{\la}(x,\bm{v}))}{G(1/x)}-1\bigg| 
\left(\left(\frac{c_\alpha}{\alpha}\right)^2+\frac{c_\alpha}{\alpha}K_h\right)^{1/\alpha}\\
&\le C \rho_{\bm{X}^\co}^{\la}(x,\bm{v})^p+C x^{-\delta} \le C' x^{-p/\alpha}G(1/x)^p \left( \frac{H(\rho_{\bm{X}^\co}^{\la}(x,\bm{v}))}{G(1/x)}\right)^p h(\rho_{\bm{X}^\co}^{\la}(x,\bm{v})\bm{v})^{p/\alpha}+Cx^{-\delta}\\
& \le 
C'' x^{-p/\alpha}G(1/x)^p+
C x^{-\delta},
\end{align*} for some constants $C, C',C''>0$, concluding the proof. 
\end{proof}

A simple consequence of Lemma~\ref{lem:bounds_h_2_multidim} is that, under Assumption~(\nameref{asm:C}), the process $\bm{X}^\co$ has as many moments as $\bm{Z}$ since $\wt h$ is bounded and $G(x)=1$ for $x\le 1$.

As with the thinning coupling above, we will construct all necessary auxiliary processes in the same probability space. Let $\bm{X}^\D$ be the driftless compound Poisson process with jump measure $\Xi^\D$ and, for all $t\in(0,1]$, let $\bm{X}^\co_t$ and $\bm{Z}$ be L\'evy processes with generating triplets $(\bm{\gamma}_{\bm{X}^\co_t},\bm{0},\nu_{\bm{X}_t^\co})$ and $(\bm{\gamma}_{\bm{Z}},\bm{0},\nu_{\bm{Z}})$, coupled via the comonotonic coupling of Section~\ref{sec:comonot_coup}, where 
\begin{align*}
\bm{\gamma}_{\bm{X}^\co_t}
&=\frac{t}{g(t)}\bm{\gamma}_{\bm{X}^\co}
+ \int_{\Sp^{d-1}} \int_1^{1/g(t)}  
x \bm{v} \rho_{\bm{X}_t^\co}(
 \D x,\bm{v})
 \sigma(\D \bm{v})\\
&=\frac{t}{g(t)}\bigg(\bm{\gamma}_{\bm{X}}
-\int_{D}\bm{w}\,\nu_{\bm{X}^\D}(\D\bm{w})
\bigg)
+ \int_{\Sp^{d-1}} \int_{\rho^{\la}_{\bm{X}_t^\co}(1,\bm{v})}^{\rho^{\la}_{\bm{X}_t^\co}(1/g(t),\bm{v})}  
 \bm{v} \rho^{\la}_{\bm{X}_t^\co}(x,\bm{v}) 
 \D x\,
 \sigma(\D \bm{v}),
\end{align*}
and the Poisson jump measures of both processes are obtained by transforming the atoms of $\Xi^\co$. Note the identity in law $\bm{X}^\co_t+\bm{X}^\D_t\eqd \bm{X}_t$, where $\bm{X}^\D_t\coloneqq (\bm{X}^\D(st)/g(t))_{s\in[0,T]}$. 

Next, as in the case of thinning, we need to do exponential tempering so that we may apply Theorem~\ref{thm:gen_bound_como} above. To that end, for any $\theta>0$ we define (by Campbell's formula~\cite[p.~28]{MR1207584})
\begin{equation}\label{eq:vartheta_defn_cm}
\vartheta \coloneqq -\log \E \bigg[\exp\bigg(-\theta\int_{[0,1]\times (0,1)\times \Sp^{d-1}} x^{-1} 
    \Xi^\co(\D t,\D x,\D\bm{v}) \bigg)\bigg]= \frac{1}{\alpha}\int_1^\infty \big(1-\e^{-\theta x} \big)  \frac{\D x}{x^2}>0. 
\end{equation}
Define next, the equivalent probability measure $\p_\theta\approx \p$ given by the Radon--Nikodym derivative
\begin{equation}
\label{eq:P_theta^t_cm}  
\frac{\D\p_\theta}{\D\p}=\exp\bigg(\vartheta T-\theta\int_{[0,T]\times (0,1)\times \Sp^{d-1}} x^{-1} \Xi^\co(\D t, \D x,\D\bm{v})\bigg).
\end{equation}
Under Assumption~(\nameref{asm:C}) and $\p_\theta$, by~\cite[Thm~33.1]{MR3185174}, the process $\bm{X}^\D_t$ is L\'evy with measure $\nu_{\bm{X}_t^\D}$ and independent of $(\bm{X}^\co_t,\bm{Z})$, whose components are L\'evy with measures given by 
\begin{align*}
\nu_{\bm{X}_t^\co}^\theta(B)
&\coloneqq 
\int_{\Sp^{d-1}}\int_0^\infty \1_{B}(\rho^{\la}_{\bm{X}_t^\co}(x,\bm{v})\bm{v}) \exp\big(-\1_{(0,1)}(x)\theta /x\big) \D x\, \sigma(\D\bm{v}),\\
\nu_{\bm{Z}}^\theta(B)
&\coloneqq \int_{\Sp^{-1}} \int_0^\infty \1_{B}(\rho^{\la}_{\bm{Z}}(x)\bm{v}) \exp\big(-\1_{(0,1)}(x)\theta /x\big)  
\D x\,
\sigma(\D \bm{v}),
\end{align*}
respectively, for $B\in \mathcal{B}(\R^d_{\bm{0}})$. Since the tempering only affects large jumps, the processes $\bm{X}^\co$ and $\bm{Z}$ have the same drift parameters $\bm{\gamma}_{\bm{X}}$ and $\bm{\gamma}_{\bm{Z}}$, respectively, as well as natural drifts $\bm{b}_{\bm{X}^\co}=\bm{b}_{\bm{Z}}=\bm{0}$ when $\alpha\in(0,1)$. However, the drifts $\bm{a}_{\bm{X}_t^\co}^\theta\coloneqq\E_\theta[\bm{X}_t^\co(1)]$ and $\bm{a}_{\bm{Z}}^\theta\coloneqq\E_\theta[\bm{Z}(1)]$ corresponding to the fully compensated jump measures under $\p_\theta$ generally depend on $\theta$. 

\begin{proposition}
\label{prop:integrals_DoA_como} 
Let Assumption~(\nameref{asm:C}) hold, $G$ be as in~\eqref{eq:defn_G_C} with $G \in \CSV_0(G_1,G_2)$ for auxiliary functions $G_1$ and $G_2$, consider the processes $\bm{X}_t^\co+\bm{X}_t^\D\eqd\bm{X}_t$, $t\in(0,1]$, and $\bm{Z}$ as above and fix $\theta>0$. Assume $\alpha \in (0,2)\setminus\{1\}$ and let $r>\alpha$. Then as $t \da 0$, we have
\begin{equation}\label{eq:bound_como_r_dist}
    \int_{\Sp^{d-1}}\int_0^\infty
|\rho_{\bm{X}_t^\co}^{\la}(x,\bm{v})-\rho_{\bm{Z}}^{\la}(x)|^r \e^{-\1_{(0,1)}(x)\theta /x} \D x \, \sigma(\D \bm{v}) 
\lesssim G_2(t)^r + g(t)^{r p}+ t^{r \delta}.
\end{equation}    
Moreover, if $\alpha \in (1,2)$ and $\int_{D^\co}|\bm{w}|\nu_{\bm{X}^\co}(\D\bm{w})<\infty$, then as $t \da 0$
\begin{equation*}
|\bm{a}_{\bm{X}_t^\co}^\theta
    -\bm{a}_{\bm{Z}}^\theta| 
\lesssim t/g(t)+G_2(t)
+G(t)^p\big(t^{p/\alpha}+t^{1-1/\alpha}[1+G_1(1/t)^p +\ov{G}_1(1/t)\1_{\{p=\alpha-1\}}]\big) +t^{\delta},
\end{equation*} 
where the term $t/g(t)$ vanishes if $\E[\bm{X}^\co(1)]=\bm{0}$ and all but the term $t/g(t)$ vanish if, for all $x>0$, $\int_{\Sp^{d-1}}\bm{v}\,\sigma(\D\bm{v})=\int_{\Sp^{d-1}}\rho^{\la}_{\bm{X}_t^\co}(x,\bm{v})\bm{v}\,\sigma(\D\bm{v})=\bm{0}$ ($\bm{a}_{\bm{X}^\co_t}^\theta\equiv \bm{a}_{\bm{Z}}^\theta=\bm{0}$ if both conditions hold). If $\alpha\in(0,1)$, then
\[
\big|\bm{a}_{\bm{X}^\co_t}^\theta-\bm{a}_{\bm{Z}}^\theta\big|
\lesssim G_2(t) + g(t)^{p} + t^{\delta},
\quad\text{as }t\da 0,
\]
and we have $\bm{a}_{\bm{X}^\co_t}^\theta \equiv \bm{a}_{\bm{Z}}^\theta = \bm{0}$ if $\int_{\Sp^{d-1}}\bm{v}\,\sigma(\D\bm{v}) = \int_{\Sp^{d-1}}\rho^{\la}_{\bm{X}_t^\co}(x,\bm{v})\bm{v}\,\sigma(\D\bm{v})=\bm{0}$ for all $x>0$.
\end{proposition}

In our applications of~\eqref{eq:bound_como_r_dist} we will choose $r=\lceil\alpha\rceil$. 

\begin{proof} 
From Lemma~\ref{lem:bounds_h_2_multidim}, we have
\begin{equation}\label{eq:time_changed_right_inverse}
\rho_{\bm{X}_t^\co}^{\la}(x,\bm{v})
=\rho_{\bm{X}^\co}^{\la}(x/t,\bm{v})/g(t)
=x^{-1/\alpha} \frac{G(t/x)}{G(t)}\wt h(\bm{v} x/t).
\end{equation}
Throughout, let $r>\alpha$. Since $G \in \CSV_0(G_1,G_2)$, from~\eqref{eq:time_changed_right_inverse} and Lemma~\ref{lem:bounds_h_2_multidim}, it follows that
\begin{align*}
    &\int_{\Sp^{d-1}}\int_0^\infty
|\rho_{\bm{X}_t^\co}^{\la}(x,\bm{v})-\rho_{\bm{Z}}^{\la}(x)|^r \e^{-\1_{(0,1)}(x)\theta /x} \D x \, \sigma(\D \bm{v})\\
&\quad = \int_{\Sp^{d-1}}\int_0^\infty
\bigg|\frac{G(t/x)}{G(t)}\wt h(\bm{v} x/t)-\left(\frac{c_\alpha}{\alpha} \right)^{1/\alpha} \bigg|^r \e^{-\1_{(0,1)}(x)\theta /x} \frac{\D x}{x^{r/\alpha}} \, \sigma(\D \bm{v})\\
&\quad \lesssim \int_{\Sp^{d-1}}\int_0^\infty \bigg(
\bigg|\frac{G(t/x)}{G(t)}-1\bigg|^r+\bigg|\wt h(\bm{v} x/t)-\left(\frac{c_\alpha}{\alpha} \right)^{1/\alpha} \bigg|^r \bigg)\e^{-\1_{(0,1)}(x)\theta /x} \frac{\D x}{x^{r/\alpha}} \, \sigma(\D \bm{v})\\
&\quad \lesssim G_2(t)^r \int_0^\infty 
G_1(1/x)^r\e^{-\1_{(0,1)}(x)\theta /x}\frac{\D x}{x^{r/\alpha}} \\
&\qquad + \int_{\Sp^{d-1}}\int_0^\infty \big[1\wedge ((x/t)^{- p/\alpha}G(t/x)^{p}+(x/t)^{-\delta})^r\big] \e^{-\1_{(0,1)}(x)\theta /x} \frac{\D x}{x^{r/\alpha}} \, \sigma(\D \bm{v})\eqqcolon I_1(t)+I_2(t).
\end{align*}

Since $G_1 \in \SV_0\cap\SV_\infty$, Karamata's theorem~\cite[Thm~1.5.11]{MR1015093} gives
\begin{equation*}
I_1(t) = G_2(t)^r\bigg(\int_1^\infty G_1(1/x)^r\frac{\D x}{x^{r/\alpha}}
+\int_0^1 G_1(1/x)^r \e^{-\theta/x} \frac{\D x}{x^{r/\alpha}}\bigg) 
\lesssim G_2(t)^r, \quad \text{ for }t \da 0.
\end{equation*}
On the other hand, since $G \in \CSV_0(G_1,G_2)$, Remark~\ref{rem:CSV} yields, as $t\da 0$,
\begin{align*}
I_2(t) &\le \int_0^1 \big((x/t)^{- p/\alpha}G(t/x)^{p}+(x/t)^{-\delta}\big)^r \e^{-\theta /x} \frac{\D x}{x^{r/\alpha}}  
+ \int_1^\infty \big((x/t)^{-p/\alpha}G(t/x)^{p}+(x/t)^{-\delta}\big)^r \frac{\D x}{x^{r/\alpha}} \\
&\lesssim t^{r p/\alpha}G(t)^{r p}\int_0^1 x^{-r p/\alpha-r/\alpha}(1+G_1(1/x))^{r p}\e^{-\theta x^{-1/\alpha}}\D x+ t^{r\delta}\int_0^1 x^{-r\delta-r/\alpha} \e^{-\theta /x} \D x \\
&\quad + t^{rp/\alpha}G(t)^{rp} \int_1^\infty x^{- rp/\alpha-r/\alpha}(1+G_1(1/x))^{rp} 
\D x + t^{r\delta} \int_1^\infty x^{-r\delta-r/\alpha} \D x
\lesssim 
 t^{rp/\alpha}G(t)^{rp} 
+ t^{r\delta}. 
\end{align*}
Combining the bounds in the previous three displays, we obtain first claim.

Suppose $\alpha\in(1,2)$. Then $\E[\bm{Z}(1)]=\bm{0}$, so that $\bm{a}_{\bm{X}_t^\co}^\theta\coloneqq\E^\theta[\bm{X}_t^\co(1)]$ and $\bm{a}_{\bm{Z}}^\theta\coloneqq\E^\theta[\bm{Z}(1)]$ are given by:
\begin{align*}
\bm{a}_{\bm{Z}}^\theta 
&= \bm{\gamma}_{\bm{Z}}+\int_{D^\co }\bm{w}\nu_{\bm{Z}}^\theta(\D\bm{w})
=\int_{D^\co }\bm{w}\big(\nu_{\bm{Z}}^\theta(\D\bm{w}) - \nu_{\bm{Z}}(\D\bm{w})\big)
=\int_{\Sp^{d-1}}\bm{v}\int_{0}^1
(\e^{-\theta/x}-1)
\rho_{\bm{Z}}^\la(x)
\D x\,\sigma(\D\bm{v}),\\
\bm{a}_{\bm{X}_t^\co}^\theta
&= \frac{t}{g(t)}\E[\bm{X}^\co(1)]
+\E^\theta[\bm{X}_t^\co(1)]-\E[\bm{X}_t^\co(1)]\\
 &= \frac{t}{g(t)}\bigg(\bm{\gamma}_{\bm{X}}
-\int_{D}\bm{w}\,\nu_{\bm{X}^\D}(\D\bm{w})
+ \int_{D^\co }\bm{w}\,\nu_{\bm{X}^\co}(\D\bm{w}) \bigg)
+\int_{\Sp^{d-1}} \int_0^1  
 \bm{v} \big(\e^{-\theta/x}-1\big)\rho^{\la}_{\bm{X}_t^\co}(x,\bm{v}) 
 \D x\,
 \sigma(\D \bm{v}).
\end{align*} 

Thus, applying Lemma~\ref{lem:bounds_h_2_multidim} and that $G \in \CSV_0(G_1,G_2)$, yields, as $t\da 0$,
\begin{align*}
|\bm{a}_{\bm{X}_t^\co}^\theta-\bm{a}_{\bm{Z}}^\theta| &\lesssim \frac{t}{g(t)} + \int _0^1 \big|\e^{-\theta/x}-1\big|\cdot |\rho^{\la}_{\bm{X}_t^\co}(x,\bm{v})-\rho^{\la}_{\bm{Z}}(x)| \D x\\
&\lesssim \frac{t}{g(t)} + \int _0^1 \bigg(\bigg|\frac{G(t/x)}{G(t)}-1\bigg|\wt h(\bm{v} x/t)
+\bigg|\,\wt h(\bm{v} x/t)-\left(\frac{c_\alpha}{\alpha}\right)^{1/\alpha}\bigg| \bigg)
\frac{\D x}{x^{1/\alpha}}\\
&\lesssim \frac{t}{g(t)} 
+ \int _0^1 \big[G_1(1/x)G_2(t)
+K_{\wt h}\big(1\wedge ((x/t)^{-p/\alpha}G(t/x)^p+(x/t)^{-\delta})\big) \big]\frac{\D x}{x^{1/\alpha}} \\
&\lesssim \frac{t}{g(t)} + G_2(t) \int _0^1 G_1(1/x) \frac{\D x}{x^{1/\alpha}}
+ \int_0^1 1\wedge \big((x/t)^{-p/\alpha}G(t/x)^p+(x/t)^{-\delta}\big) \frac{\D x}{x^{1/\alpha}}.
\end{align*} 
Note that $\int _0^1 G_1(1/x) x^{-1/\alpha}\D x<\infty$ by Karamata's theorem~\cite[Thm~1.5.11]{MR1015093} since $G_1 \in \SV_\infty$. Hence, it remains to control the final integral in the display above. Again, by Karamata's theorem, we deduce 
\begin{align*}
&\int_0^1 1\wedge \big((x/t)^{-p/\alpha}G(t/x)^p+(x/t)^{-\delta}\big) \frac{\D x}{x^{1/\alpha}}\\ 
&\quad\lesssim t^{p/\alpha} G(t)^p \int_t^1(1+G_1(1/x))^p \frac{\D x}{x^{(p+1)/\alpha}} + t^{\delta}\int_t^1 \frac{\D x}{x^{\delta+1/\alpha}}+ \int_0^t \frac{\D x}{x^{1/\alpha}} \\
& \quad\lesssim t^{p/\alpha}G(t)^p(\1_{\{p< \alpha-1\}}+(1+G_1(1/t))^p t^{1-(p+1)/\alpha}\1_{\{p>\alpha-1\}}+\ov{G}_1(1/t)\1_{\{p=\alpha-1\}}) +t^{\delta}+t^{1-1/\alpha}\\
&\quad\lesssim t^{p/\alpha}G(t)^p(1+(1+G_1(1/t))^p t^{1-(p+1)/\alpha}+\ov{G}_1(1/t)\1_{\{p=\alpha-1\}}) +t^{\delta}+t^{1-1/\alpha}, 
\quad \text{as } t \da 0.
\end{align*}
Altogether, it follows that, as $t\da 0$,
\begin{equation*}
|\bm{a}_{\bm{X}_t^\co}^\theta-\bm{a}_{\bm{Z}}^\theta| 
\lesssim t/g(t)+G_2(t)
+G(t)^p\big(t^{p/\alpha}+t^{1-1/\alpha}[1+G_1(1/t)^p +\ov{G}_1(1/t)\1_{\{p=\alpha-1\}}]\big) +t^{\delta}.
\end{equation*}

In the finite variation case $\alpha \in (0,1)$, using that $\bm{X}^\co$ (and hence $\bm{X}^\co_t$) and $\bm{Z}$ have zero natural drift, the boundedness of $x\mapsto x^{p+1}\e^{-\theta x}$ by $(\e\theta/(p+1))^{-p-1}$ on $[1,\infty)$, and~\eqref{eq:bound_como_r_dist}, we obtain, as $t \da 0$, 
\begin{equation*}
\big|\bm{a}_{\bm{X}_t^\co}^\theta
-\bm{a}_{\bm{Z}}^\theta\big|
\le \int_{\Sp^{d-1}}\int_0^\infty
|\rho_{\bm{X}_t^\co}^{\la}(x,\bm{v})-\rho_{\bm{Z}}^{\la}(x)| \e^{-\1_{(0,1)}(x)\theta /x} \D x \, \sigma(\D \bm{v})
\lesssim G_2(t)+t^{p/\alpha}G(t)^p+t^\delta.
\end{equation*}
The special case where $\int_{\Sp^{d-1}}\bm{v}\,\sigma(\D\bm{v})=\int_{\Sp^{d-1}} \rho^{\la}_{\bm{X}_t^\co}(x,\bm{v})\bm{v}\,\sigma(\D\bm{v})=\bm{0}$ for $x>0$ again follows directly.
\end{proof}

\subsection{Proofs of the Main Results}
\label{subsec:proofs_main_results}

\begin{proof}[Proof of Theorem~\ref{thm:roc_tempered_sde}] 
The idea is to use Theorems~\ref{thm:gen_bound_thin} and~\ref{thm:gen_bound_como} in conjunction with the estimates from Propositions~\ref{prop:integrals_DoA_thin} and~\ref{prop:integrals_DoA_como}. Throughout this proof, $T,\theta>0$ are fixed and we use the notation introduced in the previous subsection. 
By construction, our processes have finite moments of any order, so Theorems~\ref{thm:gen_bound_thin} and~\ref{thm:gen_bound_como} will give, for $r\coloneqq\lceil\alpha\rceil$,
\[
\E_\theta\big[\|\X_t-\ZZ\|_{[0,T]}^r\big] 
\le 
\begin{cases}
\kappa_r^\mft(T)\e^{\eta_r(T)}
&\text{under (\nameref{asm:T})}\\
\kappa_r^\mfc(T)\e^{\eta_r(T)}
&\text{under (\nameref{asm:C}).}
\end{cases}
\]
In all cases, our assumptions,~\eqref{eq:kappa_eta2},~\eqref{eq:kappa_eta1} and Propositions~\ref{prop:integrals_DoA_thin} and~\ref{prop:integrals_DoA_como} give $\eta_r(T)\lesssim 1$ as $t\da 0$, so it suffices to analyse the functions $\kappa_r^\mft$ and $\kappa_r^\mfc$. Furthermore, we recall from Potter's bound~\cite[Thm~1.5.6]{MR1015093} that any slowly varying function dominates any positive power function of $t$ as $t\da 0$.

\underline{\textbf{DoA under~(\nameref{asm:T}) with $\alpha \in (1,2)$}}. The definition of $\kappa_2^\mft$ in~\eqref{eq:kappa_eta2} and Proposition~\ref{prop:integrals_DoA_thin} gives:
\[
\kappa_2^\mft(T)
\lesssim |\bm{a}_{\bm{X}_t^\co}^\theta-\bm{a}_{\bm{Z}}^\theta|^2 
+\!\int_{\R^d_{\bm0}}\!|\bm{w}|^2\bigg|\frac{\D\nu_{\bm{Z}}}{\D\nu}(\bm{w})
    -\frac{\D\nu_{\bm{X}_t}}{\D\nu}(\bm{w})\bigg| \nu^\theta(\D\bm{w}),\\ 
\quad\text{as }t\da 0.
\]

{\underline{\textit{DoNA}}}. Proposition~\ref{prop:integrals_DoA_thin}, with $H\equiv 1$, $G\equiv 1\equiv H_1$, $H_2\equiv 0$ and $g(t)=t^{1/\alpha}$, gives, as $t\da 0$,
\begin{equation*}
\big|\bm{a}_{\bm{X}^\co_t}^\theta-\bm{a}_{\bm{Z}}^\theta\big|
\lesssim t^{p/\alpha} 
+ t^{1-1/\alpha}\big(1
    +|\log t|\1_{\{p=\alpha-1\}}\big)
\quad\text{and}\quad
\int_{\R^d_{\bm 0}} |\bm{w}|^2\bigg|\frac{\D\nu_{\bm{Z}}}{\D\nu}(\bm{w})
    -\frac{\D\nu_{\bm{X}_t}}{\D\nu}(\bm{w})\bigg| \nu^\theta(\D\bm{w})
\lesssim t^{p/\alpha}.
\end{equation*} 
This concludes the general case $\alpha \in (1,2)$ in DoNA. If~\eqref{eq:symmetry_thin} holds, then Proposition~\ref{prop:integrals_DoA_thin} gives $\bm{a}_{\bm{X}^\co_t}^\theta=\bm{a}_{\bm{Z}}^\theta=\bm{0}$, so that $\E_\theta\big[\|\X_t-\ZZ\|_{[0,T]}^2\big] = \Oh(t^{p/\alpha})$ as $t \da 0$ in this case. 

{\underline{\textit{DoNNA}}}. In this case, 
\[
\big|\bm{a}_{\bm{X}^\co_t}^\theta-\bm{a}_{\bm{Z}}^\theta\big|
+\int_{\R^d_{\bm 0}} |\bm{w}|^2\bigg|\frac{\D\nu_{\bm{Z}}}{\D\nu}(\bm{w})
-\frac{\D\nu_{\bm{X}_t^\co}}{\D\nu}(\bm{w})\bigg| \nu^\theta(\D\bm{w})
\lesssim H_2(g(t)),
\quad\text{as }t\da 0,
\]
so $\E_\theta\big[\|\X_t-\ZZ\|_{[0,T]}^2\big] =\Oh(H_2(t^{1/\alpha}G(t)))$, as $t \da 0$.

\underline{\textbf{DoA under~(\nameref{asm:T}) with $\alpha \in (0,1)$}}. The definition of $\kappa_1^\mft$ in~\eqref{eq:kappa_eta1} and Proposition~\ref{prop:integrals_DoA_thin} give
\[
\kappa_1^\mft(T)
\lesssim \int_{\R^d_{\bm0}}|\bm{w}|\cdot
\bigg|\frac{\D\nu_{\bm{Z}}}{\D\nu}(\bm{w})
-\frac{\D\nu_{\bm{X}_t}}{\D\nu}(\bm{w})\bigg| \nu^\theta(\D\bm{w}),
\quad\text{as }t\da 0.
\]
Then Proposition~\ref{prop:integrals_DoA_thin} gives, as $t\da 0$, $\kappa_1^\mft(T)\lesssim t^{p/\alpha}$ in the {\underline{\textit{DoNA}}}, and $\kappa_1^\mft(T)\lesssim H_2(g(t))$ in the {\underline{\textit{DoNNA}}}.

\underline{\textbf{DoA under~(\nameref{asm:C}) with $\alpha\in(1,2)$}}. The definition of $\kappa_2^\mfc$ in~\eqref{eq:defn_kappa_2_cm} and Proposition~\ref{prop:integrals_DoA_como}, gives
\[
\kappa^\mfc_2(T) \lesssim  
\big|\bm{a}_{\bm{X}_t^\co}^\theta-\bm{a}_{\bm{Z}}^\theta\big|^2 
+\int_{\Sp^{d-1}}\int_0^\infty
(\rho_{\bm{X}_t^\co}^{\la}(x,\bm{v})-\rho_{\bm{Z}}^{\la}(x,\bm{v}))^2 e^{-\1_{(0,1)}(x) \theta/x } \D x \, \sigma(\D \bm{v}),
\enskip t\da 0.
\]

{\underline{\textit{DoNA}}}. By Proposition~\ref{prop:integrals_DoA_como} and Remark~\ref{rem:DoNA_como_delta_large}, with $H\equiv 1$, $G\equiv 1\equiv G_1$, $G_2\equiv 0$ and $g(t)=t^{1/\alpha}$, 
\begin{gather*}
\big|\bm{a}_{\bm{X}^\co_t}^\theta-\bm{a}_{\bm{Z}}^\theta\big|
\lesssim t^{p/\alpha}+t^{1-1/\alpha}\big(1 +|\log t|\1_{\{p=\alpha-1\}}\big)
\quad\text{and}\\
\int_{\Sp^{d-1}}\int_0^\infty
(\rho_{\bm{X}_t^\co}^{\la}(x,\bm{v})-\rho_{\bm{Z}}^{\la}(x,\bm{v}))^2 e^{-\1_{(0,1)}(x) \theta/x } \D x \, \sigma(\D \bm{v})
\lesssim t^{2p/\alpha},
\quad\text{as }t\da 0.
\end{gather*} 
Hence, $\kappa_2^\mfc(T) \lesssim t^{2p/\alpha}+
t^{2(1-1/\alpha)}(1+|\log t|^2\1_{\{p=\alpha-1\}})$, so the result follows in this case. The case where~\eqref{eq:symmetry_como} also holds again follows from Proposition~\ref{prop:integrals_DoA_como} and Remark~\ref{rem:DoNA_como_delta_large} directly.

{\underline{\textit{DoNNA}}}. In this case, Proposition~\ref{prop:integrals_DoA_como} directly gives
\[
\big|\bm{a}_{\bm{X}^\co_t}^\theta-\bm{a}_{\bm{Z}}^\theta\big|^2
+\int_{\Sp^{d-1}}\int_0^\infty
(\rho_{\bm{X}_t^\co}^{\la}(x,\bm{v})-\rho_{\bm{Z}}^{\la}(x,\bm{v}))^2 e^{-\1_{(0,1)}(x) \theta/x } \D x \, \sigma(\D \bm{v})
\lesssim G_2(t)^2,
\quad\text{as }t\da 0.
\]
Hence, $\E_\theta\big[\|\X_t-\ZZ\|_{[0,T]}^2\big] =\Oh(G_2(t)^2)$, as $t \da 0$.

\underline{\textbf{DoA under~(\nameref{asm:C}) with $\alpha\in(0,1)$}}. The definition of $\kappa_1^\mfc$ in~\eqref{eq:defn_kappa_1_cm} and Proposition~\ref{prop:integrals_DoA_como} give
\[
\kappa_1^\mfc(T)
\lesssim \int_{\Sp^{d-1}} \int_0^\infty|\rho_{\bm{X}_t^\co}^{\la}(x,\bm{v})-\rho_{\bm{Z}}^{\la}(x,\bm{v})|e^{-\1_{(0,1)}(x) \theta/x }  \D x \,\sigma(\D \bm{v}),
\quad\text{as }t\da 0.
\] 
Thus, Proposition~\ref{prop:integrals_DoA_como} gives, as $t\da 0$, $\kappa_1^\mfc(T)\lesssim t^{p/\alpha}$ in the \underline{\textit{DoNA}} and $\kappa_1^\mfc(T)\lesssim G_2(t)$ in the \underline{\textit{DoNNA}}.
\end{proof}

Before presenting the proofs of Theorems~\ref{thm:main_res_T} \&~\ref{thm:main_res_C}, we introduce an auxiliary lemma.

\begin{lemma}
\label{lem:Lp-Girsanov-to-cip}
Let $\{\xi(t)\}_{t\in T}$ be a family of non-negative random variables and $M$ be an a.s. positive random variable. If $\sup_{t\in(0,1]}\E[\xi(t)^rM]<\infty$ for some $r>0$, then $\{\xi(t)\}_{t\in T}$ is tight.
\end{lemma}

\begin{proof}
We aim to show that $\lim_{u\to\infty}\sup_{t \in (0,1]}\p(|\xi(t)|>u)=0$. Note that, for all $\ve>0$,
\[
\p(|\xi(t)|>u)
\le\p(|\xi(t)|>u, M>\ve) + \p(M\le \ve)
\le \p(|\xi(t)|^rM>u^r\ve) + \p(M\le \ve).
\]
Hence, by assumption and Markov's inequality, it follows that
\[
\limsup_{u\to\infty}\sup_{t\in T}\p(|\xi(t)|>u)
\le
 \limsup_{u\to\infty}\frac{\sup_{t\in T}\E[|\xi(t)|^rM]}{u^r\ve}
+\p(M\le \ve)
=\p(M\le \ve).
\]
Since $M>0$ a.s. and $\ve>0$ is arbitrary, taking $\ve\da 0$ gives $\p(M\le \ve) \to 0$ and concludes the proof.
\end{proof}

\begin{proof}[Proof of Theorem~\ref{thm:main_res_T}]
First, recall from the construction of $\nu^\theta$, and hence $\E_\theta$, that
\begin{equation}
\label{eq:change-of-measure-thin}
\E_\theta\big[\|\X_t-\ZZ\|_{[0,T]}^p\big] 
= \E\big[\|\X_t-\ZZ\|_{[0,T]}^p M_\theta \big].
\end{equation} 
Here, $M_\theta$ is given as in~\eqref{eq:P_theta^t}, i.e.
\begin{equation}\label{eq:M_theta_thinning}   
M_\theta 
\coloneqq \frac{\D\p_\theta}{\D\p}
=\exp\bigg(\vartheta T-\theta\int_{(0,T]\times D^\co}|\bm{w}|\Xi^\co(\D s,\D\bm{w})\bigg), 
\text{ where }
\vartheta
= \int_{D^\co}\big(1-\e^{-\theta|\bm{w}|}\big)\nu(\D\bm{w})>0,
\end{equation} 
with $\Xi^\co$ and $\nu$ defined in the context of the thinning coupling from Section~\ref{sec:thinning_technic_bound} under Assumption~(\nameref{asm:T}). Note that $M_\theta>0$ a.s. and crucially does not depend on $t$. The proof now follows from~\eqref{eq:change-of-measure-thin} and Lemma~\ref{lem:Lp-Girsanov-to-cip} with $M\coloneqq M_\theta$. Indeed, it suffices to specify $r$ and the family $\{\xi(t)\}_{t\in(0,1]}$ satisfying the assumptions of Lemma~\ref{lem:Lp-Girsanov-to-cip}.

Suppose $\bm{X}$ is in the domain of normal attraction of $\bm{Z}$. Then, by Theorem~\ref{thm:roc_tempered_sde}, the family $\xi(t)\coloneqq t^{-q}\|\X_t-\ZZ\|_{[0,T]}$ satisfies the assumptions of Lemma~\ref{lem:Lp-Girsanov-to-cip} with $r\coloneqq\lceil\alpha\rceil$ and 
\[
q
\coloneqq 
\begin{cases}
\frac{p}{2\alpha}\wedge (1-\frac{1}{\alpha}),&\alpha\in(1,2)\text{ and \eqref{eq:symmetry_thin} fails},\\
\frac{p}{2\alpha},&\alpha\in(1,2)\text{ and \eqref{eq:symmetry_thin} holds},\\
\frac{p}{\alpha},&\alpha\in(0,1).
\end{cases}
\]
Suppose $\bm{X}$ is in the domain of non-normal attraction of $\bm{Z}$. Then, by Theorem~\ref{thm:roc_tempered_sde}, the family $\xi(t) \coloneqq H_2(g(t))^{-1/p}\|\X_t-\ZZ\|_{[0,T]}$ satisfies the assumptions of Lemma~\ref{lem:Lp-Girsanov-to-cip} with $r=\lceil\alpha\rceil$.
\end{proof}

\begin{proof}[Proof of Theorem~\ref{thm:main_res_C}]
We proceed as in the proof of Theorem~\ref{thm:main_res_T}. By construction of $\E_\theta$, we have
$\E_\theta[\|\X_t-\ZZ\|_{[0,T]}^p] = \E[\|\X_t-\ZZ\|_{[0,T]}^p M_\theta]$, where $M_\theta$ is as in~\eqref{eq:P_theta^t_cm}: 
\begin{equation}\label{eq:M_theta_como}
M_\theta \coloneqq \frac{\D\p_\theta}{\D\p}=\exp\bigg(\vartheta T-\theta\int_{[0,T]\times (0,1)\times \Sp^{d-1}} x^{-1} \Xi^\co(\D t, \D x,\D\bm{v})\bigg), 
\quad
\vartheta = \frac{1}{\alpha}\int_1^\infty \big(1-\e^{-\theta x} \big)  \frac{\D x}{x^2},
\end{equation} 
with $\Xi^\co$ defined in the context of the comonotonic coupling from Section~\ref{sec:comono_technic_bound} under Assumption~(\nameref{asm:C}). Again, the claims will follow from Lemma~\ref{lem:Lp-Girsanov-to-cip} with $M\coloneqq M_\theta$.

Suppose $\bm{X}$ is in the domain of normal attraction of $\bm{Z}$. Then, by Theorem~\ref{thm:roc_tempered_sde}, the family 
\[
\xi(t)\coloneqq f(t)^{-1} 
\|\X_t-\ZZ\|_{[0,T]}
\]
satisfies the assumptions of Lemma~\ref{lem:Lp-Girsanov-to-cip} with $r\coloneqq\lceil\alpha\rceil$ and 
\[
f(t)
\coloneqq 
\begin{cases}
t^{
(p/\alpha)\wedge(1-1/\alpha)}(1+|\log t|\1_{\{p=\alpha-1\}
}),
&\alpha\in(1,2)\text{ and \eqref{eq:symmetry_como} fails},\\
t^{
p/\alpha },
&\alpha\in(1,2)\text{ and \eqref{eq:symmetry_como} holds or }\alpha\in(0,1).
\end{cases}
\]
Suppose $\bm{X}$ is in the domain of non-normal attraction of $\bm{Z}$. Then, by Theorem~\ref{thm:roc_tempered_sde}, the family $\xi(t)\coloneqq G_2(t)^{-1}\|\X_t-\ZZ\|_{[0,T]}$ satisfies the assumptions of Lemma~\ref{lem:Lp-Girsanov-to-cip} with $r=\lceil\alpha\rceil$.
\end{proof}

\section{Future objectives and conjectures}
\label{sec:conclusion}

Let us comment on possible extensions and generalisations left for future work.

\begin{itemize}
\item In the present work, we excluded the case $\alpha=2$, wherein processes are attracted to a Brownian motion. Although results for such processes could be derived via Theorems~\ref{thm:gen_bound_thin} and~\ref{thm:gen_bound_como}, our goal was to highlight the thinning and comonotonic couplings, which are not used in this limiting case. Additionally, to obtain the best guarantees, it would be best to first find a sharp coupling (say, with matching rates of convergence in $L^1$ or $L^2$-Wasserstein distance as in~\cite{WassersteinPaper}) between pure-jump L\'evy processes and Brownian motion with converging distance in the small-time regime, which we currently do not have (see, e.g.,~\cite[Sec.~2.4]{WassersteinPaper}). 
\item In the present work, we excluded the case $\alpha=1$ because many cases would arise depending on the attracted process' characteristics, which would considerably increase the length of the present work. However, the same couplings and techniques used here could be employed to analyse this case.
\item As discussed in Subsection~\ref{subsec:literature}, it would be desirable to obtain bounds on the Wasserstein distance between SDE solutions under a bounded metric that yield sharp convergence rates in the small-time DoA. However, such extensions likely require BDGN-type inequalities tailored to bounded metrics.
\end{itemize}

\printbibliography

\section*{Acknowledgements}
\thanks{
\noindent DKB is supported by AUFF NOVA grant AUFF-E-2022-9-39. JGC is supported by PAPIIT grant 36-IA104425 and EPSRC grant EP/V009478/1. The authors would like to thank the Isaac Newton Institute for Mathematical Sciences, Cambridge, for support during the INI satellite programme \textit{Heavy tails in machine learning}, hosted by The Alan Turing Institute, London, and the INI programme \textit{Stochastic systems for anomalous diffusion} hosted at INI in Cambridge, where work on this paper was undertaken. This work was supported by EPSRC grant EP/R014604/1. This work was also supported by a short visit sponsored by CIC grant COIC/STIA/10133/2025.}

\appendix
\section{Stable Processes and their Domain of Attraction}

In this section, we will recall some general theory from~\cite{WassersteinPaper}. We start by recalling the formal definition of the $\alpha$-stable attractor $\bm{Z}$. 
\begin{defin}[{\cite[Defn~4.1]{WassersteinPaper}}]\label{def:alpha_stable}
For  any $\alpha\in(0,2]$, the law of 
an $\alpha$-stable 
 L\'evy process $\bm{Z}$
 is given by a generating triplet $(\bm{\gamma_Z},\bm{\Sigma_Z}\bm{\Sigma_Z}^\tra,\nu_{\bm{Z}})$
(for the cutoff function $\bm{w}\mapsto\1_{B_{\bm{0}}(1)}(\bm{w})$)
as follows:  the L\'evy measure equals
\begin{equation}
    \label{eq:Levy-measure-stable}
    \nu_{\bm{Z}}(A)
    \coloneqq c_\alpha\int_0^\infty\int_{\Sp^{d-1}}
        \1_{A}(r\bm{v})\sigma(\D\bm{v})r^{-\alpha-1}\D r,
    \quad A\in\mathcal{B}(\R^d_{\bm{0}}),
\end{equation}
where $\sigma$ is a probability measure  on $\mathcal{B}(\Sp^{d-1})$  
and  $c_\alpha\in[0,\infty)$ an ``intensity'' parameter, satisfying
\begin{itemize}[leftmargin=1em, nosep]
    \item $\alpha=2$ [Brownian motion with zero drift]: $\bm{\Sigma_Z}\ne \bm{0}$, $\bm{\gamma_Z}=\bm{0}$ and $c_\alpha=0$ (i.e. $\nu_{\bm{Z}}\equiv 0$);
    \item $\alpha\in(1,2)$ [infinite variation,  zero-mean process]: $c_\alpha>0$, $\bm{\gamma_Z}=-\int_{\R^d_{\bm0}\setminus D}\bm{x}\nu_{\bm{Z}}(\D\bm{x})$ and $\bm{\Sigma_Z}=\bm{0}$;
    \item $\alpha=1$ [Cauchy process]: either $c_\alpha>0$, with symmetric angular component $\int_{\Sp^{d-1}}\bm{v}\sigma(\D\bm{v})=\bm{0}$, or $c_\alpha=0$ and the process $\bm{Z}$ is  a deterministic nonzero linear drift, i.e. $\bm{Z}_t=\bm{\gamma_Z}t$ for all times $t$;
    \item $\alpha\in(0,1)$ [finite variation and zero natural drift]:  $c_\alpha>0$ and  
    $\bm{\gamma_Z}=\int_{D}\bm{x}\nu_{\bm{Z}}(\D\bm{x})$.
\end{itemize}
\end{defin}

From Definition~\ref{def:alpha_stable}, an $\alpha$-stable process $\bm{Z}$ satisfies $(\bm{Z}(st))_{s\in[0,1]}\eqd (t^{1/\alpha}\bm{Z}(s))_{s\in[0,1]}$ for $t>0$, and, for $\alpha\in[1,2)$ (resp. $\alpha\in(0,1)$), a non-deterministic $\alpha$-stable process $\bm{Z}$ is of infinite (resp. finite) variation by~\cite[Thm~21.9]{MR3185174}, since~\eqref{eq:Levy-measure-stable} implies  $\int_{B_{\bm{0}}(1)\setminus\{\bm0\}}|\bm{x}|\nu_{\bm{Z}}(\D\bm{x})=\infty$ (resp. $\int_{B_{\bm{0}}(1)\setminus\{\bm0\}}|\bm{x}|\nu_{\bm{Z}}(\D\bm{x})<\infty$). 
For any $\bm{a}\in \Sp^{d-1}$, define $\scrL_{\bm{a}}(r)\coloneqq\{\bm{x}\in\R^d:\langle\bm{a},\bm{x}\rangle\ge r\}$ for any $r>0$. The following known result characterises the L\'evy processes in the DoA of an $\alpha$-stable process. The result is a consequence of~\cite[Thm~15.14]{MR1876169} and~\cite[Thm~2]{MR3784492}, see~\cite[App.~B]{WassersteinPaper} for the proof.

\begin{theorem}[{Small-time DoA~\cite[Thm~4.2]{WassersteinPaper}}]\label{thm:small_time_domain_stable}
Let $\bm{X}=(\bm{X}(t))_{t \in [0,1]}$ and $\bm{Z}=(\bm{Z}(t))_{t\in [0,1]}$ be L\'evy processes in $\R^d$. Then $(\bm{X}(st)/g(t))_{s \in [0,1]}\cid (\bm{Z}(s))_{s\in [0,1]}$ as $t \da 0$ in the Skorokhod space for some positive normalising function $g:(0,1]\to(0,\infty)$ if and only if $\bm{Z}$ is $\alpha$-stable for some $\alpha\in(0,2]$, the normalising function admits the representation $g(t)=t^{1/\alpha}G(t)$, where $G$ is a slowly varying function at zero, and the generating triplets $(\bm\gamma_{\bm{X}},\bm{\Sigma_X}\bm{\Sigma}_{\bm{X}}^\tra,\nu_{\bm{X}})$ and $(\bm\gamma_{\bm{Z}},\bm{\Sigma_Z}\bm{\Sigma}_{\bm{Z}}^\tra,\nu_{\bm{Z}})$ (for the cutoff function $\bm{w} \mapsto \1_{D}(\bm{w})$) of $\bm{X}$ and $\bm{Z}$, respectively, are related as follows:
\begin{itemize}[leftmargin=1em, nosep]
\item if $\alpha=2$ (attraction to Brownian motion), then 
\begin{equation}
\label{eq:Brownian-limit}
G(t)^{-2} \bigg(\bm{\Sigma_X}\bm{\Sigma}_{\bm{X}}^\tra
    +\int_{g(t)D}\bm{x}\bm{x}^\tra\nu_{\bm{X}}(\D\bm{x})\bigg)
\to
\bm{\Sigma_Z}\bm{\Sigma}_{\bm{Z}}^\tra,\quad\text{as }t\da 0;
\end{equation}
\item if $\alpha\in(1,2)$, we have $\bm{\Sigma_X}=\bm{0}$ and
\begin{equation}
\label{eq:jump-stable-limit}
t\nu_{\bm{X}}(\scrL_{\bm{v}}(g(t)))\to \nu_{\bm{Z}}(\scrL_{\bm{v}}(1)),
\quad\text{as }t\da 0,\quad\text{for any }\bm{v}\in\Sp^{d-1};
\end{equation}
    \item if $\alpha=1$ (attraction to Cauchy process), then~\eqref{eq:jump-stable-limit} holds,
\begin{equation}
\label{eq:Cauchy-limit}
G(t)^{-1}\bigg(\bm{\gamma_X} 
- \int_{D\setminus g(t)D}\bm{x}\nu(\D\bm{x})\bigg)\to \bm{\gamma_Z},
\quad\text{as }t\da 0,
\end{equation}
and, for any $\bm{v}\in\Sp^{d-1}$, such that $\langle\bm{v},\bm{X}\rangle$ has finite variation (i.e. $\int_{D}|\langle\bm{v},\bm{x}\rangle|\nu_{\bm{X}}(\D\bm{x})<\infty$) and $\nu_{\bm{Z}}(\scrL_{\bm{v}}(1))>0$, the process $\langle\bm{v},\bm{X}\rangle$ has zero natural drift: $\langle\bm{v},\bm{\gamma_X}\rangle=\int_{D}\langle\bm{v},\bm{x}\rangle\nu_{\bm{X}}(\D\bm{x})$.
\item if $\alpha\in(0,1)$, then~\eqref{eq:jump-stable-limit} holds, $\bm{X}$ has finite variation (i.e. $\int_{D}|\bm{x}|\nu_{\bm{X}}(\D\bm{x})<\infty$) and zero natural drift (i.e. $\bm{\gamma_X}=\int_{D}\bm{x}\nu_{\bm{X}}(\D\bm{x})$).
\end{itemize}
Moreover,  the function $g$ 
satisfying the weak limit above
is asymptotically unique at $0$: a positive  function $\wt g$ 
satisfies $(\bm{X}(st)/\wt g(t))_{s \in [0,1]}\cid (\bm{Z}(s))_{s\in [0,1]}$ as $t \da 0$
if and only if $\wt g(t)/g(t)\to 1$ as $t\da 0$.
\end{theorem}

Note that in the case $\alpha=2$ in Theorem~\ref{thm:small_time_domain_stable},
we may have $\bm{\Sigma_X}=\bm{0}$ (see~\cite[Ex.~5.7]{WassersteinPaper}), but in this case $G$ cannot be asymptotically equal to a positive constant. If $\alpha\in(1,2)$, the process $\bm{X}$ does not require centring since its mean is linear in time and thus disappears in the scaling limit. Finally, when the attractor is of finite variation (i.e. when $\alpha\in(0,1)$), the process $\bm{X}$ must have zero natural drift for the scaling limit to exist.

\end{document}